\documentclass[10pt]{amsart}
\usepackage{latexsym, amsmath,amssymb}

\usepackage{enumerate}

\renewcommand{\epsilon}{\varepsilon}

\newtheorem{theorem}{Theorem}[section]

\newtheorem{lemma}[theorem]{Lemma}

\newtheorem{corr}[theorem]{Corollary}

\newtheorem{proposition}[theorem]{Proposition}
\newtheorem{deff}[theorem]{Definition}
\newtheorem{remark}[theorem]{Remark}
\newcommand{\bth}{\begin{theorem}}
	\newcommand{\ble}{\begin{lemma}}
		\newcommand{\bcor}{\begin{corr}}

			\newcommand{\bdeff}{\begin{deff}}

				\newcommand{\bprop}{\begin{proposition}}
					\newcommand{\ele}{\end{lemma}}
				\newcommand{\ecor}{\end{corr}}
			\newcommand{\edeff}{\end{deff}}
		
		\newcommand{\eprop}{\end{proposition}}

	\newcommand{\Rn}{{\mathbb R}^n}
	
	\newcommand{\la}{\lambda}
	
	\newcommand{\eps}{\varepsilon}
	\newcommand{\e}{\varepsilon}

	\newcommand{\supp}{\text{supp }}
	\renewcommand{\Pi}{\varPi}

	\renewcommand{\epsilon}{\varepsilon}

	\newcommand{\R}{{\mathbb R}}
	\newcommand{\ls}{\lesssim}
	\newcommand{\gs}{\gtrsim}
	\newcommand{\1}{{\rm 1\hspace*{-0.4ex}%
			\rule{0.1ex}{1.52ex}\hspace*{0.2ex}}}
	
	\newcommand{\ola}{\1_\la}

	\newcommand{\Log}{{\rm Log}}

	\numberwithin{equation}{section}
	
\begin{document}
		\title[weyl law]{Sharp Pointwise Weyl Laws for Schr\"odinger operators with singular potentials on flat tori}
		\author[]{Xiaoqi Huang}
		\address[X.H.]{Department of Mathematics, University of Maryland, College Park, MD, 20742}
		\email{xhuang49@umd.edu}

		\author[]{Cheng Zhang}
		\address[C.Z.]{Mathematical Sciences Center, Tsinghua University, Beijing 100084, China }
		\email{czhang98@tsinghua.edu.cn}
		\begin{abstract}
			The Weyl law of the Laplacian  on the flat torus $\mathbb{T}^n$ is concerning the number of eigenvalues $\le\la^2$, which is equivalent to counting the lattice points inside the ball of radius $\la$ in $\mathbb{R}^n$. The leading term in the Weyl law is $c_n\la^n$, while the sharp error term $O(\la^{n-2})$ is only known in dimension $n\ge5$. Determining the sharp error term in lower dimensions is  a famous open problem (e.g. Gauss circle problem). In this paper, we show that under a type of singular perturbations one can obtain the pointwise Weyl law with a sharp error term in any dimensions. This result establishes the sharpness of the general theorems for the Schr\"odinger operators $H_V=-\Delta_{g}+V$ in the previous work \cite{hz} of the authors, and extends the 3-dimensional results of Frank-Sabin \cite{fs} to any dimensions by using a different approach. Our approach is a combination of Fourier analysis techniques on the flat torus, Li-Yau's  heat kernel estimates, Blair-Sire-Sogge's eigenfunction estimates, and Duhamel's principle for the wave equation.
		\end{abstract}
		\maketitle
	Let $n\ge2$ and $(M,g)$ be a $n$-dimensional compact Riemannian manifold. Let $\Delta_g$ be the Laplace-Beltrami operator on $M$. It is known that the spectrum of $-\Delta_g$ is discrete and nonnegative (see e.g. \cite{SoggeHangzhou}):
	\[0=\la_0^2<\la_1^2\le\la_2^2\le\la_3^2\le\cdot\cdot\cdot.\]
The associated $L^2$-normalized
eigenfunctions are denoted by  $\{e_j\}_{j=1}^\infty$, respectively so
that
\[	-\Delta_ge_j=\lambda^2_j e_j, \quad \text{and }\, \, 
	\int_M |e_j(x)|^2 \, dx=1.\]
Here $\{e_j\}_{j=1}^\infty$ is an orthonormal basis for $L^2(M)$.  If $N(\la)$ denotes the 
Weyl counting function for $-\Delta_g$, then one has the  integrated Weyl law
\begin{equation}\label{weyllaw}
	N(\la):=\# \{j: \, \la_j\le \la\}=(2\pi)^{-n} \omega_n \text{Vol}_g(M) \, \la^n
	\, +\, O(\la^{n-1}), 
\end{equation}
where $\omega_n=\pi^{\frac n2}/\Gamma(\frac n2+1)$ denotes the volume of the unit ball in
$\Rn$ and $\text{Vol}_g(M)$ denotes the Riemannian volume of $M$.  
This result is due to Avakumovi\'{c}~\cite{Avakumovic} and Levitan~\cite{Levitan}, and it was generalized to general self-adjoint elliptic pseudo-differential operators
by H\"ormander~\cite{HSpec}.  The error term $O(\la^{n-1})$ in \eqref{weyllaw} 
cannot be improved on the standard round sphere $\mathbb{S}^n$. Moreover, the error term can be improved under certain geometric conditions, see e.g. Duistermaat-Guillemin \cite{dg}, B\'erard \cite{berard}, Volovoy \cite{vol}, Iosevich-Wyman \cite{iw}, Canzani-Galkowski \cite{cg}. Note that $N(\la)=\int_M\sum_{\la_j\le\la}|e_j(x)|^2dx$, thus $N(\la)$ can be studied by understanding the kernel of the spectral projection operator $\sum_{\la_j\le\la}e_j(x)\overline{e_j(y)}$. Indeed, the pointwise Weyl law holds 
\begin{equation}\label{ptweyl}
	\sum_{\la_j\le\la}|e_j(x)|^2=(2\pi)^{-n} \omega_n  \, \la^n
	\, +\, O(\la^{n-1}), \ \text{uniformly}\ \text{in}\ x.
\end{equation}
See e.g. Avakumovi\'{c}~\cite{Avakumovic}, Levitan~\cite{Levitan}, H\"ormander~\cite{HSpec}. The error term $O(\la^{n-1})$ 
cannot be improved on the standard round sphere $\mathbb{S}^n$, but it can still be improved under certain geometric conditions.  See also \cite{bhss}, \cite{hs2}, \cite{hs3}, \cite{hsz} for recent related works.

When $M$ is the flat torus $\mathbb{T}^n=\mathbb{R}^n/2\pi\mathbb{Z}^n$, there is a standard orthonormal basis on $L^2(\mathbb{T}^n)$: $\{e^{ij\cdot x}\}$, $j\in\mathbb{Z}^n$, and $e^{ij\cdot x}$ is an eigenfuntion of $-\Delta_g$ with eigenvalue $|j|^2$. Since $|e^{ij\cdot x}|\equiv1$, the integrated Weyl law of $-\Delta_g$ on the flat torus is equivalent to the pointwise version. There has been a lot of research related to the Weyl law on the  torus, which is equivalent to counting the lattice points inside the ball of radius $\la$:
\[	N(\la)=\#\{j\in \mathbb{Z}^n:|j|\le \la\}=(2\pi)^{-n}\omega_n \mathrm{Vol}_g(M) \, \la^n
+r_n(\la), \]
where the error term
\begin{equation}\label{best}r_n(\la)\lesssim\begin{cases}
	\la^{n-2} , \qquad\qquad\,\,\,\,\, \text{if} \,\,\, n\geq 5 \\
	\la^{2}(\log \la)^{2/3} ,\quad \,\,\,\,\,\, \text{if} \,\,\, n=4 \\
	\la^{\frac{21}{16}+\e} , \quad\qquad\quad\,\,\,\, \text{if} \,\,\, n=3 \\
	\la^{\frac{131}{208}}(\log \la)^{\frac{18627}{8320}}  , \,\,\, \text{if} \,\,\, n=2.\\
\end{cases}
\end{equation}
Currently, the exact order of the error term is only known when $n\geq 5$.  See e.g. Hlawka \cite{hla}, Landau \cite{landau}, Walfisz \cite{walf}, and Kr\"atzel \cite{krat}. The above best known results in lower dimensions are due to Walfisz \cite{walf} ($n$=4), Health-Brown \cite{health} ($n$=3), and Huxley \cite{hux} ($n$=2). In dimension 2, it is well known as the Gauss circle problem, and it is still open. For  more details and a discussion of recent progress 
on the problem, see e.g. the survey paper \cite{ivic} and Freeden \cite{free}. See also \cite{br}, \cite{BR2015}, \cite{hztorus} for recent lated works.

Consider	the Schr\"odinger operators $H_V=-\Delta_g+V$ on $M$. We shall assume throughout that the potentials $V$ are real-valued and $V\in \mathcal{K}(M)$, which is the Kato class. Recall that $\mathcal{K}(M)$ is all $V$ satisfying  
	\[\lim_{\delta\to 0}\sup_{x\in M}\int_{d_g(y,x)<\delta}|V(y)|W_n(d_g(x,y))dy=0,\]
	where \[W_n(r)=\begin{cases}r^{2-n},\quad\quad\quad\quad n\ge3\\
		\log(2+r^{-1}),\ \ n=2\end{cases}\]
	and $d_g$, $dy$ denote geodesic distance, the volume element on $(M,g)$.
	Note that by H\"older inequality, we have $L^{p}\subset \mathcal{K}(M)\subset L^1(M)$ for all $p>\frac n2$. The Kato class $\mathcal{K}(M)$ and $L^{n/2}(M)$ share the same critical scaling
	behavior, while neither one is contained in the other one for $n\ge3$. For instance, singularities of the type $|x|^{-\alpha}$ for $\alpha<2$ are allowed for both classes. These singular  potentials  appear naturally in physics, most notably the Coulomb potential $|x|^{-1}$ in three dimensions. See e.g. Simon \cite{SimonSurvey} for a detailed introduction to the Schr\"odinger operators with potentials in the Kato class and their physical  motivations.

The assumption that $V$ is in the Kato class is the ``minimal condition'' to ensure that $H_V$ is essentially self-adjoint and bounded from below, and  eigenfunctions of $H_V$ are bounded, see e.g. Blair-Sire-Sogge \cite{BSS}, Simon \cite{SimonSurvey}.  Since $M$ is compact, the spectrum of $H_V$ is discrete.  Also, the associated
	eigenfunctions are continuous (see \cite{SimonSurvey}).  Assuming, as we may, that
	$H_V$ is a positive operator, we shall write the spectrum
	of $\sqrt{H_V}$  as
	\begin{equation}\label{1.5}
		\{\tau_k\}_{k=1}^\infty,
	\end{equation}
	where the eigenvalues, $\tau_1\le \tau_2\le \cdots$, are arranged in increasing order and we account for multiplicity.  For each $\tau_k$ there is an 
	eigenfunction $e_{\tau_k}\in \text{Dom }(H_V)$ (the domain of $H_V$) so that
	\begin{equation}\label{1.6}
		H_Ve_{\tau_k}=\tau^2_ke_{\tau_k},\ \ {\rm and}\  \ \int_M |e_{\tau_k}(x)|^2 \, dx=1.
	\end{equation}

	After possibly adding a constant to $V$ we may, and shall, assume throughout that $H_V$ is bounded below by one, i.e.,
	\begin{equation}\label{1.7}
		\|f\|_2^2\le \langle \, H_Vf, \, f\, \rangle, \quad
		f\in \text{Dom }(H_V).
	\end{equation}
	Moreover, we shall let
	\begin{equation}\label{1.8}
		H^0=-\Delta_g
	\end{equation}
	be the unperturbed operator.  The corresponding eigenvalues and associated $L^2$-normalized
	eigenfunctions are denoted by $\{\lambda_j\}_{j=1}^\infty$ and $\{e^0_j\}_{j=1}^\infty$, respectively so
	that
	\begin{equation}\label{1.9}
		H^0e^0_j=\lambda^2_j e^0_j, \quad \text{and }\, \, 
		\int_M |e^0_j(x)|^2 \, dx=1.
	\end{equation}
	Both $\{e_{\tau_k}\}_{k=1}^\infty$ and $\{e^0_j\}_{j=1}^\infty$ are orthonormal bases for $L^2(M)$. Let $P^0=\sqrt{H^0}$ and $P_V=\sqrt{H_V}$. We denote the indicator function of the interval $[-\la,\la]$ by $\ola(\tau)$, and for simplicity we write the kernel of the spectral projection on the diagonal $x=y$ as
	\begin{equation}\label{def}\ola(P^0)(x,x)=\sum_{\la_j\le\la}|e_j^0(x)|^2,\ \ \ola(P_V)(x,x)=\sum_{\tau_k\le\la}|e_{\tau_k}(x)|^2.\end{equation}  One has the integrated Weyl law
	\begin{equation}\label{1.10}
		N^0(\la):=\# \{j: \, \la_j\le \la\}=(2\pi)^{-n} \omega_n \text{Vol}_g(M) \, \la^n
		\, +\, O(\la^{n-1}), 
	\end{equation}
	and the pointwise Weyl law 
	\begin{equation}\label{local}
\ola(P^0)(x,x)=(2\pi)^{-n}\omega_n\la^n+O(\la^{n-1}), \ \text{uniformly}\ \text{in}\ x.
	\end{equation}
 The integrated Weyl law \eqref{1.10} is due to Avakumovi\'{c}~\cite{Avakumovic} and Levitan~\cite{Levitan}, and it was generalized to general self-adjoint elliptic pseudo-differential operators
	by H\"ormander~\cite{HSpec}.  The sharpness of \eqref{1.10} means that it cannot be improved for the standard sphere. The original Weyl law  was proved by Weyl \cite{weyl} for a compact domain in $\mathbb{R}^n$ over a hundred years ago. See Arendt, Nittka, Peter and Steiner \cite{anps} for historical background on this famous problem and its solution by Weyl. The pointwise Weyl law \eqref{local} is due to Avakumovi\'{c}~\cite{Avakumovic}, following earlier partial results of Levitan~\cite{Levitan}, \cite{Levitan2}. The error term $O(\la^{n-1})$ is also sharp on the standard sphere. Proofs are presented in several texts, including H\"ormander~\cite{hbook} and Sogge \cite{fio}, \cite{SoggeHangzhou}. The pointwise Weyl law for a  compact domain in $\mathbb{R}^n$ is due to Carleman \cite{carl}. Similar results for compact manifolds with boundary are due to Seeley \cite{seeley1}, \cite{seeley2}.
	
	Recently, Huang-Sogge \cite{hs} proved that if $V\in\mathcal{K}(M)$, then the Weyl law of the same form still holds for the Schr\"odinger operators $H_V$, i.e.
	\begin{equation}\label{1.12}
		N_V(\la):=\#\{k: \, \tau_k\le \la\}=(2\pi)^{-n}\omega_n \mathrm{Vol}_g(M) \, \la^n
		+O(\la^{n-1}).
	\end{equation}
	See also \cite{sob}, \cite{fs}.  In the recent work \cite{hz}, we proved the following pointwise Weyl laws for $H_V$ on general Riemannian manifolds. 
		\begin{theorem}\label{thm0}
		Let $n\ge2$ and 
		\begin{equation}\label{rla}R(\la,x)=\ola(P_V)(x,x)-(2\pi)^{-n}\omega_n\la^n,\end{equation}where $\1_\la(P_V)(x,x)$ is defined in \eqref{def}. If  $V\in \mathcal{K}(M)$, then \[\sup_{x\in M}|R(\la,x)|=o(\la^n).\] Moreover, if $V\in L^n(M)$, then \[\sup_{x\in M}|R(\la,x)|=O(\la^{n-1}).\]
	\end{theorem}
When $V$ is smooth, the pointwise Weyl laws were proved by H\"ormander \cite{HSpec}, where $H_V$ is a self-adjoint elliptic pseudo-differential operator. See also \cite[Chapter 4]{fio} for the proof.  When $n=3$, these results were also obtained by Frank-Sabin \cite{fs}, while  the condition on $V$  for the error term $O(\la^2)$ in their paper is slightly different our $L^3$ condition.

	In this paper, we mainly prove the pointwise Weyl law of $H_V$ with a sharp error term in any dimensions,  under a type of singular perturbations $V$ on flat tori.
	\begin{theorem}\label{thm}
		Let $n\ge2$ and $M=\mathbb{T}^n=\mathbb{R}^n/2\pi\mathbb{Z}^n$. Fix $x_0\in M$. Let $R(\la,x)$ be defined in \eqref{rla}. Let $0<\eta<1$ and $V(x)=\rho(d_g(x,x_0))d_g(x,x_0)^{-2+\eta}$, where $\rho\in C_0^\infty ((-\pi,\pi))$ and $\rho(0)\ne0$. Then we have
		 \begin{equation}\label{rbdx0}|R(\la,x_0)|\gs  \la^{n-\eta}\end{equation}
		 and 
		\begin{equation}\label{rbd}\sup_{x\in M}|R(\la,x)|\approx \la^{n-\eta}.\end{equation}
	\end{theorem}
\begin{remark}
We remark that \eqref{rbdx0} and \eqref{rbd} can hold for more general non-radial potentials. Indeed,  they still hold for 
\[V(x)=\rho(d_g(x,x_0))d_g(x,x_0)^{-2+\eta}+V_0(x),\]
where $V_0(x)$ is a real-valued potential with ``lower-order singularities''. For instance, $V_0$ can be any functions in $L^{p}(\mathbb{T}^n)$ with $p>\frac{n}{2-\eta}$. It can be proved by slightly modifying the proof of Theorem \ref{thm}.
\end{remark}

These results are interesting in its own right, since the sharp error terms in the Weyl laws of the Laplacian on flat tori have not been completely understood so far, and they are closely related to many famous open problems including the Gauss circle problem and Waring's problem. Understanding the behavior of spectral properties when the operators undergo a ``small change'' is the core issue in perturbation theory. See Kato's book \cite{katobook} and Simon's survey paper \cite{SimonSurvey} for more precise and comprehensive discussions.

  These results  show that Theorem \ref{thm0} is sharp in the following sense. It is straightforward to check that  $V(x)$ in Theorem \ref{thm} is in the Kato class, and it belongs to $L^p(M)$ for all $ p<\frac{n}{2-\eta}$.  On the one hand, since $\frac{n}{2-\eta}$ can be arbitrarily close to $n$ as $\eta$ goes to 1, the condition $L^n(M)$ in Theorem \ref{thm0} cannot be replaced by $L^p(M),\ \forall p<n$. Thus, $p=n$ is the threshold for the validity of the sharp pointwise Weyl law on the $L^p$ scale.  On the other hand, as pointed by Simon \cite[Section A3]{SimonSurvey}, the Kato class is exactly the border for bounded eigenfunctions (i.e. the pointwise Weyl law). Indeed, $H_V$ may have unbounded eigenfunctions if $V\notin \mathcal{K}(M)$. Specifically, when $n\ge3$, the potential $V(x)=|x|^{-2}(\log(2+|x|^{-1}))^{-\alpha}$ belongs to the Kato class if and only if $\alpha>1$. If $\alpha=1$, it can have unbounded eigenfunctions (see \cite[Section A3]{SimonSurvey}, \cite[page 5]{BSS}). Since $\eta$ in \eqref{rbd} can be arbitrarily close to 0, the error term $o(\la^n)$ in Theorem \ref{thm0} is sharp in the sense that any power improvement is not possible. 
  
  \begin{remark}The possibility of smaller improvements, such as logarithmic improvement, is more subtle. For $x>0$ and integer $k\ge1$, let $\Log(x)=\log(2+x)$, and let \[\Log^{(k)}(x)=(\Log\circ...\circ\Log)(x)\] be the  $k$-iterated logarithm. For example, $\Log^{(2)}(x)=\log(2+\log(2+x))$. When $n\ge3$, $k\ge1$ and $\alpha>1$, the potentials
  	\[V_1(x)=|x|^{-2}(\Log(|x|^{-1}))^{-\alpha},\]
  	\[V_2(x)=|x|^{-2}(\Log(|x|^{-1}))^{-1}(\Log^{(2)}(|x|^{-1}))^{-\alpha},\]
  	...
  	\[V_k(x)=|x|^{-2}(\Log(|x|^{-1}))^{-1}(\Log^{(2)}(|x|^{-1}))^{-1}\cdot\cdot\cdot (\Log^{(k)}(|x|^{-1}))^{-\alpha}\]
  	belong to the Kato class (see \cite[Proposition A.2.5]{SimonSurvey}). One may naturally expect that the error term for $V_k$ is $\approx \la^n(\Log^{(k)}\la)^{-\delta}$ for some $\delta>0$, which means that the error term $o(\la^n)$ in Theorem \ref{thm0} cannot have any iterated-log improvement. It might be verified by refining the proof of Theorem \ref{thm}.  This problem is essentially related to the decay rate of the Fourier transform of $V_k$ in $\mathbb{R}^n$. 
  \end{remark}
  
 Our results are new in any dimension $n\ne3$. Frank-Sabin \cite{fs} proved \eqref{rbd} on general 3-dimensional manifolds for potentials $V(x)=\gamma \rho(d_g(x,x_0))d_g(x,x_0)^{-2+\eta}$, where $\gamma \in\mathbb{R}\setminus\{0\}$ and $\rho\equiv 1$ near zero. They extended the method of Avakumovi\'c \cite{Avakumovic}, which relies on Tauberian theorems and parametrix estimates. However, it seems subtle to adapt the approach to handle other dimensions $n\ne3$, since it relies on the special formula of the resolvent kernel $(-\Delta_{\mathbb{R}^3}+\la)^{-1}(x,y)$, see \cite[Remark 4.5]{fs}.

	We explore a new approach to solve the difficulties in their paper.  This approach extends the classical wave kernel method, see e.g. \cite{SoggeHangzhou}, \cite{hs}, \cite{hz}. We exploit the explicit expansion of the wave kernel $\cos t P_V(x,y)$ on general closed manifolds, which is derived from Duhamel's principle and can be  useful for many other problems related to the Schr\"odinger operators $H_V=-\Delta_g+V$. Moreover, we use two well-known properties of the standard orthonormal eigenbasis $\{e^{ij\cdot x}\}_{j\in \mathbb{Z}^n}$ on the flat torus: \textbf{(1)} uniformly bounded in $j$ and $x$, \textbf{(2)} identically equal to a constant at $0$.
	 These two properties play a key role in the proofs of Propositions \ref{prop1} and \ref{prop2}. Indeed, the first property is used to prove the upper bounds in \eqref{r1} and \eqref{upper}, and the second property is only used for the lower bound \eqref{r12}. On general manifolds (e.g. the sphere $S^n$), it is unknown whether similar orthonormal eigenbasis exists.  The standard orthonormal eigenbasis  on the flat torus enables us to apply Fourier analysis techniques in $\mathbb{R}^n$. Let 
	 \[U(x)=|x|^{-2+\eta}\chi(x),\] where $\chi\in C_0^\infty(\mathbb{R}^n)$ is a radial function satisfying $\hat \chi \ge0$ and $\hat\chi(0)>0$.  Note that the Fourier transform of $|x|^{-2+\eta}$ in $\mathbb{R}^n$ is equal to $|\xi|^{-n+2-\eta}$ (up to a constant factor dependent on $n$). A crucial observation in the proof of Theorem \ref{thm} is the following two-sided estimate (see Lemma \ref{ulemma}) for the Fourier transform of $U(x)$
\begin{equation}\label{four}\hat U(j-k)=\int_{\mathbb{R}^n}|x|^{-2+\eta}\chi(x)e^{ij\cdot x}e^{-ik\cdot x}dx\approx (1+|j-k|)^{-n+2-\eta},\ \forall j,k\in\mathbb{Z}^n.\end{equation}
Due to the limited knowledge of eigenvalues and eigenfunctions (of $-\Delta_g$) on general manifolds,  it is subtle to obtain a similar two-sided estimate for
\[\int_M d_g(x,x_0)^{-2+\eta}\chi(d_g(x,x_0))\overline{e_{j}(x)}e_{k}(x)dx,\]
where $\{e_j\}$ is an orthonormal eigenbasis.
So it would be interesting to generalize Theorem \ref{thm} to any manifold by getting around these difficulties. It is worth mentioning that the only property of $H_V$ used in this paper is just the eigenfunctions bounds in Lemma \ref{sogge888} and Corollary \ref{rough}, which hold on general compact manifolds, and essentially rely on the heat kernel estimates by Li-Yau and Sturm (Lemma \ref{heatk}). The flat torus serves as an important model case, so our approach might be useful for further studies on the eigenvalues and eigenfunctions of $H_V$ on general manifolds.

Moreover, it is elightening to compare the integrated Weyl law with the pointwise version. By applying \cite[Theorem 1.4]{hs} to the potentials on $\mathbb{T}^n$ in Theorem \ref{thm}, one can obtain
\[N_V(\la)=(2\pi)^{-n}\omega_n \mathrm{Vol}_g(M) \, \la^n+r_n(\la)\]	
	with $r_n(\la)$ defined in \eqref{best} (possibly with an extra $\la^\eps$). Note that the error term $r_n(\la)$ in the integrated Weyl law is much smaller than $\la^{n-\eta}$ in \eqref{rbd}. Thus, heuristically $|R(\la,x)|$ may only achieve the bound $\la^{n-\eta}$ near $x_0$, and should be relatively small away from $x_0$.
	
We shall also mention a recent result on the kernel of spectral projection operator
\[ \1_{(\la,\la+\e]}(P_V)(x,y)=\sum_{\tau_k\in(\la,\la+\e]}e_{\tau_k}(x)e_{\tau_k}(y).\] 
As was shown in \cite{bhss}, if $V\in \mathcal{K}(\mathbb{T}^n)\cap L^{n/2}(\mathbb{T}^n)$, which contains the potentials in Theorem \ref{thm}, then we have for any fixed  $\delta>0$,
\begin{equation} \label{pvtorus}|\1_{(\la,\la+\e]}(P_V)(x,x)|\ls \e\la^{n-1},\,\,\,\forall\,\e\gs\la^{-1/3+\delta}.\end{equation}	We may obtain \eqref{rbd} for a larger range of $\eta$ $(0<\eta<\frac43)$, by inserting \eqref{pvtorus} into the proof of Theorem \ref{thm}. One may naturally expect that the optimal range for flat tori is $0<\eta<2$. However, the range $0<\eta<1$ is essentially optimal for general manifolds, since $\eqref{rbd}$ cannot hold for $\eta>1$ on general manifolds.

	The paper is organized as follows. In Section 1, we prove Theorem \ref{thm} by assuming Proposition \ref{prop1} and Proposition \ref{prop2}. In Section 2, we prove Proposition \ref{prop1}. In Section 3, we reduce the proof of  Proposition \ref{prop2} to two cases: ``low-frequency estimates'' and ``high-frequency estimates''. In Section 4, we prove the ``low-frequency estimates''. In Section 5, we prove the ``high-frequency estimates''. Throughout the paper, $A\ls B$ (or $A\gs B$) means $A\le CB$ (or $A\ge CB$) for some implicit constant $C>0$ that may change from line to line. $A\approx B$ means $A\ls B$ and $A\gs B$. All implicit constants $C$ are independent of the parameters $\la$, $\la_j$, $\tau_\ell$.
	
	\noindent\textbf{Acknowledgement. } The authors would like to thank Allan Greenleaf, Christopher Sogge, Yannick Sire,  Rupert L. Frank and Julien Sabin for their suggestions and comments. The authors also thank the anonymous referees for
	very thorough and tremendously helpful reports. The authors are both partially supported by the AMS-Simons Travel Grants.
	
	\section{Proof of Theorem \ref{thm}}
In this section, we use a purtubation argument to reduce  Theorem \ref{thm} to Propositions \ref{prop1} and \ref{prop2}. This argument is valid on general manifolds, and we only start to assume $M=\mathbb{T}^n$ in the proof of Propositions \ref{prop1} and \ref{prop2}. The basic idea in the purtubation argument is to view $Vu$ as an inhomogeneous term in the wave equation $(\partial_t^2-\Delta_g)u=-Vu$, and then apply the Duhamel's principle iterately (see \eqref{duh}, \eqref{iter}). First,	we need the following useful lemmas on general closed manifolds. \begin{lemma}[Spectral projection bounds, \cite{sogge88}, \cite{BSS}]\label{sogge888}Let $n\ge2$ and \[\sigma(p)=\max\{\tfrac{n-1}2(\tfrac12-\tfrac1p),\tfrac{n-1}2-\tfrac{n}{p}\}.\] Then for $\la\ge1$, we have\[\|\1_{[\la,\la+1)}(P^0)\|_{L^2\to L^p}\ls \la^{\sigma(p)},\ 2\le p\le \infty.\]
		If $\,V\in \mathcal{K}(M)\cap L^{n/2}(M)$, then for $\la\ge1$, we have
		\[\|\1_{[\la,\la+1)}(P_V)\|_{L^2\to L^p}\ls \la^{\sigma(p)},\ 2\le p\le \infty.\]	
	\end{lemma}	
These $L^p$-spectral projections bounds can be viewed as the generalized Tomas-Stein restriction estimates on closed manifolds. They are first proved by Sogge \cite{sogge88}, and recently extended to the Schr\"odinger operators with critically singular potentials by Blair-Sire-Sogge \cite{BSS}. These bounds are sharp on $any$ closed manifolds. See \cite[Chapter 5]{fio}. 
	\begin{lemma}[Heat kernel bounds, \cite{Liyau}, \cite{sturm}]\label{heatk}
		If $V\in\mathcal{K}(M)$, then for $0<t\le 1$, there is a uniform constant $c=c_{M,V}>0$ so that
		\[e^{-tH_V}(x,y)\ls \begin{cases}t^{-n/2}e^{-cd_g(x,y)^2/t},\ \text{if}\  d_g(x,y)\le \text{Inj}(M)/2\\
			1,\ \ \text{otherwise}.
		\end{cases}\]
	Here Inj($M$) is the injectivity radius of $M$.
	\end{lemma}
The heat kernel bounds were proved by Li-Yau \cite{Liyau} for smooth potentials, and extended to the Kato class by Sturm \cite{sturm}.
Note that
\[\sum_{\tau_k\le \la}|e_{\tau_k}(x)|^2\ls \sum_{\tau_\ell}e^{-\la^{-2}\tau_\ell^2}|e_{\tau_\ell}(x)|^2=e^{-\la^{-2}H_V}(x,x),\]
so we have the following eigenfunction bounds.
	\begin{corr}[Rough eigenfunction bounds]\label{roub}If $V\in\mathcal{K}(M)$, then for $\la\ge1$
		\begin{equation}\label{rough}\sup_{x \in M}\sum_{\tau_k\le \la}|e_{\tau_k}(x)|^2\le C_V\la^{n}.\end{equation}
	\end{corr}

Our potentials $V$ in Theorem \ref{thm} belong to the Kato class $\mathcal{K}(M)$ as well as $L^{n/2}(M)$, so they satisfy the conditions in these lemmas. Recall that the sharp pointwise Weyl law \eqref{local} for the Laplacian  is equivalent to
	\[\sup_{x\in M}|\ola(P^0)(x,x)-(2\pi)^{-n}\omega_n\la^n|\ls \la^{n-1}.\] So to prove Theorem \ref{thm}, by the triangle inequality it suffices to show for $0<\eta<1$ 
	\begin{equation}\label{uppb}\sup_{x\in M}|\ola(P_V)(x,x)-\ola(P^0)(x,x)|\ls \la^{n-\eta}\end{equation}
	and
	\begin{equation}\label{lowb}|\ola(P_V)(x_0,x_0)-\ola(P^0)(x_0,x_0)|\gs \la^{n-\eta}.\end{equation}
Now we approximate the indicator function by its convolution with a Schwarz function.	Fix a real-valued even function $\varphi\in C_0^\infty(\mathbb{R})$ satisfying \[\1_{[-\frac12,\frac12]}\le\varphi\le \1_{[-1,1]}.\] Let $0<\eps<\frac1{10}\min(\eta,1-\eta)$ and 
\begin{equation}\label{hderi}
    h(\tau)=\frac1\pi\int \varphi(t\la^{1-\eta-\eps} )\frac{\sin \la t}{t}\cos t \tau dt.
\end{equation}
Recall that
	\[\ola(\tau)=\frac1\pi\int \frac{\sin \la t}{t}\cos t \tau dt.\]
Thus $h(\tau)$ is the convolution of $\ola(\tau)$ with a Schwarz function. Integration by parts yields for $\tau>0$ 
	\begin{equation}\label{hda} |h(\tau)-\ola(\tau)|\ls(1+\la^{-1+\eta+\eps}|\tau-\la|)^{-N},\ \forall N,\end{equation}
	and \begin{equation}\label{hdao}
		|\partial_\tau^j h(\tau)|\ls \la^{j(-1+\eta+\eps)}(1+\la^{-1+\eta+\eps}|\tau-\la|)^{-N},\ \forall N,\ j=1,2,....
	\end{equation}
Moreover, $\partial_\tau^jh(\tau)|_{\tau=0}=0$, $j=0,1,2,....$

	By using the spectral projection bounds in Lemma \ref{sogge888} or \eqref{pvtorus}, we have
	\[\sup_{x\in M}|h(P_V)(x,x)-\ola(P_V)(x,x)|\ls \la^{n-\eta-\eps}\]
	\[\sup_{x\in M}|h(P^0)(x,x)-\ola(P^0)(x,x)|\ls \la^{n-\eta-\eps}.\]
	So it suffices to show
	\begin{equation}\label{upp2}\sup_{x\in M}|h(P_V)(x,x)-h(P^0)(x,x)|\ls \la^{n-\eta}\end{equation}
	and 
	\begin{equation}\label{low2}|h(P_V)(x_0,x_0)-h(P^0)(x_0,x_0)|\gs \la^{n-\eta}.\end{equation}Let \[\cos tP^0(x,y)=\sum_{j}\cos t\la_j e_j^0(x)e_j^0(y).\]
	It is the kernel of the solution operator for $f\to (\cos tP^0)f=u^0(t,x)$, where $u^0(t,x)$ solves the wave equation
	\begin{equation}\label{p0wave}
	\begin{cases}\left(\partial_{t}^{2}+H^{0}\right) u^{0}(x, t)=0,\ \ (x, t) \in M \times \mathbb{R},\\
		\left.u^{0}\right|_{t=0}=f,\ \left.\partial_{t} u^{0}\right|_{t=0}=0.
	\end{cases}\end{equation}
	Similarly, 
	\[\bigl(\cos(tP_V)\bigr)(x,y)=\sum_{\tau_\ell} \cos t\tau_\ell \, e_{\tau_\ell}(x)e_{\tau_\ell}(y)\]is the kernel of $f\to \cos (tP_V)f=u_V(x,t)$, where 
	$u_V$ solves the wave equation
	\begin{equation}\label{pvwave}
	    \begin{cases}
		(\partial_t^2+H_V)u_V(x,t)=0, \, \, (x,t)\in M\times \R,\\
		u_V|_{t=0}=f, \, \, \partial_tu_V|_{t=0}=0.	\end{cases}
	\end{equation}
Note that \eqref{p0wave} and \eqref{pvwave} imply that 
\begin{equation}
    	(\partial_t^2+H^0) (\cos (tP_V)f-\cos(tP^0)f)=-V(x)\cos (tP_V)f.
\end{equation}
	Also, since 
\begin{equation}\Bigl(\frac{d}{dt}\Bigr)^j
\Bigl(\cos (tP_V)f-\cos(tP^0)f \Bigr)\Big|_{t=0}
=0, \,\,\,j=0,1,
\end{equation}
by Duhamel's principle for the wave equation, we have
\begin{equation}\label{duh}
    \begin{aligned}
        &\cos (tP_V)f-\cos(tP^0)f \\
&=-\int_0^t
\bigl( \tfrac{\sin(t-s)P^0}{P^0}(V\cos (sP_V)f)\bigr)(x) \, ds
\\
&=-\int_0^t \int_M\int_M
\sum_j \tfrac{\sin(t-s)\la_j}{\la_j}e_j^0(x)\overline{e_j^0(z)}
V(z) \sum_{\tau_\ell}\cos s\tau_\ell e_{\tau_k}(z)\overline{e_{\tau_\ell}(y)} f(y)\, dzdy ds.
    \end{aligned}
\end{equation}
Here $\tfrac{\sin(t-s)\la_j}{\la_j}$ is understood as its continuous extension at $\la_j=0$, and the operator $\frac{\sin((t-s)P^0)}{P^0}$ is defined by the spectral theorem.  The $ds$ integral above can be computed explicitly using the following simple calculus lemma

\begin{lemma}\label{triglemma}  If $\mu\ne \tau$ we have
\begin{equation}\label{2.11}\int_0^t \frac{\sin(t-s)\mu}\mu\,  \cos s\tau \, ds
=\frac{\cos t\tau- \cos t\mu}{\mu^2-\tau^2}.\end{equation}
Similarly, 
\begin{equation}\label{2.12}
\int_0^t \frac{\sin(t-s)\tau}\tau  \, \cos s\tau
\, ds =
\frac{t\sin t\tau}{2\tau}.
\end{equation}
\end{lemma}
The proof is straightforward, see also  \cite[Lemma 2.3]{hs} for more details on the proof. In particular, \eqref{2.12} can be understood as the continuous extension of \eqref{2.11} when $\mu=\tau$. By \eqref{duh} and Lemma~\ref{triglemma}, we have 
\begin{equation}\label{duh1}
    \begin{aligned}
        \cos& tP_V(x,y)-\cos tP^0(x,y)\\
		&=-\sum_{j}\sum_{\tau_\ell}\int_M\int_0^t\tfrac{\sin(t-s)\la_j}{\la_j}\cos s\tau_\ell\  e_j^0(x)\overline{e_j^0(z)}e_{\tau_\ell}(z)\overline{e_{\tau_\ell}(y)}V(z)dzds\\
		&=\sum_{j}\sum_{\tau_\ell}\int_M m(\tau_\ell, \la_j) e_j^0(x)\overline{e_j^0(z)}e_{\tau_\ell}(z)\overline{e_{\tau_\ell}(y)}V(z)dz.
    \end{aligned}
\end{equation}
		where
\begin{equation}\label{mtaumu}
m(\tau,\mu)=
\begin{cases}
\frac{\cos t\tau-\cos t\mu}{\tau^2-\mu^2}, 
\quad \text{if } \, \tau\ne \mu
\\ \\
-\frac{t\sin t\tau}{2\tau}, \quad \ \ \ \ 
\text{if } \, \tau=\mu.
\end{cases}
\end{equation}		
	Thus, by \eqref{duh1} and the definition \eqref{hderi} of $h(\tau)$ we have
		\begin{equation}\label{hvh0}
		    \begin{aligned}
		        h&(P_V)(x,x)-h(P^0)(x,x) \\
		        &=\frac1\pi\int \varphi(t\la^{1-\eta-\eps} )\frac{\sin \la t}{t}\big(\cos t P_V(x,x)-\cos tP^0(x,x)\big) dt \\
		        &=\sum_{j}\sum_{\tau_\ell}\frac1\pi\int \int_M \varphi(t\la^{1-\eta-\eps} )\frac{\sin \la t}{t} m(\tau_\ell, \la_j) e_j^0(x)\overline{e_j^0(z)}e_{\tau_\ell}(z)\overline{e_{\tau_\ell}(x)}V(z)dz dt \\
		       &= \sum_{j}\sum_{\tau_\ell} \int_M  \frac{h(\la_j)-h(\tau_\ell)}{\la_j^2-\tau_\ell^2} e_j^0(x)\overline{e_j^0(z)}e_{\tau_\ell}(z)\overline{e_{\tau_\ell}(x)}V(z)dz
		    \end{aligned}
		\end{equation}
	where $\frac{h(\la_j)-h(\tau_\ell)}{\la_j^2-\tau_\ell^2}=\frac{h'(\la_j)}{2\la_j}$ is still understood as its continuous extension if $\la_j=\tau_\ell$.
So to prove \eqref{upp2} and \eqref{low2},  it remains to show
	\begin{equation}\label{upp3}\sup_{x\in M}\Big|\sum_{j}\sum_{\tau_\ell}\int_M\frac{h(\la_j)-h(\tau_\ell)}{\la_j^2-\tau_\ell^2} e_j^0(x)\overline{e_j^0(z)}e_{\tau_\ell}(z)\overline{e_{\tau_\ell}(x)}V(z)dz\Big|\ls \la^{n-\eta}\end{equation}
	and
	\begin{equation}\label{low3}\Big|\sum_{j}\sum_{\tau_\ell}\int_M\frac{h(\la_j)-h(\tau_\ell)}{\la_j^2-\tau_\ell^2} e_j^0(x_0)\overline{e_j^0(z)}e_{\tau_\ell}(z)\overline{e_{\tau_\ell}(x_0)}V(z)dz\Big|\gs \la^{n-\eta}.\end{equation}

Next, we consider the operator $T$ defined by\begin{align*}(T \phi)(t,x)&=-\int_0^t\frac{\sin((t-s)P^0)}{P^0}(V(\cdot)\phi(s,\cdot))(x)ds\\
		&=-\sum_{j}\int_0^t\frac{\sin(t-s)\la_j}{\la_j}e_j^0(x)\overline{e_j^0(y)}\phi(s,y)V(y)dyds.\end{align*}
For simplicity, we denote \[\phi_0(s,y)=\cos sP^0 f(y)=\sum_{j}\cos s\la_j \int_M \overline{e_j^0(z)}f(z)dz\cdot e_k^0(y)\]
	\[\phi_V(s,y)=\cos sP_V f(y)=\sum_{\tau_\ell}\cos s\tau_\ell \int_M \overline{e_{\tau_\ell}(z)}f(z)dz\cdot e_{\tau_\ell}(y).\]
So we may rewrite \eqref{duh} as \[\phi_V(t,x)=\cos t P_Vf(x)=\cos tP^0f(x)+T\phi_V(t,x)=\phi_0(t,x)+T\phi_V(t,x),\]
	which implies
	\begin{equation}\label{iter}\cos t P_Vf(x)=\cos tP^0f(x)+T\phi_0(t,x)+T^2\phi_V(t,x).\end{equation}
It is straightformward to verify the following identity similar to Lemma \ref{triglemma}
\begin{equation}\label{doub}
	\begin{aligned}
		(-1)^2&\int_0^t\frac{\sin(t-s_1)a_1}{a_1}\int_0^{s_1}\frac{\sin(s_1-s_2)a_2}{a_2}\cos(s_2 a_{3})ds_2ds_1\\ 
		&=\frac{\cos ta_2-\cos ta_1}{(a_2^2-a_1^2)(a_2^2-a_3^2)}+\frac{\cos ta_3-\cos ta_1}{(a_3^2-a_1^2)(a_3^2-a_2^2)}
	\end{aligned}\end{equation}
where $a_1,a_2,a_3\in \mathbb{R}$. As in Lemma \ref{triglemma}, the identity is always valid, if it is understood as its continuous extension when some of $a_1,a_2,a_3$ are equal.
	Thus,  by \eqref{iter}, \eqref{doub}, \eqref{hvh0} and \eqref{hderi} we can write
	\[\sum_{j}\sum_{\tau_\ell}\int_M\frac{h(\la_j)-h(\tau_\ell)}{\la_j^2-\tau_\ell^2} e_j^0(x)\overline{e_j^0(z)}e_{\tau_\ell}(z)\overline{e_{\tau_\ell}(x)}V(z)dz=R_1+R_2,\]
	where
	\[R_1(\la,x)=\sum_{j,k}\frac{h(\la_j)-h(\la_k)}{\la_j^2-\la_k^2} e_j^0(x)V_{jk}\overline{e_k^0(x)},\]
	\[R_2(\la,x)=\sum_{j,k}\sum_{\tau_\ell}\Big(\frac{h(\la_k)-h(\la_j)}{(\la_k^2-\la_j^2)(\la_k^2-\tau_\ell^2)}+\frac{h(\tau_\ell)-h(\la_j)}{(\tau_\ell^2-\la_j^2)(\tau_\ell^2-\la_k^2)}\Big) e_j^0(x)V_{jk}\tilde V_{k\ell}\overline{e_{\tau_\ell}(x)},\]
	\[V_{jk}=\int_M \overline{e_j^0(z)}e_k^0(z)V(z)dz,\]
	\[\tilde V_{k\ell}=\int_M \overline{e_k^0(z)}e_{\tau_\ell}(z)V(z)dz.\]
We claim that on the flat torus $M=\mathbb{T}^n$, the following estimates hold.
\begin{proposition}\label{prop1}We have	\begin{equation}\label{r1}\sup_{x\in M}|R_1(\la,x)|\ls \la^{n-\eta},\end{equation}
and 	\begin{equation}\label{r12}|R_1(\la,x_0)|\gs \la^{n-\eta}.\end{equation}
Thus, we have $\sup_{x\in M}|R_1(\la,x)|\approx \la^{n-\eta}$.
	\end{proposition}
\begin{proposition}\label{prop2} We have \begin{equation}\label{upper}\sup_{x\in M}|R_2(\la,x)|\ls \la^{n-\frac32\eta+\sigma},\ \forall \sigma>0.
	\end{equation}
	\end{proposition}
These two propositions immediately imply \eqref{upp3} and \eqref{low3}, which implies Theorem \ref{thm} by the argument above. As we discussed before, the crucial fact to establish upper bounds and lower bounds in Proposition \ref{prop1} is the two-sided estimate for the Fourier transform in \eqref{four}. To see the intuition behind the proof of Proposition \ref{prop2}, one may heuristically analyze $R_2(\la,x)$ with all $H_V$-eigenfunctions $e_{\tau_\ell}(x)$ and eigenvalues $\tau_\ell^2$ replaced by the Laplace eigenfunctions and eigenvalues on the torus, and then one can easily get a better ``heuristic estimate''  $\la^{n-2\eta+\sigma}$ by explicit computations. Although  the  bound in \eqref{upper} is slightly worse than the ``heuristic estimate'', it is of a lower order   than the lower bound in \eqref{r12}, so it is still sufficient for our purpose.
	
	\section{Proof of Proposition \ref{prop1}}
Let $M=\mathbb{T}^n=\mathbb{R}^n/2\pi\mathbb{Z}^n$ be the flat torus. From now on, we start to work on the flat torus rather than general manifolds, so we always use the index $j\in \mathbb{Z}^n$ for the eigenvalue $\la_j$ for convenience.  Without loss of generality, we may assume $x_0=0\in \mathbb{T}^n$. Then $e_j^0(x)=e^{ ij\cdot x}$ is an eigenfunction of $H^0=-\Delta_g$, and the associated eigenvalue is $\la_j^2=|j|^2$, $ j\in \mathbb{Z}^n$.

	Let $\rho\in C_0^\infty(\mathbb{R}^n)$ be a radial function satisfying $\rho(0)\ne0$ and $\supp \rho\subset (-\pi,\pi)^n$. Without loss of generality, we assume that $\rho(0)>0$. 
	Let $V(x)=|x|^{-2+\eta}\rho(x)$. By the support property of $\rho$, $V(x)$ can be defined on $\mathbb{T}^n$ by the periodic extension, which is still denoted by $V(x)$, for simplicity. So we have
	\begin{equation}\label{vjk}V_{jk}=\int_{\mathbb{T}^n}\overline{e_j^0(z)}e_k^0(z)V(z)dz=\int_{\mathbb{R}^n}V(z)e^{i(k-j)\cdot z}dz=\hat V(j-k).\end{equation}

	Choose a radial function $\chi\in C_0^\infty(\mathbb{R}^n)$ such that $\hat\chi\ge0$, $\hat\chi(0)>0$, $\chi(0)=\rho(0)$ and $\supp \chi\subset (-\pi,\pi)^n$. Let $U(x)=|x|^{-2+\eta}\chi(x)$. For simplicity, we still denote its periodic extension on $\mathbb{T}^n$ by $U(x)$. Then the Fourier transform of $U$ in $\mathbb{R}^n$ is radial, real-valued, and nonnegative:
	\[\hat U(\xi)=|\xi|^{-n+2-\eta}*\hat\chi(\xi)=\int_{\mathbb{R}^n}|\xi-\omega|^{-n+2-\eta}\hat\chi(\omega)d\omega\ge0.\]
	Moreover, it is the convolution of $|\xi|^{-n+2-\eta}$ with a  nonnegative Schwartz function, so we have the following key lemma. 
	\begin{lemma}\label{ulemma}
	\begin{equation}\label{V}\hat U(\xi)\approx (1+|\xi|)^{-n+2-\eta}.\end{equation}
	\end{lemma}
		\begin{proof}
When $|\xi|\ge1$, direct computation gives
	\begin{align*}\hat U(\xi)&\ls \int_{|\omega|<\frac12|\xi|}|\xi|^{-n+2-\eta}(1+|\omega|)^{-N}d\omega+\int_{|\omega|>2|\xi|}|\omega|^{-n+2-\eta}(1+|\omega|)^{-N}d\omega\\
		&\quad\quad\quad\quad\quad\quad\quad\quad\quad\quad\quad\quad\quad\quad\quad+\int_{|\omega|\approx |\xi|}|\omega-\xi|^{-n+2-\eta}|\xi|^{-N}d\omega\\
&\ls |\xi|^{-n+2-\eta},\ \ \forall N>n+1,
		\end{align*}
	and by our assumptions on $\chi$, there exists $0<\delta<\frac12$ such that $\hat\chi(\omega)>\frac12\hat\chi(0)>0$ for $|\omega|<\delta$, so
	\[\hat U(\xi)\ge \frac12\hat\chi(0)\int_{|\omega|<\delta}|\omega-\xi|^{-n+2-\eta}d\omega\approx |\xi|^{-n+2-\eta}.\]
	When $|\xi|<1$, it is easier to see that
\begin{align*}\hat U(\xi)&\ls \int_{|\omega|\le2}|\omega-\xi|^{-n+2-\eta}d\omega+\int_{|\omega|>2}|\omega|^{-n+2-\eta}|\omega|^{-N}d\omega\\
	&\ls 1,\ \ \forall N>n+1,
\end{align*}
	and similarly
	\[\hat U(\xi)\ge \frac12\hat\chi(0)\int_{|\omega|<\delta}|\omega-\xi|^{-n+2-\eta}d\omega\approx 1.\]
	\end{proof}
	Let 
	\[\tilde R_1(\la,x)=\sum_{j}\sum_{k}\frac{h(\la_j)-h(\la_k)}{\la_j^2-\la_k^2} e_j^0(x)U_{jk}\overline{e_k^0(x)},\]
	where 
	\begin{equation}\label{ujk}U_{jk}=U_{kj}=\hat U(j-k)\approx (1+|j-k|)^{-n+2-\eta}.\end{equation}
	Similarly, let 
	\[\tilde R_1'(\la,x)=\sum_{j}\sum_{k}\frac{\ola(\la_j)-\ola(\la_k)}{\la_j^2-\la_k^2} e_j^0(x)U_{jk}\overline{e_k^0(x)},\]
		where we define $\frac{\ola(\la_j)-\ola(\la_k)}{\la_j^2-\la_k^2}=0$ when $\la_j=\la_k$.

\begin{lemma}\label{mainlemma}
\begin{equation}\label{R1x0}|\tilde R_1'(\la,x_0)|\gs \la^{n-\eta},\end{equation}
	\begin{equation}\label{R1upper}\sup_{x\in M}|\tilde R_1'(\la,x)|\ls \la^{n-\eta},\end{equation}
	\begin{equation}\label{Rdiff1}\sup_{x\in M}| \tilde R_1(\la,x)-\tilde R_1'(\la,x)|\ls \la^{n-\eta-\eps+\sigma},\ \forall \sigma>0,\end{equation}
	\begin{equation}\label{Rdiff2}\sup_{x\in M}| R_1(\la,x)-\tilde R_1(\la,x)|\ls \la^{n-\eta-\eps+\sigma},\ \forall \sigma>0.\end{equation}
\end{lemma} 
Recall that $0<\eta<1$ and $0<\eps<\frac1{10}\min(\eta,1-\eta)$. These four estimates immediately imply Proposition \ref{prop1}.

\noindent \textbf{Proof of \eqref{R1x0}.} Note that $\la_j=|j|$, and $e_j^0(x_0)=1$ for any $j\in\mathbb{Z}^n$. Moreover, we have $\ola(|j|)-\ola(|k|)=0$ if  $|j|$ and $|k|$ are less than $\la$  (or larger than $\la$) simultaneously. Since $U_{jk}=U_{kj}$ and they are positive by Lemma \ref{ulemma}, we have
	\begin{align*}|\tilde R_1'(\la,x_0)|&=2\sum_{|j|<\la}\sum_{|k|\ge\la}\frac1{|k|^2-|j|^2}U_{jk}\\
		&\approx \sum_{|j|<\la}\sum_{|k|\ge\la}\frac1{|k|^2-|j|^2} (1+|k-j|)^{-n+2-\eta}\\
		&\ge \sum_{|j|<\la}\sum_{|k|>2\la} |k|^{-n-\eta}\\
		&\approx \la^{n-\eta}.\end{align*}
	\[\]
\noindent \textbf{Proof of \eqref{R1upper}.} We define for $m\in \mathbb{N}$ and $\ell\in \mathbb{Z}$
	\[S_{\ell m}=\{(j,k)\in\mathbb{Z}^{2n}: \la/2<|j|<\la\le|k|<2\la, |k-j|\approx 2^m,\ {\rm and}\ |k|-|j|\approx 2^\ell\},\]
	\[J_{\ell m}=\{j:(j,k)\in S_{\ell m}\ \text{for some}\ k\},\]
	\[K_{\ell m}(j)=\{k:(j,k)\in S_{\ell m}\}.\]
	If $S_{\ell m}$ is nonempty, then we have  $\la^{-1}\ls 2^\ell\ls 2^m\ls\la$ and $J_{\ell m}\subset \{j: 0< \la-|j|\ls 2^\ell\}$. So we have the following simple estimates on the number of lattice points:
	\begin{equation}\label{numj}\# J_{\ell m}\ls \la^{n-1}(2^\ell+1),\end{equation}
	\begin{equation}\label{numk}\# K_{\ell m}(j)\ls 2^{(n-1)m}(2^\ell+1),\ \forall j\in \mathbb{Z}^n.\end{equation}
	The first bound means the number of lattice points in the annulus of outer radius $\la$ and width $2^\ell$. The second bound means the number of lattice points in the intersection of the annulus and a ball of radius $2^m$.
	Thus
\begin{equation}\label{nums}\# S_{\ell m}\ls \# J_{\ell m}\cdot \max_j \# K_{\ell m}(j)\ls \la^{n-1}2^{(n-1)m}(2^\ell+1)^2.\end{equation}
Note that $|e_j^0(x)|\le1$ for any $j\in\mathbb{Z}^n$ and $x\in M$. Thus for any $x\in M$ we have
	\begin{align*}|&\tilde R_1'(\la,x)|\ls \sum_{|j|<\la}\sum_{|k|\ge \la}\frac1{|k|^2-|j|^2}(1+|k-j|)^{-n+2-\eta}\\
		&\ls \sum_{|j|\le\la/2}\sum_{|k|\ge\la}|k|^{-n-\eta}+\sum_{|j|<\la}\sum_{|k|\ge2\la}|k|^{-n-\eta} \\
		&\qquad \qquad+\sum_{\la/2< |j|<\la}\sum_{\la\le|k|< 2\la}\frac1{|k|^2-|j|^2} (1+|k-j|)^{-n+2-\eta}\\
		&\ls \la^{n-\eta}+\sum_{\ell}\sum_{m}\sum_{(j,k)\in S_{\ell m}}\la^{-1}2^{-\ell}2^{(-n+2-\eta)m}\\
		&\ls \la^{n-\eta}+\sum_{\ell\in \mathbb{Z}:\la^{-1}\ls 2^{\ell}\le \la}\sum_{m\in \mathbb{N}:2^\ell\ls 2^m\ls \la}\la^{-1}2^{-\ell}2^{(-n+2-\eta)m}\cdot\la^{n-1}2^{(n-1)m}(2^\ell+1)^2\\
		&\ls \la^{n-\eta}.
	\end{align*}
	\[\]

\noindent \textbf{Proof of \eqref{Rdiff1}.}  Let $\psi(j)=h(|j|)-\ola(|j|)$. By \eqref{hda} we have \begin{equation}\label{psi}|\psi(j)|\ls(1+\la^{-1+\eta+\eps}||j|-\la|)^{-N},\ \forall N.\end{equation}  We claim that
	\begin{equation}\label{claimpsi}|\tilde R_1-\tilde R_1'|(\la,x)\le\sum_{j}\sum_{k}\Big|\frac{\psi(j)-\psi(k)}{|j|^2-|k|^2} U_{jk}\Big|\ls \la^{n-\eta-\eps+\sigma},\ \forall \sigma>0.\end{equation}
Here when $|j|=|k|$, we define
\begin{equation}\label{jek}\frac{\psi(j)-\psi(k)}{|j|^2-|k|^2}=\frac{h'(|j|)}{2|j|}.\end{equation}When $|j|=|k|=0$, it is defined to be 0, since $h'(0)=h''(0)=0$.

Now we prove the claim \eqref{claimpsi}. By the symmetry between $j$ and $k$, we may split the sum into the following six cases:
	\begin{enumerate}[(i)]
		\item $\la/2<|j|<2\la,\ \la/2<|k|<2\la$
		\item $\la/2<|j|<2\la,\ |k|\ge2\la$
		\item $\la/2<|j|<2\la,\ |k|\le\la/2$
		\item $|j|\le\la/2,\ |k|\le\la/2$ 
		\item $|j|\le\la/2,\ |k|\ge2\la$
		\item $|j|\ge2\la,\ |k|\ge2\la$.
	\end{enumerate}
In the following, we show that case (i) contributes to the desired bound $\la^{n-\eta-\eps+\sigma}$, and other cases satisfy better bounds.

In case (i), by symmetry we just need to show
	\begin{equation}\label{middle}
		\sum_{\la/2<|j|<2\la}\sum_{\la/2<|k|<2\la, |k|\ne|j|}\Big|\frac{\psi(j)}{|j|^2-|k|^2} U_{jk}\Big|\ls\la^{n-\eta-\eps+\sigma},\ \forall \sigma>0,
	\end{equation}

\begin{equation}\label{middleequal}
	\sum_{\la/2<|j|<2\la}\sum_{|k|=|j|}\Big|\frac{h'(|j|)}{2|j|} U_{jk}\Big|\ls\la^{n-1-\eta}.
\end{equation}
	
	Fix a Littlewood-Paley bump function
	$\beta\in C^\infty_0((1/2,2))$ satisfying
	$$
	\sum_{\ell=-\infty}^\infty \beta(2^{-\ell} s)=1, \quad s>0.$$
	Let $\ell_0$ be the largest integer such that $2^{\ell_0}\le \la^{1-\eta-\e}$. Recall that \eqref{psi} and \eqref{ujk} hold: 
	\[|\psi(j)|\ls (1+2^{-\ell_0}||j|-\la|)^{-N},\ \forall N,\]
	\[U_{jk}\approx (1+|j-k|)^{-n+2-\eta}.\]
	Since $j,k\in \mathbb{Z}^n$, $|j|\approx |k|\approx \la$ and $|j|\ne|k|$, we have  $\la^{-1}\ls||j|-|k||\ls \la$ and $|j|+|k|\approx \la$.  Then for $0<\eta<1$ we have
	\begin{align*}
		&\sum_\ell\sum_{\la/2<|j|,\, |k|<2\la}\Big|\frac{\psi(j)}{|j|^2-|k|^2} U_{jk}\beta(2^{-\ell}||j|-|k||)\Big|\\
		&\ls  \sum_\ell\sum_{\la/2<|j|,\, |k|<2\la}2^{-\ell}\la^{-1}(1+2^{-\ell_0}||j|-\la|)^{-N}(1+|j-k|)^{-n+2-\eta}\beta(2^{-\ell}||j|-|k||)\\
		&\ls \sum_{1<2^\ell\ls\la}\sum_{p\ge\ell_0}\sum_{m\ge0} 2^{-\ell}\la^{-1}(1+2^{-\ell_0}2^p)^{-N}2^{(-n+2-\eta)m}\cdot \la^{n-1}2^p\cdot 2^{(n-1)m}2^\ell\\
		&\ \ \ \ \ \ \ +\sum_{\la^{-1}\ls 2^\ell\le1}\sum_{p\ge\ell_0}\sum_{m\ge0}2^{-\ell}\la^{-1}(1+2^{-\ell_0}2^p)^{-N}2^{(-n+2-\eta)m}\cdot \la^{n-1}2^{p}\cdot F_n(\la,2^m)2^{\ell}\la\\
		&\ls \la^{n-2\eta-\eps}\log\la+\la^{n-\eta-\eps+\sigma},\ \forall \sigma>0.
	\end{align*}
	Here $F_n(\la,r)$ is the maximal number of lattice points on a spherical cap of size $r$ of the sphere $\la S^{n-1}$. In the calculation above, we count the lattice points in the intersection of a ball of radius $2^m$ and an annulus of outer radius $\la$ and width $2^\ell$.  We exploit the following precise estimates for $F_n(\la,r)$: $\forall \sigma>0$,
	\begin{equation}\label{cap}F_n(\la,r)\ls \begin{cases}\la^{\sigma}(r^{n-1}\la^{-1}+r^{n-3}),\ \ \ n\ge5\\
		\la^{\sigma}(r^3\la^{-1}+r^{3/2}),\ \ \ \ \ \ \ n=4\\
			\la^{\sigma}(r+1),\ \ \ \ \ \ \ \ \ \ \ \ \ \ \ \ \  n=3\\
					\la^{\sigma},\ \ \ \ \ \ \ \ \ \quad \quad \quad \quad \quad \ \ n=2
	\end{cases}\end{equation}
as well as the trivial estimate 
\begin{equation}\label{trivialcap}F_n(\la,r)\ls (1+r)^{n-1}.\end{equation}
	The precise estimates can be found in the work of Bourgain-Rudnick \cite[Proposition 1.4]{br}. Sometimes the trivial estimate is sufficient for our purpose (e.g. the proof of \eqref{R1upper}). To prove \eqref{middleequal}, we recall that \eqref{hdao} gives
	\[|h'(|j|)|\ls 2^{-\ell_0}(1+2^{-\ell_0}||j|-\la|)^{-N}.\]
	So	we have \begin{align*}
	\sum_{\la/2<|j|<2\la}&\sum_{|k|=|j|}\Big|\frac{h'(|j|)}{2|j|} U_{jk}\Big|\\
	&\ls \sum_{\la/2<|j|<2\la}\sum_{|k|=|j|}\la^{-1}2^{-\ell_0}(1+2^{-\ell_0}||j|-\la|)^{-N}(1+|j-k|)^{-n+2-\eta}\\
	&\ls \sum_{p\ge\ell_0}\sum_{m\ge0}\la^{-1}2^{-\ell_0}(1+2^{-\ell_0}2^p)^{-N}2^{(-n+2-\eta)m}\cdot \la^{n-1}2^p\cdot 2^{(n-1)m}	\\
	&\ls \la^{n-1-\eta}.
		\end{align*}
	Here we only use the trivial estimate \eqref{trivialcap}. 
	
	In case (ii), we need to show
	\begin{equation}\label{case2.1}
		\sum_{\la/2<|j|<2\la}\sum_{|k|\ge3\la}\Big|\frac{\psi(j)-\psi(k)}{|j|^2-|k|^2} U_{jk}\Big|\ls \la^{n-2\eta-\eps},
	\end{equation}
		\begin{equation}\label{case2.2}
		\sum_{\la/2<|j|\le\frac32\la}\sum_{2\la\le|k|<3\la}\Big|\frac{\psi(j)-\psi(k)}{|j|^2-|k|^2} U_{jk}\Big|\ls \la^{n-2\eta-\eps},
	\end{equation}
\begin{equation}\label{case2.3}
	\sum_{\frac32\la<|j|<2\la}\sum_{2\la\le|k|<3\la}\Big|\frac{\psi(j)-\psi(k)}{|j|^2-|k|^2} U_{jk}\Big|\ls \la^{-N},\ \forall N.
\end{equation}
To prove \eqref{case2.1}, we apply \eqref{psi} and the fact that $|k-j|\approx |k|\approx |k|-|j|\approx |k|-\la$. Then
\[|\psi(j)-\psi(k)|\ls (1+2^{-\ell_0}||j|-\la|)^{-N}+|k|^{-N},\ \forall N.\] So 
	\begin{align*}	\sum_{\la/2<|j|<2\la}&\sum_{|k|\ge3\la}\Big|\frac{\psi(j)-\psi(k)}{|j|^2-|k|^2} U_{jk}\Big|\\
	&\ls \sum_{\la/2<|j|<2\la}\sum_{|k|\ge2\la} ((1+2^{-\ell_0}||j|-\la|)^{-N}+|k|^{-N})|k|^{-2}|k|^{-n+2-\eta}\\
		&\ls \la^{n-2\eta-\eps}.
	\end{align*}
Similarly, for \eqref{case2.2}, we have
	\begin{align*}	\sum_{\la/2<|j|\le\frac32\la}&\sum_{2\la\le |k|<3\la}\Big|\frac{\psi(j)-\psi(k)}{|j|^2-|k|^2} U_{jk}\Big| \\
	&\ls \sum_{\la/2<|j|\le\frac32\la}\sum_{2\la\le |k|<3\la} ((1+2^{-\ell_0}||j|-\la|)^{-N}+\la^{-N})\la^{-2}\la^{-n+2-\eta}\\
	&\ls \la^{n-2\eta-\eps}.
\end{align*}
Moreover, \eqref{case2.3} easily follows from the fact that $|\psi(j)-\psi(k)|\ls \la^{-N}$, and $|k|^2-|j|^2\ge1$.

In case (iii), it suffices to show
	\begin{equation}\label{case3.1}
	\sum_{\frac34\la<|j|<2\la}\sum_{|k|\le\la/2}\Big|\frac{\psi(j)-\psi(k)}{|j|^2-|k|^2} U_{jk}\Big|\ls \la^{n-2\eta-\eps},\end{equation}
	\begin{equation}\label{case3.2}
	\sum_{\la/2<|j|\le\frac34\la}\sum_{|k|\le\la/2}\Big|\frac{\psi(j)-\psi(k)}{|j|^2-|k|^2} U_{jk}\Big|\ls \la^{-N},\ \forall N.
\end{equation}
To show \eqref{case3.1}, we use \eqref{psi} and the fact that $|j|-|k|\approx \la\approx \la-|k|\approx |j-k|$. So
	\begin{align*}
\sum_{\frac34\la<|j|<2\la}&\sum_{|k|\le\la/2}\Big|\frac{\psi(j)-\psi(k)}{|j|^2-|k|^2} U_{jk}\Big| \\
&\ls \sum_{\frac34\la<|j|<2\la}\sum_{|k|\le\la/2}((1+2^{-\ell_0}||j|-\la|)^{-N}+\la^{-N})\la^{-2}\la^{-n+2-\eta}\\
&\ls \la^{n-2\eta-\eps}.
	\end{align*}
Moreover,  \eqref{case3.2} follows from the fact that $|\psi(j)-\psi(k)|\ls \la^{-N}$ and $|j|^2-|k|^2\ge1$.

In case (iv), we  need to prove
	\begin{equation}\label{case4}
	\sum_{|j|\le\la/2}\sum_{|k|\le\la/2}\Big|\frac{\psi(j)-\psi(k)}{|j|^2-|k|^2} U_{jk}\Big|\ls \la^{-N},\ \forall N,\end{equation}
which easily follows from the fact that \[\Big|\frac{\psi(j)-\psi(k)}{|j|^2-|k|^2}U_{jk}\Big|\ls \la^{-N}.\]
Note that  it is still valid when $|j|=|k|$, by  \eqref{jek}.

 In case (v), we need to show
	\begin{equation}\label{case5}
	\sum_{|j|\le\la/2}\sum_{|k|\ge2\la}\Big|\frac{\psi(j)-\psi(k)}{|j|^2-|k|^2} U_{jk}\Big|\ls \la^{-N},\ \forall N.\end{equation}
It follows from the estimate:
\[\Big|\frac{\psi(j)-\psi(k)}{|j|^2-|k|^2} U_{jk}\Big|\ls \la^{-N}|k|^{-2}|k|^{-n+2-\eta}.\]

In case (vi), it suffices to prove 
	\begin{equation}\label{case6.1}
	\sum_{|j|\ge2\la}\sum_{\frac34|j|\le|k|\le2|j|}\Big|\frac{\psi(j)-\psi(k)}{|j|^2-|k|^2} U_{jk}\Big|\ls \la^{-N},\ \forall N\end{equation}
	\begin{equation}\label{case6.2}
	\sum_{|j|\ge2\la}\sum_{|k|>2|j|}\Big|\frac{\psi(j)-\psi(k)}{|j|^2-|k|^2} U_{jk}\Big|\ls \la^{-N},\ \forall N.\end{equation}
	\begin{equation}\label{case6.3}
	\sum_{|j|\ge2\la}\sum_{2\la\le|k|<\frac34|j|}\Big|\frac{\psi(j)-\psi(k)}{|j|^2-|k|^2} U_{jk}\Big|\ls \la^{-N},\ \forall N.\end{equation}
To prove \eqref{case6.1}, we notice that when $|k|\approx |j|>2\la$, \[\Big|\frac{\psi(j)-\psi(k)}{|j|^2-|k|^2} \Big|\ls |j|^{-N}.\]
It is still valid when $|j|=|k|$, by \eqref{jek}. So we get	\begin{align*}
\sum_{|j|\ge2\la}\sum_{\frac34|j|\le|k|\le2|j|}\Big|\frac{\psi(j)-\psi(k)}{|j|^2-|k|^2} U_{jk}\Big|&\ls\sum_{|j|\ge2\la}\sum_{\frac34|j|\le|k|\le2|j|}|j|^{-N}(1+|j-k|)^{-N}\\
&\ls  \la^{-N}.
\end{align*}
The estimate \eqref{case6.2} follows from the fact that when $|k|>2|j|$
\[\Big|\frac{\psi(j)-\psi(k)}{|j|^2-|k|^2} U_{jk}\Big|\ls |j|^{-N}|k|^{-2}|k|^{-n+2-\eta}.\]
Similarly, we obtain \eqref{case6.3} from the fact that when $|k|<\frac34|j|$,
\[\Big|\frac{\psi(j)-\psi(k)}{|j|^2-|k|^2} U_{jk}\Big|\ls |k|^{-N}|j|^{-2}|j|^{-n+2-\eta}.\]
So we complete all cases and finish the proof of \eqref{claimpsi}.
\[\]
	
\noindent \textbf{Proof of \eqref{Rdiff2}.} Since $\rho$ and $\chi$ are smooth radial  functions satisfying $\rho(0)=\chi(0)$, we may write the difference
		\[V(x)-U(x)=|x|^{-2+\eta}(\rho(x)-\chi(x))=|x|^{-1+\eta}\rho_1(x)\]
		for some radial function $\rho_1\in C_0^\infty(\mathbb{R}^n)$.
		Then by applying  the proof of Lemma \ref{ulemma}, we may obtain
		\begin{equation}\label{vb}
		    		|\hat V(\xi)-\hat U(\xi)|\ls (1+|\xi|)^{-n+1-\eta}.
		\end{equation}
	We claim that  \begin{equation}\label{smoother}|R_1-\tilde R_1|(\la,x)\ls \sum_{j}\sum_{k}\Big|\frac{h(|j|)-h(|k|)}{|j|^2-|k|^2}\Big|(1+|j-k|)^{-n+1-\eta}\ls \la^{n-\eta-\eps+\sigma},\ \forall \sigma>0.
			\end{equation}
	Indeed, since $h(|j|)=\ola(|j|)+\psi(j)$, we may split the sum into two parts, and it suffices to prove that
	\begin{equation}\label{part2} \sum_{j}\sum_{k}\Big|\frac{\psi(j)-\psi(k)}{|j|^2-|k|^2}\Big|(1+|j-k|)^{-n+1-\eta}\ls\la^{n-\eta-\eps+\sigma},\ \forall \sigma>0.
	\end{equation}
	\begin{equation}\label{part1} \sum_{j}\sum_{k}\Big|\frac{\ola(|j|)-\ola(|k|)}{|j|^2-|k|^2}\Big|(1+|j-k|)^{-n+1-\eta}\ls\la^{n-\eta-1}.
	\end{equation}
		
	The first bound trivially follows from \eqref{claimpsi}, since 
\[U_{jk}\approx (1+|j-k|)^{-n+2-\eta}\ge(1+|j-k|)^{-n+1-\eta} .\]
The second bound can be obtained by modifying the proof of \eqref{R1upper}. Indeed, the sum in \eqref{part1} is bounded by
\begin{align*}\sum_{|j|<\la}\sum_{|k|\ge \la}&\frac1{|k|^2-|j|^2}(1+|k-j|)^{-n+1-\eta}\\
	&\ls \sum_{|j|\le\la/2}\sum_{|k|\ge\la}|k|^{-n-\eta-1}+\sum_{|j|<\la}\sum_{|k|\ge2\la}|k|^{-n-\eta-1}\\
	&\qquad \qquad+\sum_{\la/2< |j|<\la}\sum_{\la\le|k|< 2\la}\frac1{|k|^2-|j|^2} (1+|k-j|)^{-n+1-\eta}\\
	&\ls \la^{n-\eta-1}+\sum_{\ell}\sum_{m}\sum_{(j,k)\in S_{\ell m}}\la^{-1}2^{-\ell}2^{(-n+1-\eta)m}.
\end{align*}
When $2^{\ell}<1$, we may slightly refine the estimate \eqref{numj}:
\begin{equation}\label{newnumj}\# J_{\ell m}\ls \la^{n-2+\sigma}\cdot2^{\ell}\la= \la^{n-1+\sigma}2^\ell,\ \forall \sigma>0.\end{equation}
	Here we use the fact that the number of lattice points on the sphere $\la S^{n-1}$ is $O(\la^{n-2+\sigma})$. See e.g. \cite{br}, \cite{hztorus}. Then by using \eqref{numj}, \eqref{numk} and \eqref{newnumj}, we get
	\begin{align*}
\sum_{\ell}&\sum_{m}\sum_{(j,k)\in S_{\ell m}}\la^{-1}2^{-\ell}2^{(-n+1-\eta)m}\\
&\ls \sum_{1\le2^\ell\le\la}\sum_{2^\ell\ls 2^m\ls\la}\la^{-1}2^{-\ell}2^{(-n+1-\eta)m}\cdot \la^{n-1}2^{(n-1)m}2^{2\ell}\\
&\quad\quad +\sum_{\la^{-1}\ls2^\ell<1}\sum_{1\le 2^m\ls\la}\la^{-1}2^{-\ell}2^{(-n+1-\eta)m}\cdot \la^{n-1+\sigma}2^\ell\cdot 2^{(n-1)m}\\
&\ls \la^{n-\eta-1}+\la^{n-2+\sigma},\ \forall \sigma>0.
\end{align*}
	
	 So we complete the proof of Lemma \ref{mainlemma}.

	\section{Proof of Proposition \ref{prop2}}
In this section, we prove the upper bound of $R_2$ in \eqref{upper}. For simplicity, we just assume $V\in L^{\frac n{2-\eta}}(M)$, though actually the original $V(x)=|x|^{-2+\eta}\rho(x)$ is in $L^{\frac n{2-\eta}-\sigma}(M)$, $\forall \sigma>0$.  Applying the same argument to the original $V$ leads to an extra harmless $\la^\sigma$ ($\forall \sigma>0$) factor in the upper bound of $R_2$. As before, we fix a Littlewood-Paley bump function
$\beta\in C^\infty_0((1/2,2))$ satisfying
$$
\sum_{\ell=-\infty}^\infty \beta(2^{-\ell} s)=1, \quad s>0.$$
Let $\ell_0$ be the largest integer such that $2^{\ell_0}\le \la^{1-\eta-\e}$. Let
\begin{equation}\label{beta}\beta_{\ell_0}(w)=\sum_{\ell\le \ell_0} \beta(2^{-\ell}|w|)
	\in C^\infty_0((-2^{\ell_0+1},2^{\ell_0+1})),\end{equation}
as well as 
\begin{equation}\label{beta1}\beta_{\ell}(w)=\beta(2^{-\ell}|w|),\,\,\,\text{for}\,\,\, \ell>\ell_0.
\end{equation}
For $j, k\in \mathbb{Z}^n$ and $\tau\in \R$, let
\begin{equation}\label{2.1}
	f(\tau)=\frac{h(\tau)-h(|j|)}{\tau^2-|j|^2},
\end{equation}
\begin{equation}\label{2.2}
	\begin{aligned}
		a_{jk\tau_\ell}=&\frac{f(|k|)-f(\tau_\ell)}{|k|^2-\tau_\ell^2}\\
		=&\frac{h(|k|)-h(|j|)}{(|k|^2-|j|^2)(|k|^2-\tau_\ell^2)}+\frac{h(\tau_\ell)-h(|j|)}{(\tau_\ell^2-|j|^2)(\tau_\ell^2-|k|^2)}\\
		:=&a_{jk\tau_\ell}^0+a_{jk\tau_\ell}^1.
	\end{aligned}
\end{equation}
\begin{remark}The bounds on $|a_{jk\tau_\ell}|$ will be estimated. When $|\tau_\ell-|k||$ can be small (e.g. $|\tau_\ell-|k||\ls 2^{\ell_0}$, $\tau_\ell\approx |k|$), we always apply the mean value theorem to $f(\tau)$ and treat $a_{jk\tau_\ell}$ as a whole in this case. In all the other cases, $a_{jk\tau_\ell}^0$ and $a_{jk\tau_\ell}^1$ are treated separately. Similarly, when $||k|-|j||$ or $|\tau_\ell-|j||$ can be small, we also apply the mean value theorem to $h(\tau)$.
\end{remark}

Now we rewrite $R_2$ by the following
\begin{equation}\label{2.3}
	\begin{aligned}
		R_2=&\sum_{j, k, \tau_\ell}\int a_{jk\tau_\ell}\,e_j^0(x)V_{jk}\overline{e_k^0(z)}V(z)e_{\tau_\ell}(z) \overline{e_{\tau_\ell}(x)}dz\\
		=&\sum_{\ell_1\ge\ell_0}\sum_{\ell_2\ge\ell_0}\sum_{j, k, \tau_\ell}\beta_{\ell_1}(\tau_\ell-|k|)\beta_{\ell_2}(\tau_\ell-|j|)\\ 
		&\cdot \int a_{jk\tau_\ell}\,e_j^0(x)V_{jk}\overline{e_k^0(z)}V(z)e_{\tau_\ell}(z) \overline{e_{\tau_\ell}(x)}dz,
	\end{aligned}
\end{equation}
where 
\begin{equation}\label{vjkb}|V_{jk}| \ls (1+|j-k|)^{-n+2-\eta}\end{equation} 
by \eqref{vjk} and \eqref{vb}.

\subsection{Some useful lemmas}
	To prove the desired bound for $R_2$, we need the following lemma which separate the contribution of the various components in \eqref{2.3} above.  Fix $p_0=\frac{2n}{n-2+\eta}$. Then Sogge's exponent 
\[\sigma(p_0)=\begin{cases}
	\frac{(n-1)(2-\eta)}{4n},\ \ \ \eta>\frac2{n+1}\\
	\frac{1-\eta}2,\ \ \ \ \ \ \ \ \ \ \ \eta\le \frac2{n+1}\ \ .
\end{cases}\]
Note that for $\eta\in (0,1]$ we always have $\sigma(p_0)\le \frac12-\frac14\eta$.
\begin{lemma}\label{delta1}  Let $I\subset \R_+$ and for
	eigenvalues $\tau_{\ell}\in I$ assume that $\delta_{\tau_{\ell}}
	\in [0,\delta]$.  Then if $m\in C^1(\R_+\times M)$, and $V\in L^{\frac{n}{2-\eta}}(M)$, we have 
	\begin{multline}\label{2.4}
		\int_M \Bigl| \sum_{\tau_{\ell}\in I}m(\delta_{\tau_{\ell}},y) \,
		a_{\ell} V(y)e_{\tau_{\ell}}(y)\Bigr| \, dy
		\\
		\le \|V\|_{L^{\frac{n}{2-\eta}}(M)}\cdot \Bigl(\, \|m(0, \, \cdot\, )\|_{L^{p_0}(M)}
		+\int_0^\delta \bigl\| \tfrac\partial{\partial s}
		m(s, \, \cdot\, )\bigr\|_{L^{p_0}(M)}\, ds
		\, \Bigr)\\ \times  \sup_{s\in [0, \delta]}\bigl\| \sum_{\tau_{\ell}\in I}\1_{[0,\delta_{\tau_{\ell}}]}(s)a_{\ell}e_{\tau_{\ell}}\, \bigr\|_{L^{p_0}(M)}.
	\end{multline}
\end{lemma}
\begin{proof}
	We shall use the fact that
	$$m(\delta_{\tau_{\ell}},y)=m(0,y)+\int_0^{\delta}
	\1_{[0,\delta_{\tau_{\ell}}]}(s)\, \tfrac\partial{\partial s}
	m(s,y) \, ds,$$
	where $\1_{[0,\delta_{\tau_{\ell}}]}(s)$ is the indicator function of the interval $[0,\delta_{\tau_{\ell}}]
	\subset [0,\delta]$.
	Therefore, by H\"older's inequality and Minkowski's inequality,
	the left side of \eqref{2.4} is dominated by $\|V\|_{L^{\frac{n}{2-\eta}}(M)}$ times
	\begin{align*}
		&\big(\int_M \bigl|\, m(0,y)\cdot \sum_{\tau_{\ell}\in I}a_{\ell}
		e_{\tau_{\ell}}(y)\, \bigr|^{\frac{n}{n-2+\eta}} \, dy\big)^{\frac{n-2+\eta}{n}} \\
		+& \big(\int_M \bigl| \, \sum_{\tau_{\ell}\in I}\int_0^\delta
		\1_{[0,\delta_{\tau_{\ell}}]}(s)\tfrac\partial{\partial s}
		m(s,y) a_{\ell} e_{\tau_{\ell}}(y)\, ds \, \bigr|^{\frac{n}{n-2+\eta}} \, dy \big)^{\frac{n-2+\eta}{n}}
		\\
		&\le
		\|m(0,\, \cdot\, )\|_{p_0} \cdot
		\|\sum_{\tau_{\ell}\in I}a_{\ell}e_{\tau_{\ell}}\|_{p_0}
		+\int_0^\delta\bigl( \,
		\bigl\|\tfrac\partial{\partial s}
		m(s,\, \cdot \, )\bigr\|_{p_0} \cdot \bigl\| \sum_{\tau_{\ell}\in I}\1_{[0,\delta_{\tau_{\ell}}]}(s)a_{\ell}e_{\tau_{\ell}}\, \bigr\|_{p_0}\, \bigr)\, ds \\
		&\le
		\|m(0,\, \cdot\, )\|_{p_0} \cdot
		\|\sum_{\tau_{\ell}\in I}a_{\ell}e_{\tau_{\ell}}\|_{p_1}
		+\int_0^\delta \,
		\|\tfrac\partial{\partial s}
		m(s,\, \cdot \, )\|_{p_0} ds \cdot \sup_{s\in [0, \delta]}\bigl\| \sum_{\tau_{\ell}\in I}\1_{[0,\delta_{\tau_{\ell}}]}(s)a_{\ell}e_{\tau_{\ell}}\, \bigr\|_{p_0}\,\,
		\\
		&\le \Bigl(\, \|m(0, \, \cdot\, )\|_{p_0}
		+\int_0^\delta \bigl\| \tfrac\partial{\partial s}
		m(s, \, \cdot\, )\bigr\|_{p_0}\, ds
		\, \Bigr) \times   \sup_{s\in [0, \delta]}\bigl\| \sum_{\tau_{\ell}\in I}\1_{[0,\delta_{\tau_{\ell}}]}(s)a_{\ell}e_{\tau_{\ell}}\, \bigr\|_{p_0}.
	\end{align*}
\end{proof}

When we apply Lemma \ref{delta1} in the following sections, the interval $I$ has left endpoint $\xi$ and length $\delta$. Moreover, $\delta_{\tau_\ell}=\tau_\ell-\xi\in [0,\delta]$, $a_\ell=e_{\tau_\ell}(x)$, and 
\begin{equation}
   m (\delta_{\tau_\ell},y)=\sum_{|k|\in I_1} b_{k,\tau_\ell}e_k^0(y),
\end{equation}
 for some interval $I_1$. 

If $\tau_\ell\le A$, $|k|\le B$, $|b_{k,\tau_\ell}|\le C$, and $|\partial_{\tau_\ell} b_{k,\tau_\ell}|\ls \delta^{-1}C$, then by using Lemma \ref{sogge888}, we get
\[\sup_{s\in [0, \delta]}\bigl\| \sum_{\tau_{\ell}\in I}\1_{[0,\delta_{\tau_{\ell}}]}(s)a_{\ell}e_{\tau_{\ell}}\, \bigr\|_{p_0}\ls A^{\sigma(p_0)}\delta^\frac12(\sum_{\tau_\ell\in I}|e_{\tau_\ell}(x)|^2)^\frac12\ls  A^{\sigma(p_0)}\delta^\frac12A^{\frac{n-1}2}\delta^\frac12,\]
\[\|m(0, \, \cdot\, )\|_{p_0}\ls B^{\sigma(p_0)}|I_1|^\frac12(\sum_{|k|\in I_1}|b_{k,\tau_\ell}|^2)^\frac12\ls B^{\sigma(p_0)}|I_1|^\frac12B^{\frac{n-1}2}|I_1|^\frac12C,\]
and similarly
\[\int_0^\delta \bigl\| \tfrac\partial{\partial s}
m(s, \, \cdot\, )\bigr\|_{p_0}\ls B^{\sigma(p_0)}|I_1|^\frac12B^{\frac{n-1}2}|I_1|^\frac12C, \]
where $|I_1|$ is the length of the interval $I_1$.
We will apply Lemma \ref{delta1} with these types of estimates many times. 

The following useful lemma on the smooth function of $P^0$, which will be used several times later. Indeed, only the case $n+\mu<0$ is used in the paper, though we discuss three cases here for completeness.
\begin{lemma}\label{pdo}
	Let $\mu\in \mathbb{R}$, and $m\in C^\infty(\mathbb{R})$ belong to the symbol class $S^\mu$, that is, assume that
	\[|\partial_t^{\alpha} m(t)| \leq C_{\alpha}(1+|t|)^{\mu-\alpha}, \quad \forall \alpha.\]
	Then $m(P^0)$ is a pseudo-differential operator of order $\mu$. Moreover, if $R\ge1$, then the kernel of the operator $m(P^0/R)$ satisfies for all $N\in\mathbb{N}$
	\begin{equation}\label{withla}
		|m(P^0/R)(x,y)|\\\le  \begin{cases}CR^n\big(R d_g(x,y)\big)^{-n-\mu}\big(1+R d_g(x,y)\big )^{-N},\,\quad\quad\,\,\,\, \ n+\mu>0\\
			CR^n\log(2+(Rd_g(x,y))^{-1})\big(1+R d_g(x,y)\big)^{-N},\,\, \ n+\mu=0\\
			CR^n(1+Rd_g(x,y))^{-N},\ \ \ \ \ \quad\quad\quad\quad\quad\quad\quad\quad\ \  n+\mu<0.
		\end{cases}
	\end{equation}
\end{lemma}
We mean that the estimates hold near the diagonal (so that $d_g(x,y)$ is well-defined) and that   outside the neighborhood of the diagonal we have $|m(P^0/R)(x,y)|\ls R^{-N},\ \forall N$. See \cite[Lemma 4]{hswz}, \cite[Theorem 4.3.1]{fio}, \cite[Prop.1 on page 241]{steinbook} for the proof. Roughly speaking, modulo lower order terms, $m(P^0/R)(x, y)$  equals
\[(2\pi)^{-n}\int_{\mathbb{R}^n}m(|\xi|/R)e^{i d_g(x,y)\xi_1}d\xi\]
near the diagonal, which satisfies the bounds in \eqref{withla}, while outside of a fixed neighborhood of the diagonal  $m(P^0/R)(x, y)$ is $O(R^{-N})$.

\subsection{Outline for the proof of Proposition \ref{prop2}}

We will consider $\tau_\ell\le2\la$ (low-frequency terms) and $\tau_\ell>2\la$ (high-frequency terms) seperately, and prove the ``low-frequency estimates'' 
	\begin{equation}\label{Low}
	\Big|\sum_{j,k}\sum_{\tau_\ell\le 2\la}
	\int a_{jk\tau_\ell}\,e_j^0(x)V_{jk}\overline{e_k^0(z)}V(z)e_{\tau_\ell}(z) \overline{e_{\tau_\ell}(x)}dz\Big|
	\ls \la^{n-\frac32\eta}(\log\la)^3
\end{equation}
and the ``high-frequency estimates'' 
	\begin{equation}\label{high}
	\Big|\sum_{j,k}\sum_{\tau_\ell> 2\la}
	\int a_{jk\tau_\ell}\,e_j^0(x)V_{jk}\overline{e_k^0(z)}V(z)e_{\tau_\ell}(z) \overline{e_{\tau_\ell}(x)}dz\Big|
	\ls \la^{n-\frac32\eta}\log\la.
\end{equation}
These estimates immediately imply Proposition \ref{prop2}. 

In the next two sections, we will prove these two estimates. We need to split the sum over three different types of frequencies ($|j|, |k|, \tau_\ell$) in a reasonable way, 
so that within each case, the size of $a_{jk\tau_\ell}$ is essentially fixed, and we also have the explicit information on the spectral intervals that $|j|, |k|$ and $\tau_\ell$ belong to. This allows us to analyze each case efficiently by applying  Lemma \ref{sogge888} (spectral projection bounds) and Corollary \ref{roub} (rough eigenfunction bounds) plus one of the following three key lemmas:  \textbf{Lemma \ref{delta1}} (H\"older+Minkowski),  \textbf{Lemma \ref{pdo}} (kernel estimates for pseudo-differential operators) and \textbf{Lemma \ref{heatk}} (heat kernel bounds). 
For the reader's convenience, we list the key lemma for each case here.

\textbf{Section 4: Low-frequency estimates}
\begin{enumerate}[$\bullet$]
	\item \textbf{Case 1: Lemma \ref{delta1}}
 
 In this case, we have $\tau_\ell$, $|j|$, $|k|\ls \la$. Thus, after further dividing into sub-cases depending on the sizes of $\tau-|k|$ and $\tau-|j|$ so that within the subcases the size of $a_{jk\tau_\ell}$ is essentially fixed,  this allows us to use Lemma 3.2 and the $L^2 \rightarrow L^p$ type estimates (Lemma \ref{sogge888}+ Corollary \ref{roub}) to get the desired bounds.
	\item \textbf{Case 2: Lemma \ref{delta1}}

  In this case, we have $\tau_\ell$, $|k|\ls \la$, but $|j|$ may be unbounded. However, the coefficients $V_{jk}$ and $a_{jk\tau_\ell}$ decay for large $|j|$, so we can use the strategy in Case 1.  
  
	\item \textbf{Case 3: Lemma \ref{delta1}}
 
 In this case,  we have $\tau_\ell$, $|j|\ls \la $, but $|k|$ may be unbounded. The strategy is almost the same as Case 2.
 
	\item \textbf{Case 4: Lemma \ref{pdo}}
 
In this case, we have $\tau_\ell\ls \la$, but both $|k|, |j|$ may be unbounded. When both $|k|, |j|$ are large and $|k|\approx |j|$, the strategy in Case 2 and 3 may not work since $V_{jk}$
 does not  decay. So in this case, we shall use Lemma 3.3, which involves the kernel estimates of  pseudo-differential operators to treat the sum in $j, k$.

\end{enumerate}

\textbf{Section 5: High-frequency estimates}

In the high-frequency case, we mainly divide our discussion into the following three cases:
$\tau_\ell\approx |k|$, $\tau_\ell\ls |k|$, and $\tau_\ell\gs |k|$. The first two cases can be handled in the same way as the low-frequency case by using the decay properties of $h(\tau)$ for large $\tau$ as in \eqref{hda}. More explicitly, 
\begin{enumerate}[$\bullet$]
	\item \textbf{Case 1: Lemma \ref{delta1}}
	
	In this case, we have $\tau_\ell, |k|>1.5\la$, so the coefficients $a_{jk\tau_\ell}$ rapidly decay for large $|j|$. This allows us to use Lemma \ref{delta1}.

	\item \textbf{Case 2.1: Lemma \ref{delta1}}
	
	In this case, we have $|k|>1.1\tau_\ell>2.2\la$ and $|j|\ls \tau_\ell$. The coefficients $a_{jk\tau_\ell}$ rapidly decay for large $|j|$. So we can still apply Lemma \ref{delta1}.

	\item \textbf{Case 2.2: Lemma \ref{pdo}}
	
	In this case, we have $|j|,|k|\gs\tau_\ell$, so $|k|, |j|$ may be unbounded. When both $|k|, |j|$ are large and $|k|\approx |j|$, the strategy above may not work since $V_{jk}$ does not decay. So in this case, we shall use Lemma 3.3, which involves the kernel estimates of  pseudo-differential operators to treat the sum in $j, k$.
\end{enumerate}
 The third case is more complicated. We first handle the terms involving $a_{jk\tau_\ell}^0$ by using the relation 
 \[\frac1{|k|^2-\tau_\ell^2}=\int_0^\infty e^{-t(|k|^2-\tau_\ell^2)}dt\]
 to split the sum into three subcases (see \eqref{hde}). This relation allows us to apply the heat kernel bound in Lemma \ref{heatk}.
\begin{enumerate}[$\bullet$]
	\item \textbf{Case 3.1.1: Lemma \ref{heatk}}
	\item \textbf{Case 3.1.2: Lemma \ref{delta1}}
	\item \textbf{Case 3.1.3: Lemma \ref{delta1}}
\end{enumerate}
And then we handle the terms involving $a_{jk\tau_\ell}^1$. We write for $N=1,2,...$,
$$
	\frac{1}{\tau_\ell^2-|j|^2}=\tau_\ell^{-2}
	+\tau_\ell^{-2}\bigl(|j|/\tau_\ell\bigr)^2+
	\cdots + \tau_\ell^{-2}\bigl(|j|/\tau_\ell\bigr)^{2N-2}
	\\
	+(|j|/\tau_\ell)^{2N}\frac{1}{\tau_\ell^2-|j|^2}
$$
and similarly
$$
	\frac{1}{\tau_\ell^2-|k|^2}=\tau_\ell^{-2}
	+\tau_\ell^{-2}\bigl(|k|/\tau_\ell\bigr)^2+
	\cdots + \tau_\ell^{-2}\bigl(|k|/\tau_\ell\bigr)^{2N-2}
	\\
	+(|k|/\tau_\ell)^{2N}\frac{1}{\tau_\ell^2-|k|^2}.
$$
Then we split their product into three parts
\begin{align*}
	\frac{1}{(\tau_\ell^2-|j|^2)(\tau_\ell^2-|k|^2)}=\frac{(|j|/\tau_\ell)^{2N}}{(\tau_\ell^2-|j|^2)(\tau_\ell^2-|k|^2)}+\sum_{\mu=0}^{N-1}\frac{|j|^{2\mu}(|k|/\tau_\ell)^{2N}}{\tau_\ell^{2+2\mu}(\tau_\ell^2-|k|^2)}+\sum_{\mu_1=0}^{N-1}\sum_{\mu_2=0}^{N-1}\frac{|j|^{2\mu_2}|k|^{2\mu_1}}{\tau_\ell^{2(\mu_1+\mu_2+2)}}
\end{align*}	
The terms involving the first two parts are handled in the following two subcases
\begin{enumerate}[$\bullet$]
	\item \textbf{Case 3.2.1: Lemma \ref{delta1}}
	\item \textbf{Case 3.2.2: Lemma \ref{delta1}}
\end{enumerate}
We handle the third part by using the relation
$$\tau_\ell^{-2(\mu_1+\mu_2+2)} =\int_0^\infty t^{\mu_1+\mu_2+1}e^{-t\tau_k^2}dt$$ to split the sum into three subcases Case 3.2.3a-c (see \eqref{hde1}). This relation allows us to use the heat kernel bound in Lemma \ref{heatk}.
\begin{enumerate}[$\bullet$]
	\item \textbf{Case 3.2.3a: Lemma \ref{heatk}}
	\item \textbf{Case 3.2.3b: Lemma \ref{delta1}}
	\item \textbf{Case 3.2.3c: Lemma \ref{delta1}}
\end{enumerate}
One may observe that most of the cases are handled by Lemma \ref{delta1}. These cases are relatively straightforward, since we just need to estimate the coefficients like $a_{jk\tau_\ell}$ and $b_{k,\tau_\ell}$, and then plug the bounds into Lemma \ref{delta1}.
Moreover, Case 4 and Case 2.2 are very similar and handled by Lemma \ref{pdo}. In these two cases,  $|j|$ and $|k|$ are much larger than $\tau_\ell$, so we split 
	\[a^1_{jk\tau_\ell}=\frac{h(\tau_\ell)}{(\tau_\ell^2-|j|^2)(\tau_\ell^2-|k|^2)}-\frac{h(|j|)}{(\tau_\ell^2-|j|^2)(\tau_\ell^2-|k|^2)}\] and sum over $k$ and $j$ seperately to obtain pseudo-differential operators with kernels like
	\[m(P^0)(x,y)=\sum_{j}m(|j|)e_j^0(x)\overline{e_j^0(y)},\ \ m\in C^\infty(\mathbb{R}).\]And the way to deal with $a^0_{jk\tau_\ell}$ is similar. Furthermore, Case 3.1.1 and Case 3.2.3a are  handled by Lemma \ref{heatk}. These correspond to the cases where $\tau_\ell$ is large, where, if we proceed as before by applying Lemma \ref{delta1}, we would get bounds that are much bigger than desired.
 Instead, in these two cases, we shall use the following relation between the resolvent kernel and the heat kernel
  \[\sum_{\tau_\ell>2\la}\tau_\ell^{-2}e_{\tau_\ell}(z)\overline{e_{\tau_\ell}(x)}=\sum_{\tau_\ell>2\la}\int_0^\infty e^{-t\tau_\ell^2}e_{\tau_\ell}(z)\overline{e_{\tau_\ell}(x)} \,dt.\]
 And we control the ```truncated heat kernel'' by the heat kernel and a finite sum:
\[\sum_{\tau_\ell>2\la}e^{-t\tau_\ell^2}e_{\tau_\ell}(z)\overline{e_{\tau_\ell}(x)}=e^{-tH_V}(z,x)-\sum_{\tau_\ell\le2\la}e^{-t\tau_\ell^2}e_{\tau_\ell}(z)\overline{e_{\tau_\ell}(x)},\]
where Lemma \ref{heatk} gives us good bounds for the first term, and the second term, which is a finite sum,  can be easily handled by H\"older and Sobolev inequalities.

	\section{Low-frequency estimates}
	
	In this section, we shall handle the low frequencies  $1\le\tau_\ell\le2\la$. Our goal is to show that 
	\begin{equation}\label{L.1}
		\Big|\sum_{j,k}\sum_{\tau_\ell\le 2\la}
	\int a_{jk\tau_\ell}\,e_j^0(x)V_{jk}\overline{e_k^0(z)}V(z)e_{\tau_\ell}(z) \overline{e_{\tau_\ell}(x)}dz\Big|
		\ls \la^{n-\frac32\eta}(\log\la)^3,
	\end{equation}
	for $a_{jk\tau_\ell}$ defined as in \eqref{2.2}.
Recall that in \eqref{2.3}, we can insert smooth cut-off functions and rewrite the sum as
\begin{equation}\label{lowmain}
	\begin{aligned}
\sum_{\ell_1\ge\ell_0}\sum_{\ell_2\ge\ell_0}\sum_{j, k}\sum_ {\tau_\ell\le2\la}\beta_{\ell_1}(\tau_\ell-|k|)\beta_{\ell_2}(\tau_\ell-|j|)\cdot \int a_{jk\tau_\ell}\,e_j^0(x)V_{jk}\overline{e_k^0(z)}V(z)e_{\tau_\ell}(z) \overline{e_{\tau_\ell}(x)}dz.
	\end{aligned}
\end{equation}
	
	We shall divide our discussion into the following four cases:
\begin{enumerate}[$\bullet$]
	\item \textbf{Case 1}: $2^{\ell_1}\le 100\la,\ 2^{\ell_2}\le 2^{\ell_1+2}$ 
	\item \textbf{Case 2}: $2^{\ell_1}\le100\la,2^{\ell_2}>2^{\ell_1+2}$
	\item \textbf{Case 3}: $2^{\ell_1}>100\la,2^{\ell_2}\le10\la$
	\item \textbf{Case 4}: $2^{\ell_1}>100\la,2^{\ell_2}>10\la$.
	\end{enumerate}
Recall that the key lemma to handle the first three cases is Lemma \ref{delta1} (H\"older+Minkowski), while the fourth case is handled by Lemma \ref{pdo} (kernel estimates for pseudo-differential operators). Meanwhile, Lemma \ref{sogge888} (spectral projection bounds) and Corollary \ref{roub} (rough eigenfunction bounds) are basic tools applied to all the cases.

	\subsection*{\textbf{Case 1}} $2^{\ell_1}\le 100\la,\ 2^{\ell_2}\le 2^{\ell_1+2}$.
	
	It suffices to estimate
	\begin{equation}\label{eq1.2}
		\begin{aligned}\sum_{j, k, \tau_\ell}\beta_{\ell_1}&(\tau_\ell-|k|)\beta_{\ell_2}(\tau_\ell-|j|) \int a_{jk\tau_\ell}\,e_j^0(x)V_{jk}\overline{e_k^0(z)}V(z)e_{\tau_\ell}(z) \overline{e_{\tau_\ell}(x)}dz
		\end{aligned}
	\end{equation}
	for each integer $\ell_1$ and $\ell_2$ satisfying $2^{\ell_1}\le 100\la,\ 2^{\ell_2}\le 2^{\ell_1+2}$.
	
	Let $\tau_\ell\approx 2^s\le 2\la$, $s=0,1,2,...$. If $2^s>10\cdot 2^{\ell_0}$, we divide the range $[2^{s-1},2^s]$ into intervals of length $2^{\ell_0}$. Then the interval $[1,2\la]$ is divided into a collection of intervals $\{ I_{s,\nu}\}$, $2^s\le 2\la$, $\nu=0,1,2,...$. 
	If $2^s>10\cdot 2^{\ell_0}$, then the length $|I_{s,\nu}|\approx 2^{\ell_0}$ and $\nu\ls 2^s2^{-\ell_0}$. Otherwise, we have $|I_{s,\nu}|\approx 2^s\ls 2^{\ell_0}$ and $\nu=0$.
	
	We consider two subcases seperately: (i) $\ell_1>\ell_0$ (ii) $\ell_1=\ell_0.$ When $\ell_1>\ell_0$, we have $|\tau_\ell-|k||\approx 2^{\ell_1}$ so we just need to estimate $a^0_{jk\tau_\ell}$ and $a^1_{jk\tau_\ell}$ seperately. However, $|\tau_\ell-|k||$ can be zero when $\ell_1=\ell_0$, so we apply the mean value theorem and treat $a_{jk\tau_\ell}$ as a whole in this case.

	\subsection*{(i)} $\ell_1>\ell_0$.
	
	In this case, $|\tau_\ell-|k||\approx 2^{\ell_1}$ and $|\tau_\ell-|j||\ls 2^{\ell_2}$.
	Recall that \begin{align*}
		a_{jk\tau_\ell}=&\frac{h(|k|)-h(|j|)}{(|k|^2-|j|^2)(|k|^2-\tau_\ell^2)}+\frac{h(\tau_\ell)-h(|j|)}{(\tau_\ell^2-|j|^2)(\tau_\ell^2-|k|^2)}\\
		=&a_{jk\tau_\ell}^0+a_{jk\tau_\ell}^1.
	\end{align*} We handle these two terms separately. 
	
	First, we deal with $a_{jk\tau_\ell}^0$. We just need to estimate
	\begin{equation}\label{eq1.3}
		\begin{aligned}\sum_{j, k, \tau_\ell}\beta_{\ell_1}&(\tau_\ell-|k|)\beta_{\ell_2}(\tau_\ell-|j|)\beta_{\ell_3}(|j|-|k|) \int a_{jk\tau_\ell}^0\,e_j^0(x)V_{jk}\overline{e_k^0(z)}V(z)e_{\tau_\ell}(z) \overline{e_{\tau_\ell}(x)}dz
		\end{aligned}
	\end{equation}
	for each integer $\ell_3$ satisfying $2^{\ell_0}\le2^{\ell_3}\ls \la$, as there are  only $\log\la$ terms in the sum over $\ell_3$. 
	If $|k|+|j|> \la/2$ and $|k|+\tau_\ell> \la/20$, then \[|a_{jk\tau_\ell}^0|\ls \la^{-2}2^{-\ell_1}2^{-\ell_3}.\]
	If $|k|+|j|> \la/2$ and $|k|+\tau_\ell\le \la/20$, then $2^{\ell_1}\approx 2^{\ell_2}\approx 2^{\ell_3}\approx \la$,  and
	\[|a_{jk\tau_\ell}^0|\ls \la^{-4}.\]
	If $|k|+|j| \le\la/2$, then \[|a_{jk\tau_\ell}^0|\ls 2^{-\ell_1}2^{-\ell_3}\la^{-N},\ \forall N.\]
	So we always have
	\[|a_{jk\tau_\ell}^0|\ls \la^{-2}2^{-\ell_1}2^{-\ell_3}\]
	and similarly
	\[|\partial_{\tau_\ell}a_{jk\tau_\ell}^0|\ls \la^{-2}2^{-2\ell_1}2^{-\ell_3}.\]

	In the following, we apply Lemma \ref{delta1}. Fix an interval $I_{s,\nu}$ with left endpoint $\xi$. Let $\tau_\ell \in I_{s,\nu}$. Since $|\tau_\ell-|k||\approx 2^{\ell_1}$, the range of $|k|$ is an interval $I_{s,\nu}'$ of length $\approx 2^{\ell_1}$. Let 
	\[b_{k,\tau_\ell}=\sum_j \beta_{\ell_1}(\tau_\ell-|k|)\beta_{\ell_2}(\tau_\ell-|j|)\beta_{\ell_3}(|j|-|k|) a_{jk\tau_\ell}^0\,e_j^0(x)V_{jk}.\]
	Note that for fixed $k$, the number of $j\in \mathbb{Z}^n$ satisfying $1+|j-k|\approx 2^m$ and $\big||j|-|k|\big|\ls 2^{\ell_3}$ is bounded by $2^{(n-1)m}2^{\ell_3}$. Here $m\in\mathbb{N}$ and $\ell_3\ge\ell_0$. Recall the bound \eqref{vjkb} on $V_{jk}$. Then \begin{align*}|b_{k,\tau_\ell}|&\ls \sum_{2^m\ls \la}\la^{-2}2^{-\ell_1}2^{-\ell_3}2^{(-n+2-\eta)m}2^{(n-1)m}2^{\ell_3}\\
		&\ls \la^{-1-\eta}2^{-\ell_1}
	\end{align*}
	and similarly
	\[|\partial_{\tau_\ell}b_{k,\tau_\ell}|\ls \la^{-1-\eta}2^{-\ell_1}2^{-\ell_0}.\]
By using the spectral projection bounds in Lemma \ref{sogge888}, we obtain
	\begin{align*}
		\|\sum_{|k|\in I^\prime_{s,\nu}} b_{k,\xi}e_k^0(\cdot)\|&_{L^{p_0}(M)}\\ &\ls 2^{\ell_1/2} \la^{\sigma(p_0)} \|\sum_{|k|\in I^\prime_{ s,\nu}} b_{k,\xi}e_k^0(\cdot)\|_{L^{2}(M)} \\
		&\ls 2^{\ell_1/2} \la^{\sigma(p_0)}2^{\ell_1/2} \la^{\frac{n-1}{2}}\cdot \la^{-1-\eta}2^{-\ell_1}\\
		&\ls \la^{\frac n2-1-\frac54\eta}.
	\end{align*}
	And similarly for $\tau\in I_{s,\nu}$
	\[\|\sum_{|k|\in I^\prime_{s,\nu}}\partial_{\tau}  b_{k,\tau}e_k^0(\cdot)\|_{L^{p_0}(M)}\ls 2^{-\ell_0}\la^{\frac n2-1-\frac54\eta}.
	\]
	By applying Lemma~\ref{delta1} with $\delta=2^{\ell_0}$, we have 
	\begin{equation}\label{eq1.4}
		\begin{aligned}
			\sum_\nu\Big|\int \sum_{\tau_\ell\in I_{s,\nu}}\sum_{k} &b_{k,\tau_\ell} \overline{e_k^0(z)}V(z)e_{\tau_\ell}(z) \overline{e_{\tau_\ell}(x)}dz \Big|\\
			\ls & \|V\|_{L^{\frac{n}{2-\eta}}(M)}\la^{\frac n2-1-\frac54\eta}\cdot 2^{\ell_0/2}\la^{\sigma(p_0)}\sum_\nu ( \sum_{\tau_\ell\in I_{s,\nu}}|{e_{\tau_\ell}(x)}|^2 )^{1/2}  \\
			\ls &\|V\|_{L^{\frac{n}{2-\eta}}(M)} \la^{\frac {n-1}2-\frac32\eta}\cdot 2^{\ell_0/2} (2^{ns/2}+2^{\frac{n-1}2s}2^{\ell_0/2}\cdot 2^{s}2^{-\ell_0}).
		\end{aligned}
	\end{equation}
	Here in the last line we used the fact that 
\begin{equation}\label{etaubound}
\sum_\nu ( \sum_{\tau_\ell\in I_{s,\nu}}|{e_{\tau_\ell}(x)}|^2 )^{1/2} \ls
    \begin{cases}
2^{ns/2}, \,\,\,\text{if}\,\, 2^s\ls 2^{\ell_0}\\
2^{\frac{n-1}2s}2^{\ell_0/2}\cdot 2^{s}2^{-\ell_0}\,\,\,\text{if}\,\, 2^s\gs 2^{\ell_0},
    \end{cases}
\end{equation}
which is a consequence of Lemma~\ref{sogge888} at $p=\infty$, along with the fact that $\# \{\nu\}\ls 1$, if $ 2^s\ls 2^{\ell_0}$ and $\# \{\nu\}\ls 2^{s}2^{-\ell_0}$, if $ 2^s\gs 2^{\ell_0}$.

Since $2^{\ell_1}\le 100\la$, $2^{\ell_2}\le 2^{\ell_1+2}$ and $2^s\le 2\la$,
 		summing over all possible choices of $s,\ell_1, \ell_2$, we get the desired bound $\la^{n-\frac32\eta}(\log\la)^3$. 

  \vspace{2mm}
	Next, we deal with  $a_{jk\tau_\ell}^1$. We just need to estimate
	\begin{equation}\label{eq1.3'}
		\begin{aligned}\sum_{j, k, \tau_\ell}\beta_{\ell_1}&(\tau_\ell-|k|)\beta_{\ell_2}(\tau_\ell-|j|) \int a_{jk\tau_\ell}^1\,e_j^0(x)V_{jk}\overline{e_k^0(z)}V(z)e_{\tau_\ell}(z) \overline{e_{\tau_\ell}(x)}dz
		\end{aligned}
	\end{equation}
	If $\tau_\ell+|j|>\la/2$ and $\tau_\ell+|k|>\la/20$, then
	\[|a_{jk\tau_\ell}^1|\ls \la^{-2}2^{-\ell_1}2^{-\ell_2}.\]
	If $\tau_\ell+|j|>\la/2$ and $\tau_\ell+|k|\le\la/20$, then $2^{\ell_1}\approx 2^{\ell_2}\approx \la$, and
	\[|a_{jk\tau_\ell}^1|\ls \la^{-4}.\]
	If  $\tau_\ell+|j|\le\la/2$, then
	\[|a_{jk\tau_\ell}^1|\ls 2^{-\ell_1}2^{-\ell_2}\la^{-N},\ \forall N.\]
	So we always have
	\[|a_{jk\tau_\ell}^1|\ls \la^{-2}2^{-\ell_1}2^{-\ell_2}.\]
	and similarly
	\[|\partial_{\tau_\ell}a_{jk\tau_\ell}^1|\ls \la^{-2}2^{-\ell_1}2^{-\ell_2}2^{-\ell_0}.\]
	Let 
	\[b_{k,\tau_\ell}=\sum_j \beta_{\ell_1}(\tau_\ell-|k|)\beta_{\ell_2}(\tau_\ell-|j|) a_{jk\tau_\ell}^1\,e_j^0(x)V_{jk}.\]
	Note that  for fixed $k$ and $\tau_\ell$, the number of $j\in \mathbb{Z}^n$ satisfying $1+|j-k|\approx 2^m$ and $\big|\tau_\ell-|j|\big|\ls 2^{\ell_2}$ is bounded by $2^{(n-1)m}2^{\ell_2}$. Here $m\in\mathbb{N}$ and $\ell_2\ge\ell_0$. Recall the bound \eqref{vjkb} on $V_{jk}$. Then \begin{align*}|b_{k,\tau_\ell}|&\ls \sum_{2^m\ls \la}\la^{-2}2^{-\ell_1}2^{-\ell_2}2^{(-n+2-\eta)m}2^{(n-1)m}2^{\ell_2}\\
		&\ls \la^{-1-\eta}2^{-\ell_1}
	\end{align*}
	and similarly
	\[|\partial_{\tau_\ell}b_{k,\tau_\ell}|\ls \la^{-1-\eta}2^{-\ell_1}2^{-\ell_0}.\]
	By using the spectral projection bounds in Lemma \ref{sogge888}, we get
	\begin{align*}
		\|\sum_{|k|\in I^\prime_{s,\nu}} b_{k,\xi}e_k^0(\cdot)\|&_{L^{p_0}(M)}\\ &\ls 2^{\ell_1/2} \la^{\sigma(p_0)} \|\sum_{|k|\in I^\prime_{s,\nu}} b_{k,\xi}e_k^0(\cdot)\|_{L^{2}(M)} \\
		&\ls 2^{\ell_1/2} \la^{\sigma(p_0)}2^{\ell_1/2} \la^{\frac{n-1}{2}} \la^{-1-\eta}2^{-\ell_1}\\
		&\ls \la^{\frac{n}{2}-1-\frac54\eta}
	\end{align*}
	and similarly for $\tau\in I_{s,\nu}$
	\[\|\sum_{|k|\in I^\prime_{s,\nu}}\partial_{\tau}  b_{k,\tau}e_k^0(\cdot)\|_{L^{p_0}(M)}\ls 2^{-\ell_0}\la^{\frac{n}{2}-1-\frac54\eta}.
	\]
	By applying Lemma~\ref{delta1} with $\delta=2^{\ell_0}$, we have
	\begin{equation}\label{eq1.4'}
		\begin{aligned}
			\sum_\nu|\int \sum_{\tau_\ell\in I_{s,\nu}}\sum_{k} &b_{k,\tau_\ell} \overline{e_k^0(z)}V(z)e_{\tau_\ell}(z) \overline{e_{\tau_\ell}(x)}dz |\\
			\ls & \|V\|_{L^{\frac{n}{2-\eta}}(M)}\la^{\frac{n}{2}-1-\frac54\eta}\cdot 2^{\ell_0/2}\la^{\sigma(p_0)} \sum_\nu( \sum_{\tau_\ell\in I_{s,\nu}}|{e_{\tau_\ell}(x)}|^2 )^{1/2}  \\
			\ls &\|V\|_{L^{\frac{n}{2-\eta}}(M)} \la^{\frac{n-1}{2}-\frac32\eta}2^{\ell_0/2} (2^{ns/2}+2^{\frac{n-1}2s}2^{\ell_0/2}\cdot 2^s2^{-\ell_0}),
		\end{aligned}
	\end{equation}
 where in the last line we used \eqref{etaubound}.
	Summing over $s,\ell_1,\ell_2$, we get the desired bound $\la^{n-\frac32\eta}(\log\la)^2$. 
	
	\subsection*{(ii)} $\ell_1=\ell_0$.
	
	In this case, $|\tau_\ell-|k||\ls 2^{\ell_0}$ and $|\tau_\ell-|j||\ls 2^{\ell_0}$. 
	We claim that 
	\begin{equation}\label{meanvalue 1}
	   | a_{jk\tau_\ell}|\ls  \la^{-2}2^{-2\ell_0}.
	\end{equation}
Indeed, by the mean value theorem,
\begin{equation}\label{meanvalue 2}
    	a_{jk\tau_\ell}=\frac{f(|k|)-f(\tau_\ell)}{|k|^2-\tau_\ell^2}=\frac{1}{|k|+\tau_\ell} f'(\tau)
\end{equation}
	for some $\tau$ between $\tau_\ell$ and $|k|$. Recall that 
	\begin{equation}\label{meanvalue 3}
	    f(\tau)=\frac{h(\tau)-h(|j|)}{\tau^2-|j|^2}.
	\end{equation}
By using the mean value theorem again, we have
	\begin{equation}\label{meanvalue 4}
	\begin{aligned}
	        f'(\tau)&=-\frac{1}{(\tau+|j|)^2}\frac{h(\tau)-h(|j|)}{\tau-|j|} +\frac{1}{\tau+|j|}\frac{h'(\tau)(\tau-|j|)-(h(\tau)-h(|j|))}{(\tau-|j|)^2} \\
	        &= -\frac{1}{(\tau+|j|)^2} h'(\tau_1)+\frac{1}{\tau+|j|} \frac{h''(\tau_2)}2
	\end{aligned}
	\end{equation}
for some $\tau_1,\tau_2$ between $\tau$ and $|j|$.	

 If $1\le \tau_\ell\le \la/4$, then we have $|k|, |j|\le \la/2$, since $2^{\ell_0}\approx \la^{1-\eta-\eps}$. So by \eqref{hdao}, \eqref{meanvalue 2} and \eqref{meanvalue 4}, we have 
 \[| a_{jk\tau_\ell}|\ls |f'(\tau)|\ls |h'(\tau_1)|+|h'(\tau_2)|\ls \la^{-N},\ \forall N. \]
If $\la/4\le \tau_\ell\le 2\la $,  then $|k|, |j|\approx \la$, since $2^{\ell_0}\approx \la^{1-\eta-\eps}$. So by \eqref{hdao}, \eqref{meanvalue 2} and \eqref{meanvalue 4}, we have 
\[| a_{jk\tau_\ell}|\ls \la^{-1}|f'(\tau)|\ls \la^{-3}|h'(\tau_1)|+\la^{-2}|h'(\tau_2)|\ls \la^{-2}2^{-2\ell_0}. \] This finishes the proof of \eqref{meanvalue 1}.

	By using a similar argument as above, it is not hard to see that 
	\begin{equation}
	    |\partial_{\tau_\ell}a_{jk\tau_\ell}|\ls \la^{-2}2^{-3\ell_0}.
	\end{equation}

	Let 
	\[b_{k,\tau_\ell}=\sum_j \beta_{\ell_1}(\tau_\ell-|k|)\beta_{\ell_2}(\tau_\ell-|j|) a_{jk\tau_\ell}\,e_j^0(x)V_{jk}.\]
	Note that  for fixed $k$ and $\tau_\ell$, the number of $j\in \mathbb{Z}^n$ satisfying $1+|j-k|\approx 2^m$ and $\big|\tau_\ell-|j|\big|\ls 2^{\ell_0}$ is bounded by $2^{(n-1)m}2^{\ell_0}$. Here $m\in\mathbb{N}$. Recall the bound \eqref{vjkb} on $V_{jk}$. Then \begin{align*}|b_{k,\tau_\ell}|&\ls \sum_{2^m\ls \la}\la^{-2}2^{-2\ell_0}\cdot 2^{(n-1)m}2^{\ell_0}2^{(-n+2-\eta)m}\\
		&\ls \la^{-1-\eta}2^{-\ell_0}
	\end{align*}
	and similarly
	\[|\partial_{\tau_\ell}b_{k,\tau_\ell}|\ls \la^{-1-\eta}2^{-2\ell_0}.\]
	Thus by using the spectral projection bounds in Lemma \ref{sogge888}, we have
	\begin{align*}
		\|\sum_{|k|\in I^\prime_{s,\nu}} b_{k,\xi}e_k^0(\cdot)\|&_{L^{p_0}(M)}\\ &\ls 2^{\ell_0/2} \la^{\sigma(p_0)} \|\sum_{|k|\in I^\prime_{s,\nu}} b_{k,\xi}e_k^0(\cdot)\|_{L^{2}(M)} \\
		&\ls 2^{\ell_0/2} \la^{\sigma(p_0)}2^{\ell_0/2} \la^{\frac{n-1}{2}}\cdot \la^{-1-\eta}2^{-\ell_0}\\
		&\ls \la^{\frac{n}{2}-1-\frac54\eta}
	\end{align*}
	and similarly for $\tau\in I_{s,\nu}$
	\[\|\sum_{|k|\in I^\prime_{s,\nu}}\partial_{\tau}  b_{k,\tau}e_k^0(\cdot)\|_{L^{p_0}(M)}\ls 2^{-\ell_0}\la^{\frac{n}{2}-1-\frac54\eta}.
	\]
	By using Lemma~\ref{delta1} and repeating the argument in (i), we get
	\begin{align*}
		\sum_\nu|\int \sum_{\tau_\ell\in I_{s,\nu}}\sum_{k} &b_{k,\tau_\ell} \overline{e_k^0(z)}V(z)e_{\tau_\ell}(z) \overline{e_{\tau_\ell}(x)}dz |\\
		\ls & \|V\|_{L^{\frac{n}{2-\eta}}(M)}\la^{\frac{n}{2}-1-\frac54\eta}\cdot 2^{\ell_0/2}\la^{\sigma(p_0)} \sum_\nu( \sum_{\tau_\ell\in I_{ s,\nu}}|{e_{\tau_\ell}(x)}|^2 )^{1/2}  \\
		\ls &\|V\|_{L^{\frac{n}{2-\eta}}(M)} \la^{\frac{n-1}{2}-\frac32\eta}\cdot 2^{\ell_0/2} (2^{ns/2}+2^{\frac{n-1}2s}2^{\ell_0/2}\cdot 2^s2^{-\ell_0}).
	\end{align*}
	Summing over $\ell_1,\ell_2,s$, we get the desired bound $\la^{n-\frac32\eta}\log\la$.

	\subsection*{\textbf{Case 2}}$2^{\ell_1}\le100\la,2^{\ell_2}>2^{\ell_1+2}$.
	
	It suffices to estimate
	\begin{equation}\label{eq2.2}
		\begin{aligned}\sum_{j, k, \tau_\ell}\beta_{\ell_1}&(\tau_\ell-|k|)\beta_{\ell_2}(\tau_\ell-|j|) \int a_{jk\tau_\ell}\,e_j^0(x)V_{jk}\overline{e_k^0(z)}V(z)e_{\tau_\ell}(z) \overline{e_{\tau_\ell}(x)}dz
		\end{aligned}
	\end{equation}
	for each integer $\ell_1$ and $\ell_2$ satisfying $2^{\ell_1}\le 100\la,\  2^{\ell_2}>2^{\ell_1+2}$. 
	We consider two subcases separately: $2^{\ell_2}\le  1000\la$ or $2^{\ell_2}> 1000\la$.
	
	Let $\tau_\ell\approx 2^s\le2\la$.  We divide the interval $[1,2\la]$ as in Case 1. In this case, we have $||k|-|j||\approx 2^{\ell_2}$, $\tau_\ell+|k|\gs 2^s$.

	First, we deal with  $2^{\ell_2}\le 1000\la$. If $\ell_1=\ell_0$, then by the mean value theorem, if we argue as in the proof of \eqref{meanvalue 1}, it is not hard to see that
	\[|a_{jk\tau_\ell}|\ls \la^{-1}2^{-s}2^{-\ell_2}2^{-\ell_0}.\] Suppose that $\ell_1>\ell_0$. If $|k|+|j|>\la/2$ and $|k|+\tau_\ell>\la/20$, then 
	\[|a_{jk\tau_\ell}|\ls \la^{-1}2^{-s}2^{-\ell_1}2^{-\ell_2}.\]
	If $|k|+|j|>\la/2$ and $|k|+\tau_\ell\le\la/20$, then  $2^{\ell_2}\approx \la$, and
	\[|a_{jk\tau_\ell}|\ls\la^{-2}2^{-s}2^{-\ell_1}.\]
	If $|k|+|j|\le\la/2$, then 
	\[|a_{jk\tau_\ell}|\ls \la^{-2}2^{-\ell_1}2^{-\ell_2}.\]
	So we always have
	\[|a_{jk\tau_\ell}|\ls \la^{-1}2^{-s}2^{-\ell_1}2^{-\ell_2}\]
	and similarly
	\[|\partial_{\tau_\ell}a_{jk\tau_\ell}|\ls \la^{-1}2^{-s}2^{-\ell_1}2^{-\ell_2}(2^{-s}+2^{-\ell_0}).\]
	Let
	\[b_{k,\tau_\ell}=\sum_j \beta_{\ell_1}(\tau_\ell-|k|)\beta_{\ell_2}(\tau_\ell-|j|)
	a_{jk\tau_\ell} \,e_j^0(x)V_{jk},\]
	Note that for fixed $k$ and $\tau_\ell$, the number of $j\in \mathbb{Z}^n$ satisfying $1+|j-k|\approx 2^m\ls \la$ and $\big||j|-|\tau_\ell|\big|\ls 2^{\ell_2}$ is bounded by $2^{(n-1)m}2^{\ell_2}$. Here $m\in\mathbb{N}$ and $\ell_2\ge\ell_0$. Recall the bound \eqref{vjkb} on $V_{jk}$. Then 
	\begin{align*}
		|b_{k,\tau_\ell}|\ls &\sum_{2^m\ls \la} \la^{-1}2^{-s} 2^{-\ell_2}2^{-\ell_1}\la^{(n-1)m}2^{\ell_2}2^{(-n+2-\eta)m}  \\
		\ls& \la^{-\eta}2^{-s}2^{-\ell_1}.
	\end{align*}
	and similarly
	\[
	|\partial_{\tau_\ell} b_{k,\tau_\ell}|\ls \la^{-\eta}2^{-s}2^{-\ell_1}(2^{-s}+2^{-\ell_0}).\]
	
	By using the spectral projection bounds in Lemma \ref{sogge888}, we have
	\begin{align*}
		\|\sum_{|k|\in I^\prime_{s,\nu}} b_{k,\xi}e_k^0(\cdot)\|_{L^{p_0}(M)} &\ls 2^{\ell_1/2} \la^{\sigma(p_0)} \|\sum_{|k|\in I^\prime_{s,\nu}} b_{k,\xi}e_k^0(\cdot)\|_{L^{2}(M)} \\
		&\ls 2^{\ell_1/2} \la^{\sigma(p_0)}2^{\ell_1/2} \la^{\frac{n-1}{2}} \la^{-\eta}2^{-s}2^{-\ell_1}\\
		&\ls \la^{\frac{n}{2}-\frac54\eta} 2^{-s}.
	\end{align*}
	And similarly for $\tau\in I_{s,\nu}$
	\[\|\sum_{|k|\in I^\prime_{s,\nu}}\partial_\tau b_{k,\tau}e_k^0(\cdot)\|_{L^{p_0}(M)}\ls \la^{\frac{n}{2}-\frac54\eta} 2^{-s}(2^{-s}+2^{-\ell_0}).\]
	By applying Lemma~\ref{delta1} with $\delta=2^{\ell_0}$, we have 
	\begin{equation}\label{5.14}
		\begin{aligned}
			\sum_\nu|\int \sum_{ \tau_\ell\in I_{s,\nu}}\sum_k &b_{k,\tau_\ell} \overline{e_k^0(z)}V(z)e_{\tau_\ell}(z) \overline{e_{\tau_\ell}(x)}dz |\\
			\ls & \|V\|_{L^{\frac{n}{2-\eta}}(M)}\la^{\frac{n}{2}-\frac54\eta} 2^{-s}\cdot 2^{\ell_0/2}\la^{\sigma(p_0)}  \sum_\nu( \sum_{\tau_\ell\in I_{s,\nu}}|e_{\tau_\ell}(x)|^2 )^{1/2}  \\
			\ls &\|V\|_{L^{\frac{n}{2-\eta}}(M)} \la^{\frac{n+1}{2}-\frac32\eta} 2^{-s}2^{\ell_0/2} (2^{ns/2}+2^{\frac{n-1}2s}2^{\ell_0/2}\cdot 2^s2^{-\ell_0}),
		\end{aligned}
	\end{equation} 
  where in the last line we used \eqref{etaubound}.
	Summing over $s,\ell_1,\ell_2$,  we get the desired bound $\la^{n-\frac32\eta}(\log\la)^2$.
	
	Next, we deal with $2^{\ell_2}> 1000\la$. In this case, we have $|j-k|\approx |j|\approx  2^{\ell_2}$.
	If $\ell_1=\ell_0$, then by the mean value theorem, if we argue as in the proof of \eqref{meanvalue 1}, one can show that
	\[|a_{jk\tau_\ell}|\ls2^{-s}2^{-2\ell_2}2^{-\ell_0}.\]
	Suppose that $\ell_1>\ell_0$.
	If $|k|+\tau_\ell>\la/2$, we have
	\[|a_{jk\tau_\ell}|\ls \la^{-1}2^{-2\ell_2}2^{-\ell_1}.\]
	If $|k|+\tau_\ell\le\la/2$, we have
	\[|a_{jk\tau_\ell}|\ls 2^{-s}2^{-2\ell_2}2^{-\ell_1}.\]
	So we always have
	\[|a_{jk\tau_\ell}|\ls 2^{-s}2^{-2\ell_2}2^{-\ell_1}\]
	and similarly
	\[|\partial_{\tau_\ell}a_{jk\tau_\ell}|\ls 2^{-s}2^{-2\ell_2}2^{-\ell_1}(2^{-s}+2^{-\ell_0}).\]
	Let
	\[
	b_{k,\tau_\ell}=\sum_j \beta_{\ell_1}(\tau_\ell-|k|)\beta_{\ell_2}(\tau_\ell-|j|)
	a_{jk\tau_\ell} \,e_j^0(x)V_{jk}.\]
	Note that for fixed $k$, the number of $j\in \mathbb{Z}^n$ satisfying $|j-k|\approx 2^{\ell_2}$ is bounded by $2^{n\ell_2}$. Recall the bound \eqref{vjkb} on $V_{jk}$.  Then
	\begin{align*}
		|b_{k,\tau_\ell}|\ls & 2^{n\ell_2}2^{(-n+2-\eta)\ell_2}\cdot 2^{-s} 2^{-2\ell_2}2^{-\ell_1}  \\
		\ls& 2^{-s}2^{-\eta\ell_2}2^{-\ell_1},
	\end{align*}
	and similarly,
	\[
	|\partial_{\tau_\ell} b_{k,\tau_\ell}|\ls 2^{-s}2^{-\eta\ell_2}2^{-\ell_1}(2^{-s}+2^{-\ell_0}).\]
	By using the spectral projection bounds in Lemma \ref{sogge888}, we have
	\begin{align*}
		\|\sum_{|k|\in I^\prime_{s,\nu}} b_{k,\xi}e_k^0(\cdot)\|_{L^{p_0}(M)} &\ls 2^{\ell_1/2} \la^{\sigma(p_0)} \|\sum_{|k|\in I^\prime_{s,\nu}} b_{k,\xi}e_k^0(\cdot)\|_{L^{2}(M)} \\
		&\ls 2^{\ell_1/2} \la^{\sigma(p_0)} 2^{\ell_1/2} \la^{\frac{n-1}{2}} 2^{-s}2^{-\eta\ell_2}2^{-\ell_1}\\
		&\ls \la^{\frac{n}{2}-\frac14\eta}2^{-\eta\ell_2}2^{-s}.
	\end{align*}
	And  similarly for $\tau\in I_{s,\nu}$
	\[
	\|\sum_{|k|\in I^\prime_{s,\nu}}\partial_{\tau}  b_{k,\tau}e_k^0(\cdot)\|_{L^{p_0}(M)}\ls \la^{\frac{n}{2}-\frac14\eta}2^{-\eta\ell_2}2^{-s}(2^{-s}+2^{-\ell_0}) .\]
	By applying Lemma~\ref{delta1} with $\delta=2^{\ell_0}$, we have for $\tau_\ell\in I_{s,\nu}$
	\begin{align*}
		\sum_\nu|\int \sum_{ \tau_\ell\in I_{s,\nu}}\sum_k &b_{k,\tau_\ell}\overline{e_k^0(z)}V(z)e_{\tau_\ell}(z) \overline{e_{\tau_\ell}(x)}dz |\\
		\ls & \|V\|_{L^{\frac{n}{2-\eta}}(M)}\la^{\frac{n}{2}-\frac14\eta}2^{-\eta\ell_2}2^{-s}\cdot 2^{\ell_0/2}\la^{\sigma(p_0)}  \sum_\nu( \sum_{\tau_\ell\in I_{s,\nu}}|e_{\tau_\ell}(x)|^2 )^{1/2}  \\
		\ls &\|V\|_{L^{\frac{n}{2-\eta}}(M)} \la^{\frac{n+1}{2}-\frac12\eta}2^{-\eta\ell_2}2^{-s}\cdot 2^{\ell_0/2} (2^{ns/2}+2^{\frac{n-1}{2}s} 2^{\ell_0/2}\cdot 2^s2^{-\ell_0}),
	\end{align*}
  where in the last line we used \eqref{etaubound}.
	Summing over $s,\ell_1,\ell_2$, we get the desired bound $\la^{n-\frac32\eta}\log\la$.
	
	\subsection*{\textbf{Case 3}}$2^{\ell_1}>100\la,2^{\ell_2}\le10\la$.
	
	It suffices to estimate
	\begin{equation}\label{equ3.2}
		\begin{aligned}\sum_{j, k, \tau_\ell}\beta_{\ell_1}&(\tau_\ell-|k|)\beta_{\ell_2}(\tau_\ell-|j|) \int a_{jk\tau_\ell}\,e_j^0(x)V_{jk}\overline{e_k^0(z)}V(z)e_{\tau_\ell}(z) \overline{e_{\tau_\ell}(x)}dz
		\end{aligned}
	\end{equation}
	for each integer $\ell_1$ and $\ell_2$ satisfying $2^{\ell_1}> 100\la,\  2^{\ell_2}\le10\la$. Let $\tau_\ell\approx 2^s\le 2\la$. We divide the interval $[1,2\la]$ as in Case 1.
	
	If $|j|+\tau_\ell>\la/20$, then
	\[|a_{jk\tau_\ell}|\ls 2^{-4\ell_1}+\la^{-1}2^{-2\ell_1}2^{-\ell_2}.\]
	If $|j|+\tau_\ell\le\la/20$, then
	\[|a_{jk\tau_\ell}|\ls2^{-4\ell_1}+2^{-2\ell_1}2^{-\ell_2}2^{-s}\la^{-N},\ \forall N.\]
	So we always have
	\[|a_{jk\tau_\ell}|\ls \la^{-1}2^{-2\ell_1}2^{-\ell_2}\]
	and similarly
	\[|\partial_{\tau_\ell}a_{jk\tau_\ell}|\ls \la^{-1}2^{-2\ell_1}2^{-\ell_2}(2^{-s}+2^{-\ell_0}).\]
	Let \[
	b_{k,\tau_\ell}=\sum_j \beta_{\ell_1}(\tau_\ell-|k|)\beta_{\ell_2}(\tau_\ell-|j|)
	a_{jk\tau_\ell} \,e_j^0(x)V_{jk}.\]
	Note that for fixed $k$ and $\tau_\ell$, the number of $j\in \mathbb{Z}^n$ satisfying $||j|-\tau_\ell|\ls 2^{\ell_2}$ and $|j|\ls \la$ is bounded by $\la^{n-1}2^{\ell_2}$. Recall the bound \eqref{vjkb} on $V_{jk}$.  Then
	\begin{align*}
		|b_{k,\tau_\ell}|\ls & \la^{n-1}2^{\ell_2}2^{(-n+2-\eta)\ell_1}\cdot \la^{-1} 2^{-\ell_2}2^{-2\ell_1}  \\
		\ls& \la^{n-2}2^{(-n-\eta)\ell_1}
	\end{align*}
	and similarly,
	\[
	|\partial_{\tau_\ell} b_{k,\tau_\ell}|\ls \la^{n-2}2^{(-n-\eta)\ell_1}(2^{-s}+2^{-\ell_0}).\]
By using the spectral projection bounds in Lemma \ref{sogge888}, we have
	\begin{align*}
		\|\sum_{|k|\approx 2^{\ell_1}} b_{k,\xi}e_k^0(\cdot)\|_{L^{p_0}(M)}
		&\ls 2^{\ell_1/2}2^{\sigma(p_0)\ell_1}2^{\frac{n\ell_1}{2}} \cdot \la^{n-2}2^{(-n-\eta)\ell_1}\\
		&\ls \la^{n-2}2^{(-n/2+1-\frac54\eta)\ell_1}
	\end{align*}
	and similarly for $\tau\in I_{s,\nu}$
	\[
	\|\sum_{|k|\approx 2^{\ell_1}}\partial_{\tau}  b_{k,\tau}e_k^0(\cdot)\|_{L^{p_0}(M)}\ls \la^{n-2}2^{(-n/2+1-\frac54\eta)\ell_1}(2^{-s}+2^{-\ell_0}).\]
	By applying Lemma~\ref{delta1} with $\delta=2^{\ell_0}$, we have 
	\begin{align*}
		\sum_\nu|\int &\sum_{\tau_\ell\in I_{s,\nu}}\sum_k b_{k,\tau_\ell} e_k^0(z)V(z)e_{\tau_\ell}(z) \overline{e_{\tau_\ell}(x)}dz |\\
		\ls & \|V\|_{L^{\frac{n}{2-\eta}}(M)} \la^{n-2}2^{(-n/2+1-\frac54\eta)\ell_1} \cdot2^{\ell_0/2}\la^{\sigma(p_0)}\sum_\nu( \sum_{\tau_\ell\in I_{s,\nu}}|e_{\tau_\ell}(x)|^2 )^{1/2}\\
		\ls & \|V\|_{L^{\frac{n}{2-\eta}}(M)} \la^{n-\frac32-\frac14\eta}2^{(-n/2+1-\frac54\eta)\ell_1} \cdot2^{\ell_0/2} (2^{ns/2}+2^{\ell_0/2}2^{\frac{n-1}{2}s}\cdot2^s2^{-\ell_0}),
	\end{align*}
  where in the last line we used \eqref{etaubound}.
	Note that $-n/2+1-\frac54\eta<0$, summing over $\ell_1,s,\ell_2$,  we get the desired bound $\la^{n-\frac32\eta}\log\la.$
	
	\subsection*{\textbf{Case 4}}$2^{\ell_1}>100\la,2^{\ell_2}>10\la$.
	
	It suffices to estimate
	\begin{equation}\label{eq3.2}
		\begin{aligned}\sum_{j, k, \tau_\ell}\beta_{\ell_1}&(\tau_\ell-|k|)\beta_{\ell_2}(\tau_\ell-|j|) \int a_{jk\tau_\ell}\,e_j^0(x)V_{jk}\overline{e_k^0(z)}V(z)e_{\tau_\ell}(z) \overline{e_{\tau_\ell}(x)}dz
		\end{aligned}
	\end{equation}
	for each integer $\ell_1$ and $\ell_2$ satisfying $2^{\ell_1}> 100\la,\  2^{\ell_2}>10\la$. Since $\tau_\ell\le2\la$, we have $|k|,|j|>2\tau_\ell$ in this case. As before we split
	\begin{equation}\label{6.6}\nonumber
		\begin{aligned}
			a_{jk\tau_\ell}=&\frac{h(|k|)-h(|j|)}{(|k|^2-|j|^2)(|k|^2-\tau_\ell^2)}+\frac{h(\tau_\ell)-h(|j|)}{(\tau_\ell^2-|j|^2)(\tau_\ell^2-|k|^2)}\\
			=&a^0_{jk\tau_\ell}+a^1_{jk\tau_\ell}.
		\end{aligned}
	\end{equation}
	We first handle $a^1_{jk\tau_\ell}$. Since $|j|>2\tau_\ell$, we split
	\begin{equation}\label{6.8'}\nonumber
		\begin{aligned}
			a^1_{jk\tau_\ell}=&\frac{h(\tau_\ell)}{(\tau_\ell^2-|j|^2)(\tau_\ell^2-|k|^2)}-\frac{h(|j|)}{(\tau_\ell^2-|j|^2)(\tau_\ell^2-|k|^2)}\\
			=&c^0_{jk\tau_\ell}+c^1_{jk\tau_\ell}.
		\end{aligned}
	\end{equation}
	The arguments that we shall use to control terms involving $c^0_{jk\tau_\ell}$ and $c^1_{jk\tau_\ell}$ are similar, so for simplicity we shall only give the details for 
	$c^0_{jk\tau_\ell}$ here. Recall that by definition, 
	\[V_{jk}=\int_M \overline{e_j^0(y)}e_k^0(y)V(y)dy.\]
	Thus, it suffices to estimate
	\begin{align*}
		&\sum_{j, k, \tau_\ell}\beta_{\ell_1}(\tau_\ell-|k|)\beta_{\ell_2}(\tau_\ell-|j|)\iint c^0_{jk\tau_\ell}\,e_j^0(x)\overline{e_j^0(y)}e_k^0(y)V(y)\overline{e_k^0(z)}V(z)e_{\tau_\ell}(z) \overline{e_{\tau_\ell}(x)}dzdy \\
		=&\sum_{\tau_\ell\in[1,2\la]}\iint K_2(\tau_\ell, x, y) K_1(\tau_\ell, y, z) h(\tau_\ell) V(y)V(z)e_{\tau_\ell}(z) \overline{e_{\tau_\ell}(x)} dzdy
	\end{align*}
	where 
	$$K_\nu(\tau_\ell, x, y)=\sum_j \frac{\beta_{\ell_\nu}(\tau_\ell-|j|)}{\tau_\ell^2-|j|^2}e_j^0(x)\overline{e_j^0(y)}, \,\,\, \nu=1,2.
	$$
We claim that $K_\nu(\tau_\ell, x, y)$ are pseudo-differential operators and their kernels satisfy
	\begin{equation}\label{6.10}
		|K_\nu(\tau_\ell, x, y)|\ls 2^{(n-2)\ell_\nu} (1+2^{\ell_\nu}|x-y|)^{-N},\ \forall\, N.
	\end{equation}
To prove the claim \eqref{6.10}, we observe that $K_\nu(\tau_\ell,x,y)$ is the kernel of $2^{-2\ell_\nu} m(P^0/2^{\ell_\nu})$ where
\[m(t)=\beta(2^{-\ell_\nu}\tau_\ell-t)((2^{-\ell_\nu}\tau_\ell)^2-t^2) ^{-1}.\] Since $2^{-\ell_\nu}\tau_\ell\ls 1$ and $\beta\in C_0^\infty$, $m(t)$ is a symbol of order $\mu$ for any $\mu<-n$, i.e.
\[|\partial_t^\alpha m(t)|\le C_\alpha (1+|t|)^{\mu-\alpha},\ \forall \alpha. \]
Then  the claim follows from Lemma \ref{pdo}.

 Moreover, by using the rough eigenfunction bounds \eqref{rough},
	\begin{equation}\label{6.12}
		\sum_{\tau_\ell\in[1,2\la]}\big|h(\tau_\ell)e_{\tau_\ell}(z) \overline{e_{\tau_\ell}(x)} \big|\ls \la^n,\,\,\,\forall\, z,\,\,x\in M.
	\end{equation}
	By \eqref{6.12} and using H\"older's inequality twice, we have
	\begin{align*}
		&\sum_{\tau_\ell\in[1,2\la]}\iint K_2(\tau_\ell, x, y) K_1(\tau_\ell, y, z) h(\tau_\ell) V(y)V(z)e_{\tau_\ell}(z) \overline{e_{\tau_\ell}(x)} dzdy \\
		&\ls  \la^n\,\|V\|^2_{L^{\frac{n}{2-\eta}}(M)} \sup_{\tau_\ell, x}\|K_2(\tau_\ell, x, \cdot)\|_{L^{\frac{n}{n-2+\eta}}(M)}\cdot\sup_{\tau_\ell, y}\|K_1(\tau_\ell, y, \cdot)\|_{L^{\frac{n}{n-2+\eta}}(M)} \\
		&\ls  \la^{n} 2^{-\ell_1\eta}2^{-\ell_2\eta}.
	\end{align*}
	Summing over $\ell_1, \ell_2$, we get the desired bound $\la^{n-2\eta}$.
	
	For the other term $c_{jk\tau_\ell}^1$, we just need to replace $K_2$ above by
	\[K_2(\tau_\ell, x, y)=\sum_j \frac{\beta_{\ell_2}(\tau_\ell-|j|)}{\tau_\ell^2-|j|^2}h(|j|)e_j^0(x)\overline{e_j^0(y)}\]
	and the kernel estimate \eqref{6.10} still holds by Lemma \ref{pdo}. Moreover, by using the rough eigenfunction bounds \eqref{rough}, we have
\[	\sum_{\tau_\ell\in[1,2\la]}\big|e_{\tau_\ell}(z) \overline{e_{\tau_\ell}(x)} \big|\ls \la^n,\,\,\,\forall\, z,x\in M.\]Then we can get the same bound $\la^{n-2\eta}$ by repeating the argument above.
	\quad\\
	
	Next, we handle $a^0_{jk\tau_\ell}$. We consider three subcases separately: $2^{\ell_2}\approx 2^{\ell_1}$, $2^{\ell_2}< 2^{\ell_1-3}$, $2^{\ell_2}>2^{\ell_1+3}$. Since $|k|,|j|>2\tau_\ell$, we have $|k|\approx |j|$, $|k|\ge2|j|$, $|j|\ge2|k|$ respectively in these cases.
	
	\subsection*{(i)} $2^{\ell_1}\approx 2^{\ell_2}$, namely $|\ell_1-\ell_2|\le 3$. 
	
	In this case, $|k|\approx|j|\approx 2^{\ell_1}$, so
	by the mean value theorem, 
	\[
	\Big|\frac{h(|k|)-h(|j|)}{|k|^2-|j|^2}\Big|\ls 2^{-\ell_1} |h^\prime(|k|)|\ls 2^{-\ell_1}2^{-\ell_0}2^{-N\ell_1},\,\,\, \forall  N.\]
	Then
	\[|a_{jk\tau_\ell}^0|\ls 2^{-3\ell_1}2^{-\ell_0}2^{-N\ell_1} \ls 2^{-N\ell_1}.\]
	By \eqref{rough} and H\"older's inequality, 
	\begin{align*}
		&|\sum_{j, k, \tau_\ell}\beta_{\ell_1}(\tau_\ell-|k|)\beta_{\ell_2}(\tau_\ell-|j|) \iint a^0_{jk\tau_\ell}\,e_j^0(x)\overline{e_j^0(y)}V(y)e_k^0(y)\overline{e_k^0(z)}V(z)e_{\tau_\ell}(z) \overline{e_{\tau_\ell}(x)}dydz| \\
		&\ls\|V\|_{L^1(M)}^2\sum_{|j|\approx2^{\ell_1}}\sum_{|k|\approx 2^{\ell_1}}2^{-N\ell_1}\cdot\la^n\\
		&\approx 2^{2n\ell_1}2^{-N\ell_1}\cdot\la^{n}.
	\end{align*}
	
	Summing over $\ell_1,\ell_2$, we get  $\la^{3n-N}$, which is bounded by a constant for large $N$.
	
	\subsection*{(ii)} $2^{\ell_2}> 2^{\ell_1+3}$.

	In this case, $2^{\ell_2}\approx |j|\ge 2|k|\approx 2^{\ell_1}$. We write
	\begin{align*}
		a^0_{jk\tau_\ell}=&\frac{h(|k|)}{(|k|^2-|j|^2)(|k|^2-\tau_\ell^2)}-\frac{h(|j|)}{(|k|^2-|j|^2)(|k|^2-\tau_\ell^2)} \\
		=&d^0_{jk\tau_\ell}+d^1_{jk\tau_\ell}
	\end{align*}
	The second term satisfies
	\[
	|d_{jk\tau_\ell}^1|\ls 2^{-N\ell_2},\,\,\, \forall N,\]
	which implies 
	\begin{align*}
		&|\sum_{j, k, \tau_\ell}\beta_{\ell_1}(\tau_\ell-|k|)\beta_{\ell_2}(\tau_\ell-|j|) \iint d^1_{jk\tau_\ell}\,e_j^0(x)\overline{e_j^0(y)}V(y)e_k^0(y)\overline{e_k^0(z)}V(z)e_{\tau_\ell}(z) \overline{e_{\tau_\ell}(x)}dz| \\
		&\ls\|V\|_{L^1(M)}^2\sum_{|j|\approx2^{\ell_2}}\sum_{|k|\approx 2^{\ell_1}}2^{-N\ell_2}\cdot\la^n\\
		&\approx 2^{n\ell_2}2^{n\ell_1}2^{-N\ell_2}\cdot\la^{n}.
	\end{align*}
	Summing over $\ell_1$, $\ell_2$, we get $\la^{3n-N}$, which is bounded by a constant for large $N$.
	
	It remains to deal with $d_{jk\tau_\ell}^0$. We write
	\begin{align*}
		\sum_{j, k, \tau_\ell}\beta_{\ell_1}&(\tau_\ell-|k|)\beta_{\ell_2}(\tau_\ell-|j|)\iint d^0_{jk\tau_\ell}\,e_j^0(x)\overline{e_j^0(y)}e_k^0(y)V(y)\overline{e_k^0(z)}V(z)e_{\tau_\ell}(z) \overline{e_{\tau_\ell}(x)}dzdy \\
		= &\sum_{k, \tau_\ell}\iint K_2(\tau_\ell, k, x, y) K_1(\tau_\ell, k, y, z)  V(y)V(z)e_{\tau_\ell}(z) \overline{e_{\tau_\ell}(x)} dzdy
	\end{align*}
	where 
	$$K_2(\tau_\ell, k, x, y)=\sum_j \frac{\beta_{\ell_2}(\tau_\ell-|j|)}{|k|^2-|j|^2}e_j^0(x)\overline{e_j^0(y)},
	$$
	and 
	$$K_1(\tau_\ell, k, y, z)=\frac{\beta_{\ell_1}(\tau_\ell-|k|)}{|k|^2-\tau_\ell^2}h(|k|)e_k^0(y)\overline{e_k^0(z)}.
	$$
	By Lemma \ref{pdo}, if we argue as in the proof of \eqref{6.10}, we can see that $K_2$ is a pseudo-differential operators and its kernel $K_2(\tau_\ell,k, x, y)$ satisfies
	\begin{equation}\label{6.18}
		\sup_{\tau_\ell,k}	|K_2(\tau_\ell, k, x, y)|\ls 2^{(n-2)\ell_2} (1+2^{\ell_2}|x-y|)^{-N}\,\,\,\forall  N,
	\end{equation}
	Moreover,
	\begin{equation}\label{6.19}
		\sup_{\tau_\ell, k, y, z}|K_1(\tau_\ell, k, y, z)|\ls  2^{-N\ell_1},\,\,\, \forall N.
	\end{equation}
	By H\"older's inequality and \eqref{rough}, we have
	\begin{align*}
		&\sum_{|k|\approx 2^{\ell_1}}\sum_{\tau_\ell}\iint K_2(\tau_\ell, k, x, y) K_1(\tau_\ell, k, x, y)  V(y)V(z)e_{\tau_\ell}(z) \overline{e_{\tau_\ell}(x)} dzdy \\
		&\ls  2^{n\ell_1}\la^n\,\|V\|_{L^{\frac{n}{2-\eta}}(M)}\|V\|_{L^{1}(M)} \sup_{\tau_\ell,k, x}\|K_2(\tau_\ell, k,x, \cdot)\|_{L^{\frac{n}{n-2+\eta}}(M)}\sup_{\tau_\ell,k, y}\|K_1(\tau_\ell,k, y, \cdot)\|_{L^{\infty}(M)} \\
		&\ls  2^{n\ell_1}\la^{n} 2^{-\eta\ell_2}2^{-N\ell_1}.
	\end{align*}
	Summing over $\ell_1, \ell_2$, we get $\la^{2n-\eta-N}$, which is bounded by a constant for large $N$.
	
	\subsection*{(iii)} $2^{\ell_2}< 2^{\ell_1-3}$. 
	
	The argument is similar to (ii), but we give the details for completeness. In this case, $2^{\ell_1}\approx |k|\ge 2|j|\approx 2^{\ell_2}$. As before, we write
	\begin{align*}
		a^0_{jk\tau_\ell}=&\frac{h(|k|)}{(|k|^2-|j|^2)(|k|^2-\tau_\ell^2)}-\frac{h(|j|)}{(|k|^2-|j|^2)(|k|^2-\tau_\ell^2)} \\
		=&d^0_{jk\tau_\ell}+d^1_{jk\tau_\ell}
	\end{align*}
	The first term  satisfies
	\[|d_{jk\tau_\ell}^0|\ls 2^{-N\ell_1},\ \forall N.\]
	which implies 
	\begin{align*}
		&|\sum_{j, k, \tau_\ell}\beta_{\ell_1}(\tau_\ell-|k|)\beta_{\ell_2}(\tau_\ell-|j|) \iint d^0_{jk\tau_\ell}\,e_j^0(x)\overline{e_j^0(y)}V(y)e_k^0(y)\overline{e_k^0(z)}V(z)e_{\tau_\ell}(z) \overline{e_{\tau_\ell}(x)}dz| \\
		&\ls\|V\|_{L^1(M)}^2\sum_{|j|\approx2^{\ell_2}}\sum_{|k|\approx 2^{\ell_1}}2^{-N\ell_1}\cdot\la^n\\
		&\approx 2^{n\ell_2}2^{n\ell_1}2^{-N\ell_1}\cdot\la^{n}.
	\end{align*}
	Summing over $\ell_1$, $\ell_2$ we get $\la^{3n-N}$, which is bounded by a constant for large $N$.
	
	Now it remains to handle $d_{jk\tau_\ell}^1$. We write
	\begin{align*}
		\sum_{j, k, \tau_\ell}\beta_{\ell_1}&(\tau_\ell-|k|)\beta_{\ell_2}(\tau_\ell-|j|)\iint d^1_{jk\tau_\ell}\,e_j^0(x)\overline{e_j^0(y)}e_k^0(y)V(y)\overline{e_k^0(z)}V(z)e_{\tau_\ell}(z) \overline{e_{\tau_\ell}(x)}dzdy\big| \\
		= &\sum_{\tau_\ell, j}\iint K_2(\tau_\ell,j, x, y) K_1(\tau_\ell, j, y, z)  V(y)V(z)e_{\tau_\ell}(z) \overline{e_{\tau_\ell}(x)} dzdy
	\end{align*}
	where 
	$$K_2(\tau_\ell, j, x, y)=\beta_{\ell_2}(\tau_\ell-|j|)h(|j|)e_j^0(x)\overline{e_j^0(y)},
	$$
	and 
	$$K_1(\tau_\ell, j, y, z)=\sum_k \frac{\beta_{\ell_1}(\tau_\ell-|k|)}{(|k|^2-|j|^2)(|k|^2-|\tau_\ell|^2)}e_k^0(y)\overline{e_k^0(z)}.
	$$
	By Lemma \ref{pdo}, if we argue as in the proof of \eqref{6.10}, we can show that $K_1$ is a pseudo-differential operator and its kernel $K_1(\tau_\ell,j, y, z)$ satisfies
	\begin{equation}\label{6.23}
		\sup_{\tau_\ell,j}|K_1(\tau_\ell, j, y, z)|\ls 2^{(n-4)\ell_1} (1+2^{\ell_1}|y-z|)^{-N},\,\,\,\forall N,
	\end{equation}
	Moreover,
	\[\sup_{\tau_\ell, j, x,y}|K_2(\tau_\ell, j, x, y)|\ls  2^{-N\ell_2},\,\,\, \forall N.
	\]
	By H\"older's inequality and \eqref{rough}, we have
	\begin{equation}\nonumber
		\begin{aligned}
			&\sum_{|j|\approx 2^{\ell_2}}\sum_{\tau_\ell}\iint K_2(\tau_\ell, j,x, y) K_1(\tau_\ell,j, y, z)  V(y)V(z)e_{\tau_\ell}(z) \overline{e_{\tau_\ell}(x)} dzdy \\
			\ls & 2^{n\ell_2}\la^n\,\|V\|_{L^{\frac{n}{2-\eta}}(M)}\|V\|_{L^{1}(M)} \sup_{\tau_\ell,j, x}\|K_2(\tau_\ell,j, x, \cdot)\|_{L^{\infty}(M)}\sup_{\tau_\ell,j, y}\|K_1(\tau_\ell, j,y, \cdot)\|_{L^{\frac{n}{n-2+\eta}}(M)} \\
			\ls & 2^{n\ell_2}\la^{n} 2^{-N\ell_2}2^{-\ell_1(\eta+2)}.
		\end{aligned}
	\end{equation}
	Summing over $\ell_1, \ell_2$, we get $\la^{2n-2-\eta-N}$, which is bounded by a constant for large $N$.

	\section{ High-frequency estimates}
	
	In this section we shall handle the high frequencies $\tau_\ell>2\la$. Our goal is to show that 
	\begin{equation}\label{h.1}
		\Big|\sum_{j,k}\sum_{\tau_\ell> 2\la}
		\int a_{jk\tau_\ell}\,e_j^0(x)V_{jk}\overline{e_k^0(z)}V(z)e_{\tau_\ell}(z) \overline{e_{\tau_\ell}(x)}dz\Big|
		\ls \la^{n-\frac32\eta}\log\la,
	\end{equation}for $a_{jk\tau_\ell}$ defined as in \eqref{2.2}.
	
	We shall divide our discussion into the following three cases:
	\begin{enumerate}[$\bullet$]
		\item $\tau_\ell\approx |k|$
		\item $\tau_\ell\ls |k|$
		\item $\tau_\ell\gs |k|$.
	\end{enumerate}
 The first two cases can be handled in a way that is similar to the low-frequency case using the decay properties of $h(\tau)$ for large $\tau$ as in \eqref{hda}. More explicitly, 
the first case is straightforward and can be handled by Lemma \ref{delta1}. The second case consists of two subcases: Case 2.1 ($|j|\ls \tau_\ell\ls |k|$) and Case 2.2 ($|j|,|k|\gs\tau_\ell$). We still apply Lemma \ref{delta1} to the first one, and use the kernel estimates of pseudo-differential operators (Lemma \ref{pdo}) for the other. The third case is more involved, and $a_{jk\tau_\ell}^0$, $a_{jk\tau_\ell}^1$ are treated seperately. We first split the term with $a_{jk\tau_\ell}^0$ into three parts (see \eqref{hde}), and handle them in Cases 3.1.1, 3.1.2, 3.1.3. The first one is handled by the heat kernel estimates (Lemma \ref{heatk}), while the other two follow from Lemma \ref{delta1}. Next, to handle the term with $a_{jk\tau_\ell}^1$, we expand 
\[
	\frac{1}{\tau_\ell^2-|j|^2}=\tau_\ell^{-2}
	+\tau_\ell^{-2}\bigl(|j|/\tau_\ell\bigr)^2+
	\cdots + \tau_\ell^{-2}\bigl(|j|/\tau_\ell\bigr)^{2N-2}
	\\
	+(|j|/\tau_\ell)^{2N}\frac{1}{\tau_\ell^2-|j|^2}\]
and expand $\frac{1}{\tau_\ell^2-|k|^2}$ similarly. Then we split the product of the two expansions  into three parts, and handle them seperately in Cases 3.2.1, 3.2.2, 3.2.3. The first two cases follow from Lemma \ref{delta1}. To handle Case 3.2.3, we split it into three cases (see \eqref{hde1}): Cases 3.2.3a, 3.2.3b, 3.2.3c. Only Case 3.2.3a is handled by the heat kernel estimates (Lemma \ref{heatk}), while the remaining cases still follow from  Lemma \ref{delta1}.

Moreover, Lemma \ref{sogge888} (spectral projection bounds) and Corollary \ref{roub} (rough eigenfunction bounds) are basic tools  applied to all the cases.

	\subsection*{\textbf{Case 1}} $\tau_\ell\approx |k|$.
	
	Fix a smooth bump function $\phi \in C_0^\infty (\mathbb{R})$ satisfying $\1_{[0.9,1.1]}\le \phi\le \1_{[0.8,1.2]}$. Our current task is to show that 
	\begin{equation}\label{h.1.1}
		|\sum_{j, k}\sum_{\tau_\ell> 2\la}\phi(|k|/\tau_\ell) \int a_{jk\tau_\ell}\,e_j^0(x)V_{jk}\overline{e_k^0(z)}V(z)e_{\tau_\ell}(z) \overline{e_{\tau_\ell}(x)}dz| 
		\ls \la^{n-\frac32\eta}.
	\end{equation}
	Let $\tau_\ell\in [2^s,2^{s+1}]$, $2^s>2\la$.  In this case, we have $0.8\tau_\ell\le |k|\le 1.2\tau_\ell$, and then  $|k|\approx 2^s$.
	If $|j|>1.3\tau_\ell$, then $|j|>\frac{13}{12}|k|$ and by \eqref{hda}
	\[|a_{jk\tau_\ell}|\ls |j|^{-2}2^{-Ns}, \ \forall N. \]
	If $0.7\tau_\ell\le |j|\le 1.3\tau_\ell$, then by the mean value theorem, if we argue as in the proof of \eqref{meanvalue 1}, one can show that
	\[|a_{jk\tau_\ell}|\ls 2^{-Ns},\ \forall N.\]
	If $|j|<0.7\tau_\ell$, then similarly by the mean value theorem, 
	\[|a_{jk\tau_\ell}|\ls 2^{-4s}.\]
	Let
	\[b_{k,\tau_\ell}=\sum_j \phi(|k|/\tau_\ell)
	a_{jk\tau_\ell} \,e_j^0(x)V_{jk}.\]
Combining the estimates on $a_{jk\tau_\ell}$ with the bound \eqref{vjkb} on $V_{jk}$, we get
	\[|b_{k,\tau_\ell}|\ls \la^n2^{(-n-2-\eta)s}\]
	and similarly
	\[|\partial_{\tau_\ell}b_{k,\tau_\ell}|\ls \la^n2^{(-n-2-\eta)s}2^{-s}.\]
	Thus by Sobolev inequality, we have
	\begin{equation}\label{h.1.13}
		\begin{aligned}
			\|\sum_{|k|\approx 2^s} b_{k,2^s}e_k^0(\cdot)\|_{L^{p_0}(M)}&\ls  2^{s/2}2^{\sigma(p_0)s}\cdot 2^{ns/2}\cdot \la^n2^{(-n-2-\eta)s}\\ &
			\ls \la^n 2^{(-n/2-1-\frac54\eta)s} 
		\end{aligned}
	\end{equation}
	And similarly for $\tau\in[2^s,2^{s+1}]$
	\begin{equation}\label{h.1.14}
		\|\sum_{|k|\approx 2^s}\partial_{\tau}  b_{k,\tau}e_k^0(\cdot)\|_{L^{p_0}(M)}\ls \la^n 2^{(-n/2-1-\frac54\eta)s} 2^{-s}. 
	\end{equation}
	In view of \eqref{h.1.13} and \eqref{h.1.14}, by applying Lemma~\ref{delta1} with $\delta=2^s$, we have 
	\begin{equation}\label{h.1.15}
		\begin{aligned}
			|\int &\sum_{\tau_\ell\in  [2^{s}, 2^{s+1}]}\sum_{k} b_{k,\tau_\ell}\overline{e_k^0(z)}V(z)e_{\tau_\ell}(z) \overline{e_{\tau_\ell}(x)}dz |\\
			\ls & \|V\|_{L^{\frac{n}{2-\eta}}(M)} \la^n 2^{(-n/2-1-\frac54\eta)s} \cdot 2^{s/2}2^{\sigma(p_0)s}  ( \sum_{\tau_\ell\in  [2^{s}, 2^{s+1}]}|e_{\tau_\ell}(x)|^2 )^{1/2}  \\
			\ls &\|V\|_{L^{\frac{n}{2-\eta}}(M)} \la^n 2^{(-n/2-1-\frac54\eta)s}\cdot 2^{s/2}2^{(\frac12-\frac14\eta)s} 2^{ns/2} .
		\end{aligned}
	\end{equation}
	Summing over $s$, we get the desired bound $\la^{n-\frac32\eta}$.

	\subsection*{\textbf{Case 2}} $\tau_\ell\ls |k|$.
	
	Let $\rho_1=(1-\phi)\1_{(1,\infty)}\in C^\infty(\mathbb{R})$. We have $\1_{[1.2,\infty)}\le \rho_1\le \1_{[1.1,\infty)}$. Our goal in this section is 
	\begin{equation}\label{h.2.1}
		|\sum_{j, k}\sum_{\tau_\ell> 2\la}\rho_1(|k|/\tau_\ell) \int a_{jk\tau_\ell}\,e_j^0(x)V_{jk}\overline{e_k^0(z)}V(z)e_{\tau_\ell}(z) \overline{e_{\tau_\ell}(x)}dz|
		\ls\la^{n-\frac32\eta}\log\la.
	\end{equation}
	Let $\tau_\ell\in [2^s,2^{s+1}]$, $2^s>2\la$.  In this case, we have $|k|\ge 1.1\tau_\ell>2.2\la$.
	We just need to estimate 
	\[	\sum_{\tau_\ell\in[2^s,2^{s+1}]}\sum_{j, k}\rho_1(|k|/\tau_\ell)\beta_{\ell_1}(|k|)\beta_{\ell_2}(|j|) \int a_{jk\tau_\ell}\,e_j^0(x)V_{jk}\overline{e_k^0(z)}V(z)e_{\tau_\ell}(z) \overline{e_{\tau_\ell}(x)}dz\]for $\ell_1\ge s, \ell_2\ge\ell_0$.
	We consider two cases separately: $2^{\ell_2}\le 2^{s+2}$, $2^{\ell_2}> 2^{s+2}$.

	\subsection*{\textbf{Case 2.1}} $2^{\ell_2}\le 2^{s+2}$.
	
	In this case, we have $|k|\pm\tau_\ell\approx |k|+|j|\approx |k|\approx 2^{\ell_1}$, $\tau_\ell+|j|\approx 2^s$. So we get for all $N$,
	\[|a_{jk\tau_\ell}^0|\ls  \begin{cases}2^{-4\ell_1},\ \ \ \ \ \ \ \ \ \  |j|\le 1.2\la\\
		2^{-4\ell_1}2^{-N\ell_2},\ \ \ 1.2\la<|j|<0.8|k|\\
		2^{-N\ell_1},\  \ \ \ \ \ \ \ \  |j|\ge 0.8|k|
	\end{cases}\]
	and
	\[|a_{jk\tau_\ell}^1|\ls \begin{cases}2^{-2\ell_1}2^{-2s},\ \ \ \ \ \ \ \ \ \  |j|<1.2\la\\
		2^{-2\ell_1}2^{-2s}2^{-N\ell_2},\ \ \ 1.2\la<|j|<0.8\tau_\ell\\
		2^{-2\ell_1}2^{-Ns},\ \ \ \ \  \ \ \ \  |j|\ge 0.8\tau_\ell.
	\end{cases}\]
	So we have
	\[|a_{jk\tau_\ell}|\ls \begin{cases}2^{-2\ell_1}2^{-2s},\ \ \ \ \ \ \ \ \ \ \  |j|\le1.2\la\\
		2^{-2\ell_1}2^{-2s}2^{-N\ell_2},\ \ \ 1.2\la<|j|<0.8|k|\\
		2^{-2\ell_1}2^{-Ns},\ \ \ \ \ \ \ \ \ \ |j|\ge 0.8|k|.
	\end{cases}\]
	Let 
	\[b_{k,\tau_\ell}=\sum_{j}\rho_1(|k|/\tau_\ell)\beta_{\ell_1}(|k|)\beta_{\ell_2}(|j|) a_{jk\tau_\ell}\,e_j^0(x)V_{jk}.\]
Combining the estimates on $a_{jk\tau_\ell}$ with the bound \eqref{vjkb} on $V_{jk}$, we have
	\[|b_{k,\tau_\ell}|\ls \la^n2^{(-n-\eta)\ell_1}2^{-2s}\]
	and similarly
	\[|\partial_{\tau_\ell}b_{k,\tau_\ell}|\ls \la^n2^{(-n-\eta)\ell_1}2^{-3s}.\]
	By using Sobolev inequality, we have
	\begin{equation}\label{h.2.30}
		\begin{aligned}
			\|\sum_{|k|\approx 2^{\ell_1}} b_{k,2^s}e_k^0(\cdot)\|_{L^{p_0}(M)}
			&\ls 2^{\ell_1/2}2^{\sigma(p_0)\ell_1}2^{n\ell_1/2}\cdot \la^n2^{(-n-\eta)\ell_1} 2^{-2s}\\
			&\approx \la^n2^{(-n/2+1-\frac54\eta)\ell_1}2^{-2s}.
		\end{aligned}
	\end{equation}
	And similarly for $\tau\in[2^s,2^{s+1}]$
	\begin{equation}\label{h.2.31}
		\|\sum_{|k|\approx 2^{\ell_1}}\partial_{\tau}  b_{k,\tau}e_k^0(\cdot)\|_{L^{p_0}(M)}\ls \la^n2^{(-n/2+1-\frac54\eta)\ell_1}2^{-3s}.  \end{equation}
	In view of \eqref{h.2.30}, by applying Lemma~\ref{delta1} with $\delta=2^{s}$, we have
	\begin{equation}\label{h.2.33}
		\begin{aligned}
			|\int \sum_{\tau_\ell\in  [2^s, 2^{s+1}]}\sum_{k} &b_{k,\tau_\ell} e_k^0(z)V(z)e_{\tau_\ell}(z) \overline{e_{\tau_\ell}(x)}dz |\\
			\ls & \|V\|_{L^{\frac{n}{2-\eta}}(M)} \la^n2^{(-n/2+1-\frac54\eta)\ell_1}2^{-2s}\cdot2^{s/2}2^{\sigma(p_0)s}2^{ns/2}\\
			\ls&\|V\|_{L^{\frac{n}{2-\eta}}(M)} \la^n2^{(-n/2+1-\frac54\eta)\ell_1}2^{(n/2-1-\frac14\eta)s}
		\end{aligned}
	\end{equation}
	Note that $-n/2+1-\frac54\eta<0$. Summing over $\ell_1,s,\ell_2$, we get the desired bound $\la^{n-\frac32\eta}\log\la$.

	\subsection*{Case 2.2} $2^{\ell_2}> 2^{s+2}$.
	
	In this case, we have $|j|\ge 2\tau_\ell$, $|k|>1.1\tau_\ell$. So we may apply the same method in Case 4 of the low-frequency section.

	We shall first deal with the terms involving $a^1_{jk\tau_\ell}$. Let
	\begin{equation}\label{h.2.5}\nonumber
		\begin{aligned}
			a^1_{jk\tau_\ell}=&\frac{h(\tau_\ell)}{(\tau_\ell^2-|j|^2)(\tau_\ell^2-|k|^2)}-\frac{h(|j|)}{(\tau_\ell^2-|j|^2)(\tau_\ell^2-|k|^2)}\\
			=&c^0_{jk\tau_\ell}+c^1_{jk\tau_\ell}.
		\end{aligned}
	\end{equation}
	
	The arguments that we shall use to control terms involving $c^0_{jk\tau_\ell}$ and $c^1_{jk\tau_\ell}$ are similar, for simplicity we shall only give the details for 
	$c^0_{jk\tau_\ell}$ here. Recall that by definition, 
	\[V_{jk}=\int_M \overline{e_j^0(z)}e_k^0(z)V(z)dz.\]
	Thus, it suffices to estimate
	\begin{equation}\label{h.2.6}
		\begin{aligned}
			&\sum_{\tau_\ell\in[2^s, 2^{s+1}]}\sum_{j, k}\rho_1(|k|/\tau_\ell)\beta_{\ell_1}(|k|)\beta_{\ell_2}(|j|)  \\
			&\qquad\cdot\iint c^0_{jk\tau_\ell}\,e_j^0(x)\overline{e_j^0(y)}e_k^0(y)V(y)\overline{e_k^0(z)}V(z)e_{\tau_\ell}(z) \overline{e_{\tau_\ell}(x)}dzdy \\
			= &\sum_{\tau_\ell}\iint K_2(\tau_\ell, x, y) K_1(\tau_\ell, y, z)  h(\tau_\ell)V(y)V(z)e_{\tau_\ell}(z) \overline{e_{\tau_\ell}(x)} dzdy
		\end{aligned}
	\end{equation}
	where 
	$$K_1(\tau_\ell, x, y)=\sum_k \rho_1(|k|/\tau_\ell)\frac{\beta_{\ell_1}(|k|)}{\tau_\ell^2-|k|^2}e_k^0(x)\overline{e_k^0(y)}, 
	$$
	and 
	$$K_2(\tau_\ell, x, y)=\sum_j \frac{\beta_{\ell_2}(|j|)}{\tau_\ell^2-|j|^2}e_j^0(x)\overline{e_j^0(y)}.
	$$
By Lemma \ref{pdo}, if we argue as in the proof of \eqref{6.10}, we can show that $K_\nu$ is a pseudo-differential operator and its kernel $K_\nu(\tau_\ell,x,y)$ satisfies
	\begin{equation}\label{h.2.7}
		|K_\nu(\tau_\ell, x, y)|\ls 2^{(n-2)\ell_\nu} (1+2^{\ell_\nu}|x-y|)^{-N},\ \nu=1,2,\,\,\,\forall N.
	\end{equation}

	As a result of \eqref{h.2.7}, by direct calculation, we have
	\begin{equation}\label{h.2.8}
		\sup_{\tau_\ell, x}\|K_\nu(\tau_\ell, x, \cdot)\|_{L^{\frac{n}{n-2+\eta}}(M)}\ls 2^{-\eta\ell_\nu},\ \nu=1,2.
	\end{equation}
	On the other hand, since $\tau_\ell\ge2\la$, by using the rough eigenfunction bounds \eqref{rough} we have
	\begin{equation}\label{h.2.9}
		\sum_{\tau_\ell\in [2^s,2^{s+1}]}\big|h(\tau_\ell)e_{\tau_\ell}(z) \overline{e_{\tau_\ell}(x)} \big|\ls 2^{-Ns},\,\,\,\forall N.
	\end{equation}
	By using H\"older's inequality twice, we have
	\begin{equation}\nonumber
		\begin{aligned}
			\sum_{\tau_\ell\in[2^s,2^{s+1}]}\iint &K_2(\tau_\ell, x, y) K_1(\tau_\ell, y, z)  h(\tau_\ell)V(y)V(z)e_{\tau_\ell}(z) \overline{e_{\tau_\ell}(x)} dzdy \\
			\ls & 2^{-Ns}\|V\|^2_{L^{\frac{n}{2-\eta}}(M)} \sup_{\tau_\ell, x}\|K_2(\tau_\ell, x, \cdot)\|_{L^{\frac{n}{n-2+\eta}}(M)}\sup_{\tau_\ell, y}\|K_1(\tau_\ell, y, \cdot)\|_{L^{\frac{n}{n-2+\eta}}(M)} \\
			\ls & 2^{-Ns} 2^{-\ell_1\eta}2^{-\ell_2\eta}.
		\end{aligned}
	\end{equation}
	After summing over $s, \ell_1, \ell_2$, this gives us a constant bound for large $N$. 
	
	For the other term $c_{jk\tau_\ell}^1$, we just need to replace $K_2$ above by
	\[K_2(\tau_\ell, x, y)=\sum_j \frac{\beta_{\ell_2}(|j|)}{\tau_\ell^2-|j|^2}h(|j|)e_j^0(x)\overline{e_j^0(y)}\]
	and similarly it satisfies
	\[|K_2(\tau_\ell, x, y)|\ls 2^{-Ns}2^{(n-2)\ell_\nu} (1+2^{\ell_\nu}|x-y|)^{-N},\ \forall N.\]Moreover, by using the rough eigenfunction bounds \eqref{rough} we have
\[	\sum_{\tau_\ell\in[2^s,2^{s+1}]}\big|e_{\tau_\ell}(z) \overline{e_{\tau_\ell}(x)} \big|\ls 2^{ns},\,\,\,\forall\, z,x\in M.\]Then we can get the same constant bound for large $N$, by repeating the argument above.
	
	\quad\\
	
	Now we shall deal with the terms involving $a^0_{jk\tau_\ell}$, we shall divide our discussion into three cases: $2^{\ell_1}\approx 2^{\ell_2}$, $2^{\ell_2}>2^{\ell_1+3}$, $2^{\ell_2}< 2^{\ell_1-3}$. Since $|k|,|j|>1.1\tau_\ell$, we have $|k|\approx |j|$, $|k|\ge2|j|$, $|j|\ge2|k|$ respectively in these cases.
	
	\subsection*{(i)} $2^{\ell_1}\approx 2^{\ell_2}$, or in other words $|\ell_1-\ell_2|\le 3$. 
	
	In this case since $|k|, |j|\ge \tau_\ell\ge 2\la$,
	by the mean value theorem, 
	\begin{equation}\label{h.2.10}
		\Big|\frac{h(|k|)-h(|j|)}{|k|^2-|j|^2}\Big|\ls 2^{-N\ell_1},\,\,\, \forall N.
	\end{equation}
	It is easy to see that 
	\begin{equation}\label{h.2.11}
		\begin{aligned}
			|\sum_{\tau_\ell\in[2^s, 2^{s+1}]}&\sum_{ j, k}\rho_1(|k|/\tau_\ell)\beta_{\ell_1}(|k|)\beta_{\ell_2}(|j|) \\
			&\cdot \int a^0_{jk\tau_\ell}\,e_j^0(x)V_{jk}\overline{e_k^0(z)}V(z)e_{\tau_\ell}(z) \overline{e_{\tau_\ell}(x)}dz| \\
			\ls &2^{-N\ell_1}.
		\end{aligned}
	\end{equation}
	Summing over $s,\ell_2,\ell_1$, we get  a constant bound for large $N$.

	\subsection*{(ii)} $2^{\ell_2}> 2^{\ell_1+3}$. 
	
	In this case, $|j|\ge 2|k|\ge2.2\tau_\ell$. Let
	\begin{align*}
		a^0_{jk\tau_\ell}=&\frac{h(|k|)}{(|k|^2-|j|^2)(|k|^2-\tau_\ell^2)}-\frac{h(|j|)}{(|k|^2-|j|^2)(|k|^2-\tau_\ell^2)} \\
		=&d^0_{jk\tau_\ell}+d^1_{jk\tau_\ell}
	\end{align*}
	Note that
	\begin{equation}\label{h.2.12}
		|h(|j|)|\ls 2^{-N\ell_2},\,\,\, \forall N,
	\end{equation}
	which implies 
	\begin{equation}\label{h.2.13}
		\begin{aligned}
			|&\sum_{\tau\in[2^s, 2^{s+1}]}\sum_{j, k}\rho_1(|k|/\tau_\ell)\beta_{\ell_1}(|k|)\beta_{\ell_2}(|j|) \\
			&\cdot \int d^1_{jk\tau_\ell}\,e_j^0(x)V_{jk}\overline{e_k^0(z)}V(z)e_{\tau_\ell}(z) \overline{e_{\tau_\ell}(x)}dz| \\
			&\ls2^{-N\ell_2}.
		\end{aligned}
	\end{equation}
	Summing over $s,\ell_1,\ell_2$, we get a constant bound for large $N$.
	
	So it suffices to estimate
	\begin{equation}\label{h.2.14}
		\begin{aligned}
			&\sum_{\tau\in[2^s, 2^{s+1}]}\sum_{j, k}\rho_1(|k|/\tau_\ell)\beta_{\ell_1}(|k|)\beta_{\ell_2}(|j|)
			\\ 
			&\qquad\cdot\iint d^0_{jk\tau_\ell}\,e_j^0(x)\overline{e_j^0(y)}e_k^0(y)V(y)\overline{e_k^0(z)}V(z)e_{\tau_\ell}(z) \overline{e_{\tau_\ell}(x)}dzdy \\
			&= \sum_{k, \tau_\ell}\iint K_2(\tau_\ell, k, x, y) K_1(\tau_\ell, k, y, z)  V(y)V(z)e_{\tau_\ell}(z) \overline{e_{\tau_\ell}(x)} dzdy
		\end{aligned}
	\end{equation}
	where 
	$$K_2(\tau_\ell, k, x, y)=\sum_j \frac{\beta_{\ell_2}(|j|)}{|k|^2-|j|^2}e_j^0(x)\overline{e_j^0(y)},
	$$
	and 
	$$K_1(\tau_\ell, k, x, y)=\rho_1(|k|/\tau_\ell)\frac{\beta_{\ell_1}(|k|)}{|k|^2-|\tau_\ell|^2}h(|k|)e_k^0(y)\overline{e_k^0(z)}.
	$$
By Lemma \ref{pdo}, if we argue as in the proof of \eqref{6.10}, we can show that $K_2$ is a pseudo-differential operator with the kernel
	\begin{equation}\label{h.2.15}
		|K_2(\tau_\ell, k, x, y)|\ls 2^{(n-2)\ell_2} (1+2^{\ell_2}|x-y|)^{-N},\,\,\,\forall\, N,
	\end{equation}
	where the implicit constant is independent of $|k|$.
	On the other hand, 
	\begin{equation}\nonumber
		|h(|k|)|\ls 2^{-N\ell_1},\,\,\, \forall N,
	\end{equation}
	which implies 
	\begin{equation}\label{h.2.16}
		\sup_{\tau_\ell, k, x, y}|K_1(\tau_\ell, k, x, y)|\ls  2^{-N\ell_1},\,\,\, \forall N.
	\end{equation}
	By H\"older's inequality and the fact that 
	\begin{equation}\label{h.2.17}
		\sum_{\tau_\ell\in[2^s, 2^{s+1}]}\big|e_{\tau_\ell}(z) \overline{e_{\tau_\ell}(x)} \big|\ls 2^{ns},\,\,\,\forall z,x\in M,
	\end{equation}
	we have
	\begin{equation}\nonumber
		\begin{aligned}
			\sum_{|k|\approx 2^{\ell_1}}\sum_{\tau_\ell\approx 2^s}\iint &K_2(\tau_\ell, k, x, y) K_1(\tau_\ell, k, y, z)  V(y)V(z)e_{\tau_\ell}(z) \overline{e_{\tau_\ell}(x)} dzdy \\
			\ls & 2^{n\ell_1}2^{ns}\,\|V\|_{L^{\frac{n}{2-\eta}}(M)}\|V\|_{L^{1}(M)} \sup_{\tau_\ell, x}\|K_2(\tau_\ell,k, x, \cdot)\|_{L^{\frac{n}{n-2+\eta}}(M)}\\
			&\,\,\cdot\sup_{\tau_\ell, y}\|K_1(\tau_\ell, k,y, \cdot)\|_{L^{\infty}(M)} \\
			\ls & 2^{ns} 2^{-N\ell_1}2^{-\eta\ell_2}.
		\end{aligned}
	\end{equation}
	Summing over $\ell_1, \ell_2,s$, we get a constant bound for large $N$.

	\subsection*{(iii)} $2^{\ell_2}< 2^{\ell_1-3}$.
	
	In this case, $|k|\ge 2|j|\ge4\tau_\ell$. Let
	\begin{align*}
		a^0_{jk\tau_\ell}=&\frac{h(|k|)}{(|k|^2-|j|^2)(|k|^2-\tau_\ell^2)}-\frac{h(|j|)}{(|k|^2-|j|^2)(|k|^2-\tau_\ell^2)} \\
		=&d^0_{jk\tau_\ell}+d^1_{jk\tau_\ell}
	\end{align*}
	Note that 
	\begin{equation}\label{h.2.18}
		|h(|k|)|\ls 2^{-N\ell_1},\,\,\, \forall N,
	\end{equation}
	which implies 
	\begin{equation}\label{h.2.19}
		\begin{aligned}
			|\sum_{\tau\in[2^s, 2^{s+1}]}&\sum_{j, k}\rho_1(|k|/\tau_\ell)\beta_{\ell_1}(|k|)\beta_{\ell_2}(|j|) \\
			&\cdot \int d^0_{jk\tau_\ell}\,e_j^0(x)V_{jk}\overline{e_k^0(z)}V(z)e_{\tau_\ell}(z) \overline{e_{\tau_\ell}(x)}dz| \\
			\ls &2^{-N\ell_1},\,\,\, \forall\,N.
		\end{aligned}
	\end{equation}
	Summing over $\ell_1$, $\ell_2,s$ gives us a constant bound for large $N$.
	
	So it suffices to estimate
	\begin{equation}\label{h.2.20}
		\begin{aligned}
			\big|&\sum_{\tau\in[2^s, 2^{s+1}]}\sum_{j, k}\rho_1(|k|/\tau_\ell)\beta_{\ell_1}(|k|)\beta_{\ell_2}(|j|)
			\\ 
			&\qquad\cdot\iint d^1_{jk\tau_\ell}\,e_j^0(x)\overline{e_j^0(y)}e_k^0(y)V(y)\overline{e_k^0(z)}V(z)e_{\tau_\ell}(z) \overline{e_{\tau_\ell}(x)}dzdy\big| \\
			&=\sum_{\tau_\ell, j}\iint K_2(\tau_\ell, j, x, y) K_1(\tau_\ell, j, y, z)  V(y)V(z)e_{\tau_\ell}(z) \overline{e_{\tau_\ell}(x)} dzdy
		\end{aligned}
	\end{equation}
	where 
	$$K_2(\tau_\ell, j,x, y)=\beta_{\ell_2}(|j|)h(|j|)e_j^0(x)\overline{e_j^0(y)},
	$$
	and 
	$$K_1(\tau_\ell, j, y, z)=\sum_k \rho_1(|k|/\tau_\ell)\frac{\beta_{\ell_1}(|k|)}{(|k|^2-|j|^2)(|k|^2-|\tau_\ell|^2)}e_k^0(y)\overline{e_k^0(z)}.
	$$
	By Lemma \ref{pdo}, if we argue as in the proof of \eqref{6.10}, we can show that $K_1$ is a pseudo-differential operator with kernel
	\begin{equation}\label{h.2.21}
		|K_1(\tau_\ell, j, x, y)|\ls 2^{(n-4)\ell_1} (1+2^{\ell_1}|x-y|)^{-N}\,\,\,\forall\, N,
	\end{equation}
	where the implicit constant is independent of $|j|$.
	On the other hand, 
	\begin{equation}\nonumber
		|h(|j|)|\ls 2^{-N\ell_2},\,\,\, \forall N,
	\end{equation}
	which implies 
	\begin{equation}\label{h.2.22}
		\sup_{\tau_\ell, j, y, z}|K_2(\tau_\ell, j, y, z)|\ls  2^{-N\ell_2},\,\,\, \forall\,\,\,N.
	\end{equation}
	By H\"older's inequality and \eqref{h.2.17}, we have
	\begin{equation}\nonumber
		\begin{aligned}
			&\sum_{|j|\approx 2^{\ell_2}}\sum_{\tau_\ell\approx 2^s}\iint K_2(\tau_\ell, j, x, y) K_1(\tau_\ell, j, y, z)  V(y)V(z)e_{\tau_\ell}(z) \overline{e_{\tau_\ell}(x)} dzdy \\
			\ls & 2^{n\ell_2}2^{ns}\,\|V\|_{L^{\frac{n}{2-\eta}}(M)}\|V\|_{L^{1}(M)} \sup_{\tau_\ell,j, x}\|K_2(\tau_\ell,j, x, \cdot)\|_{L^{\infty}(M)}\cdot\sup_{\tau_\ell,j, y}\|K_1(\tau_\ell, j,y, \cdot)\|_{L^{\frac{n}{n-2+\eta}}(M)} \\
			\ls & 2^{ns} 2^{-N\ell_2}2^{-(\eta+2)\ell_1}.
		\end{aligned}
	\end{equation}
	Summing over $\ell_1, \ell_2,s$, we get a constant bound for large $N$.

	\subsection*{\textbf{Case 3}} $\tau_\ell\gs |k|$.
	
	Let $\rho_2=(1-\phi)\1_{(-\infty,1)}\in C^\infty(\mathbb{R})$. We have $\1_{(-\infty,0.8]}\le \rho_2\le \1_{(-\infty,0.9]}$. Our goal in this section is 
	\begin{equation}\label{h.3.1}
		\begin{aligned}
			|\sum_{j, k}&\sum_{\tau_\ell> 2\la}\rho_2(|k|/\tau_\ell)
			\int a_{jk\tau_\ell}\,e_j^0(x)V_{jk}\overline{e_k^0(z)}V(z)e_{\tau_\ell}(z) \overline{e_{\tau_\ell}(x)}dz| \\
			\ls &\la^{n-\frac32\eta}\log\la.
		\end{aligned}
	\end{equation}
	
	As before, we shall write
	\begin{equation}\label{h.3.3}\nonumber
		\begin{aligned}
			a_{jk\tau_\ell}=&\frac{h(|k|)-h(|j|)}{(|k|^2-|j|^2)(|k|^2-\tau_\ell^2)}+\frac{h(\tau_\ell)-h(|j|)}{(\tau_\ell^2-|j|^2)(\tau_\ell^2-|k|^2)}\\
			=&a^0_{jk\tau_\ell}+a^1_{jk\tau_\ell}.
		\end{aligned}
	\end{equation}
	Let
	\begin{equation}\label{h.3.7}
		b_k=\sum_j \frac{h(|k|)-h(|j|)}{|k|^2-|j|^2}e_j^0(x)V_{jk}.
	\end{equation}
	Then we claim that
	\begin{equation}\label{h.3.7''}
		|b_k|\ls \begin{cases}\la^{-\eta}\log\la,\  \ \ \ \ \ \  |k|\le10\la\\
			\la^n|k|^{-n-\eta},\ \ \ \ \  |k|>10\la.\end{cases}
	\end{equation}
	Indeed, recall the bound \eqref{vjkb} on $V_{jk}$:
	\[|V_{jk}|\ls (1+|j-k|)^{-n+2-\eta}.\]
	For any fixed $k$, the number of $j\in\mathbb{Z}^n$ satisfying $1+|j-k|\approx  2^{m}$ and $||k|-|j||\ls2^{\ell_3}$ is bounded by $ 2^{(n-1)m}2^{\ell_3}$. Here $m\in\mathbb{N}$ and $\ell_3\ge\ell_0$.
	
	If $|k|\le 10\la$ and $|j|>20\la$, then we have
	\[|b_k|\ls \sum_{|j|>20\la}|j|^{-2}|j|^{-n+2-\eta}\ls \la^{-\eta}.\]
	If $|k|\le 10\la$, $|j|\le 20\la$, and $|k|+|j|>\la/2$, then
	\[\begin{aligned}
		|b_k|&=\Big|\sum_{\ell_3\ge\ell_0}\sum_j\beta_{\ell_3}(|j|-|k|) \frac{h(|k|)-h(|j|)}{|k|^2-|j|^2}e_j^0(x)V_{jk}\big|\\
		&\ls \sum_{2^{\ell_3}\ls \la}\sum_{2^m\ls \la}\la^{-1}2^{-\ell_3}2^{(-n+2-\eta)m}\cdot 2^{(n-1)m}2^{\ell_3}\\
		&\ls \la^{-\eta}\log\la.
	\end{aligned}\]
Here we applied the mean value theorem when $||k|-|j||\ls 2^{\ell_0}$.

	If $|k|\le 10\la$, $|j|\le 20\la$, and $|k|+|j|\le\la/2$, then similarly
	\[\begin{aligned}|b_k|&=\Big|\sum_{\ell_3\ge\ell_0}\sum_j\beta_{\ell_3}(|j|-|k|) \frac{h(|k|)-h(|j|)}{|k|^2-|j|^2}e_j^0(x)V_{jk}\big|\\
		&\ls \sum_{2^{\ell_3}\ls \la}\sum_{2^m\ls \la}2^{-\ell_3}\la^{-N}2^{(-n+2-\eta)m}\cdot 2^{(n-1)m}2^{\ell_3}\\
		&\ls \la^{1-\eta-N}\log\la,\ \forall  N.\end{aligned}\]
	If $|k|>10\la$ and $|j|\le 2\la$, then 
	\[|b_k|\ls \sum_{|j|\le2\la}|k|^{-2}|k|^{-n+2-\eta}\ls \la^n|k|^{-n-\eta}.\]
	If $|k|>10\la$, $|j|>2\la$, and $|j|>2|k|$, then
	\[|b_k|\ls \sum_{|j|>2|k|}|j|^{-2}|k|^{-N}|j|^{-n+2-\eta}\ls |k|^{-\eta-N},\ \forall N.\]
	If $|k|>10\la$, $|j|>2\la$, and $|j|<|k|/2$, then
	\[|b_k|\ls \sum_{|j|<|k|/2}|k|^{-2}|j|^{-N}|k|^{-n+2-\eta}\ls |k|^{-n-\eta},\ \forall N.\]
	If $|k|>10\la$, $|j|>2\la$, and $|k|/2\le |j|\le2|k|$, then
	\[\begin{aligned}|b_k|&=\Big|\sum_{\ell_3\ge\ell_0}\sum_j\beta_{\ell_3}(|j|-|k|) \frac{h(|k|)-h(|j|)}{|k|^2-|j|^2}e_j^0(x)V_{jk}\big|\\
		&\ls \sum_{2^{\ell_3}\ls |k|}\sum_{2^m\ls |k|}|k|^{-1}2^{-\ell_3}|k|^{-N}2^{(-n+2-\eta)m}\cdot 2^{(n-1)m}2^{\ell_3}\\
		&\ls |k|^{-\eta-N}\log|k|, \ \forall N.\end{aligned}\]
	Thus the claim \eqref{h.3.7''} is proved.

	For $\ell_1\in \mathbb{N}$ satisfying $\ell_1\ge[\log_2(10\la)]$, we let $I_{\ell_1}=[2^{\ell_1},2^{\ell_1+1})$ if $\ell_1>[\log_2(10\la)]$, and $I_{\ell_1}=[0,10\la)$ if $\ell_1=[\log_2(10\la)]$. Then \eqref{h.3.7''} implies
		\begin{equation}\label{h.3.7'}
		|b_k|\ls  2^{(-n-\eta)\ell_1}\la^n \log\la,\ \ {\rm if}\ |k|\in I_{\ell_1}.
	\end{equation}
\[\]
\noindent $\bullet$ \textbf{Handle $a^0_{jk\tau_\ell}$.}
	We first handle the terms involving $a^0_{jk\tau_\ell}.$ For any fixed $|k|\in I_{\ell_1}$,  we write	
	\begin{equation}\label{hde}
	    \begin{aligned}
	        \frac{\rho_2(|k|/\tau_\ell)}{|k|^2-\tau_\ell^2}&=\rho_2(|k|/\tau_\ell)\int_0^\infty e^{t(|k|^2-\tau_\ell^2)}dt\\
		&=\int_0^{2^{-2\ell_1}} e^{t(|k|^2-\tau_\ell^2)}dt -\big(1-\rho_2(|k|/\tau_\ell)\big)\int_0^{2^{-2\ell_1}} e^{t(|k|^2-\tau_\ell^2)}dt\\ &\quad\quad\quad\quad\quad\quad\quad\quad\quad\quad -\frac{\rho_2(|k|/\tau_\ell)e^{2^{-2\ell_1}(|k|^2-\tau_\ell^2)}}{|k|^2-\tau_\ell^2}\\
		&=c^0_{k,\tau_\ell}-c^1_{k,\tau_\ell}-c^2_{k,\tau_\ell},
	    \end{aligned}
	\end{equation}

	Next, we handle these three terms in Cases 3.1.1, 3.1.2, and 3.1.3, respectively.
	\subsection*{Case 3.1.1} 
	
	In this case, we shall deal with the term $c^0_{k,\tau_\ell}$. Our goal is to show
	\begin{equation}\label{h.3.4}
		\begin{aligned}
			|&\sum_{k}\sum_{j}\sum_{\tau_\ell> 2\la}\int \frac{h(|k|)-h(|j|)}{|k|^2-|j|^2} c^0_{k,\tau_\ell}\,e_j^0(x)V_{jk}\overline{e_k^0(z)}V(z)e_{\tau_\ell}(z) \overline{e_{\tau_\ell}(x)}dz| \\
			&\ls\la^{n-\frac32\eta}\log\la.
		\end{aligned}
	\end{equation}

	We shall divide our task into proving 
	\begin{equation}\label{h.3.5}
		\begin{aligned}
			|&\sum_{k}\sum_{j}\sum_{\tau_\ell< 2\la}\int \frac{h(|k|)-h(|j|)}{|k|^2-|j|^2} c^0_{k,\tau_\ell}\,e_j^0(x)V_{jk}\overline{e_k^0(z)}V(z)e_{\tau_\ell}(z) \overline{e_{\tau_\ell}(x)}dz| \\
			&\ls\la^{n-\frac32\eta}\log\la
		\end{aligned}
	\end{equation}
	as well as 
	\begin{equation}\label{h.3.6}
		\begin{aligned}
			|&\sum_{k}\sum_{j,\tau_\ell}\int \frac{h(|k|)-h(|j|)}{|k|^2-|j|^2} c^0_{k,\tau_\ell}\,e_j^0(x)V_{jk}\overline{e_k^0(z)}V(z)e_{\tau_\ell}(z) \overline{e_{\tau_\ell}(x)}dz| \\
			&\ls\la^{n-2\eta}\log\la.
		\end{aligned}
	\end{equation}
	Note that by \eqref{h.3.7'} and  Sobolev inequality, we have for $0\le t\le 2^{-2\ell_1}$
	\begin{equation}\label{h.3.8}
		\begin{aligned}
			\|\sum_{|k|\in I_{\ell_1}} b_{k}e^{t|k|^2}e_k^0(\cdot)\|_{L^{p_0}(M)}&\ls  2^{\ell_1/2}2^{\sigma(p_0)\ell_1} \|\sum_{|k|\in I_{\ell_1}} b_{k}e^{t|k|^2}e_k^0(\cdot)\|_{L^{2}(M)} \\ &
			\ls 	\la^n2^{(-\frac n2+1-\frac54\eta)\ell_1}\log\la.
		\end{aligned}
	\end{equation}

	In view of \eqref{h.3.8}, by H\"older's inequality, we have for fixed $0\le t\le 2^{-2\ell_1}$
	\begin{equation}\label{h.3.9}
		\begin{aligned}
			|\int \sum_{|k|\in I_{\ell_1}}\sum_{\tau_\ell\le2\la} &b_{k} e^{t|k|^2}e^{-t\tau_\ell^2}\overline{e_k^0(z)}V(z)e_{\tau_\ell}(z) \overline{e_{\tau_\ell}(x)}dz |\\
			\ls & \|V\|_{L^{\frac{n}{2-\eta}}(M)} \|\sum_{|k|\in I_{\ell_1}} b_{k}e^{t|k|^2}e_k^0(\cdot)\|_{L^{p_0}(M)}\cdot\la^{1/2}\la^{\sigma(p_0)}( \sum_{\tau_\ell \le 2\la}|e_{\tau_\ell}(x)|^2 )^{1/2}  \\
			\ls &
				\la^{3n/2+1-\frac14\eta}2^{(-\frac n2+1-\frac54\eta)\ell_1}\log\la.
		\end{aligned}
	\end{equation} Integrating over $0\le t\le 2^{-2\ell_1}$  and  summing over $2^{\ell_1}\gs \la$, we obtain \eqref{h.3.5}.
	
	To prove \eqref{h.3.6}, we shall need the heat kernel bounds in Lemma \ref{heatk}. By \eqref{h.3.7'} and the heat kernel bounds, we have for 
	\begin{align}\label{ker02} \Big| \sum_{|k|\in I_{\ell_1}}\sum_{\tau_\ell}\int_0^{2^{-2\ell_1}} &b_{k} e^{t|k|^2}e^{-t\tau_\ell^2}\overline{e_k^0(z)}e_{\tau_\ell}(z) \overline{e_{\tau_\ell}(x)}dt\Big|\\ \nonumber
		&\ls \int_0^{2^{-2\ell_1}} t^{-\frac n2}e^{-cd_g(z, x)^2/t}dt\cdot\sum_{|k|\in I_{\ell_1}}|b_k| \\\nonumber
		&\ls 
		K(z, x)\cdot 
			\la^n 2^{-\eta\ell_1}\log\la.
	\end{align}
	where
	$$K(z, x)= \begin{cases}
		\log(2+(2^{\ell_1}  d_g(z, x))^{-1})(1+2^{\ell_1} d_g(z, x))^{-N},\ n=2\\\nonumber
		d_g(z, x)^{2-n}(1+2^{\ell_1}  d_g(z, x))^{-N},\quad\quad\quad\quad\quad\ \  n\ge3
	\end{cases}
	$$		 
	It is straightforward to check that $\|K(\cdot, x)\|_{L^{\frac{n}{n-2+\eta}}}\ls 2^{-\eta\ell_1}$. Thus by H\"older's inequality, we have
	\begin{equation}\label{h.3.10}
		\begin{aligned}
			\Big|\int \sum_{|k|\in I_{\ell_1}}\sum_{\tau_\ell}\int_0^{2^{-2\ell_1}} &b_{k} e^{-t\tau_\ell^2}\overline{e_k^0(z)}V(z)e_{\tau_\ell}(z) \overline{e_{\tau_\ell}(x)}dz \Big|\\
			\ls &\|V\|_{L^{\frac{n}{2-\eta}}(M)} \cdot 
				\la^n 2^{-2\eta\ell_1}\log\la.
		\end{aligned}
	\end{equation}
	Summing over $2^{\ell_1}\gs \la$, we obtain \eqref{h.3.6}.
	
	\subsection*{Case 3.1.2} 
	
	In this case, we shall deal with the term $c^1_{k,\tau_\ell}$. Our goal is to show
	\begin{equation}\label{h.3.11}
		\begin{aligned}
			|&\sum_{k}\sum_{j}\sum_{\tau_\ell>2\la}\int \frac{h(|k|)-h(|j|)}{|k|^2-|j|^2} c^1_{k,\tau_\ell}\,e_j^0(x)V_{jk}\overline{e_k^0(z)}V(z)e_{\tau_\ell}(z) \overline{e_{\tau_\ell}(x)}dz| \\
			&\ls\la^{n-\frac32\eta}\log\la.
		\end{aligned}
	\end{equation}
	For $b_k$ defined as in \eqref{h.3.7}, let 
	\begin{equation}\label{h.3.12}
		b_{k,\tau_\ell}=b_k (1-\rho_2(|k|/\tau_\ell)).
	\end{equation}
	Let $|k|\in I_{\ell_1}, \tau_\ell\approx 2^s$, $2\la<2^s\ls2^{\ell_1}$, $|k|\ge0.8\tau_\ell$. Then by \eqref{h.3.7'} we have 
	\begin{equation}\label{h.3.13}
		|b_{k,\tau_\ell}|\ls 
			\la^n2^{(-n-\eta)\ell_1}\log \la
	\end{equation}
	as well as
	\begin{equation}\label{h.3.14}
		|\partial_{\tau_\ell}b_{k,\tau_\ell}|\ls 2^{-s}\cdot
			\la^n2^{(-n-\eta)\ell_1}\log\la.
	\end{equation}
	Thus by Sobolev inequality, we have for $0\le t\le 2^{-2\ell_1}$
	\begin{align*}
		\|\sum_{|k|\in I_{\ell_1}} b_{k, 2^s}e^{t|k|^2}e_k^0(\cdot)\|_{L^{p_0}(M)} &\ls 2^{\ell_1/2}2^{\sigma(p_0)\ell_1}	\|\sum_{|k|\in I_{\ell_1}} b_{k, 2^s}e^{t|k|^2}e_k^0(\cdot)\|_{L^{2}(M)} \\
		&\ls  2^{(n/2+1-\frac14\eta)\ell_1}\cdot \sup_{k,\tau_\ell}|b_{k,\tau_\ell}|.
	\end{align*}
	Similarly, we may obtain the estimates of $\partial_\tau b_{k,\tau}$.
	By applying Lemma~\ref{delta1} with $\delta=2^s$, we have for $0\le t\le 2^{-2\ell_1}$,
	\begin{align*}
		|&\int \sum_{\tau_\ell\approx 2^s }\sum_{|k|\in I_{\ell_1}} b_{k,\tau_\ell}e^{t|k|^2}e^{-t\tau_\ell^2} \overline{e_k^0(z)}V(z)e_{\tau_\ell}(z) \overline{e_{\tau_\ell}(x)}dz |\\
		\ls & \|V\|_{L^{\frac{n}{2-\eta}}(M)} ( \sum_{\tau_\ell\approx 2^s}|{e_{\tau_\ell}(x)}|^2 )^{1/2}\cdot 2^{s/2}2^{\sigma(p_0)s}\cdot 2^{(n/2+1-\frac14\eta)\ell_1}\cdot 
			\la^n2^{(-n-\eta)\ell_1}\log\la\\
		\ls &2^{(n/2+1-\frac14\eta)s} 2^{(n/2+1-\frac14\eta)\ell_1}\cdot 
			\la^n2^{(-n-\eta)\ell_1}\log\la.
	\end{align*}
	Note that $2\la\le 2^s\ls 2^{\ell_1}$. Integrating over $0\le t\le 2^{-2\ell_1}$ and summing over $\ell_1,s$, we obtain \eqref{h.3.11}.

	\subsection*{Case 3.1.3} 
	
	In this case, we shall deal with the term $c^2_{k,\tau_\ell}$. Our goal is to show
	\begin{equation}\label{h.3.16}
		\begin{aligned}
			|&\sum_{k}\sum_{j}\sum_{\tau_\ell> 2\la}\int \frac{h(|k|)-h(|j|)}{|k|^2-|j|^2} c^2_{k,\tau_\ell}\,e_j^0(x)V_{jk}\overline{e_k^0(z)}V(z)e_{\tau_\ell}(z) \overline{e_{\tau_\ell}(x)}dz| \\
			&\ls\la^{n-\frac32\eta}\log\la.
		\end{aligned}
	\end{equation}
	For $b_k$ defined as in \eqref{h.3.7}, let 
	\begin{equation}\label{h.3.17}
		b_{k,\tau_\ell}=\frac{b_k \rho_2(|k|/\tau_\ell)}{|k|^2-\tau_\ell^2}.
	\end{equation}
	Let $|k|\in I_{\ell_1}, \tau_\ell\approx 2^s$, $2\la<2^s$, $2^{\ell_1}\ls 2^s$, $|k|\le0.9\tau_\ell$. 
	Then by \eqref{h.3.7'} we have 
	\begin{equation}\label{h.3.13'}
		|b_{k,\tau_\ell}|\ls 2^{-2s}\cdot
			\la^n2^{(-n-\eta)\ell_1}\log\la
	\end{equation}
	as well as
	\begin{equation}\label{h.3.14'}
		|\partial_{\tau_\ell}b_{k,\tau_\ell}|\ls 2^{-3s}\cdot
			\la^n2^{(-n-\eta)\ell_1}\log\la.
	\end{equation}
	Thus by Sobolev inequality, we have
	\begin{align*}
		\|\sum_{|k|\in I_{\ell_1}} b_{k, 2^s}e^{2^{-2\ell_1}|k|^2}e_k^0(\cdot)\|_{L^{p_0}(M)}&\ls 2^{\ell_1/2}2^{\sigma(p_0)\ell_1}	\|\sum_{|k|\in I_{\ell_1}} b_{k, 2^s}e^{2^{-2\ell_1}|k|^2}e_k^0(\cdot)\|_{L^{2}(M)} \\
		&\ls  2^{(n/2+1-\frac14\eta)\ell_1}\cdot\sup_{k,\tau_\ell}|b_{k,\tau_\ell}|
	\end{align*}
	Similarly, we may obtain the estimates of $\partial_\tau b_{k,\tau}$.
	By applying Lemma~\ref{delta1} with $\delta=2^{s}$, we have for 
	\begin{align*}
		|\int &\sum_{\tau_\ell\approx 2^s}\sum_{|k|\in I_{\ell_1}} b_{k,\tau_\ell}e^{2^{-2\ell_1}|k|^2}e^{-2^{-2\ell_1}\tau_\ell^2} \overline{e_k^0(z)}V(z)e_{\tau_\ell}(z) \overline{e_{\tau_\ell}(x)}dz |\\
		\ls & \|V\|_{L^{\frac{n}{2-\eta}}(M)} (\sum_{\tau_\ell\approx 2^s}|{e_{\tau_\ell}(x)}|^2 )^{1/2}2^{s/2}2^{\sigma(p_0)s} e^{-2^{-2\ell_1}2^{2s}}\cdot 2^{(n/2+1-\frac14\eta)\ell_1}2^{-2s}\\
		&\,\,\cdot
			\la^n2^{(-n-\eta)\ell_1}\log\la \\
		\ls & (2^{-\ell_1}2^s)^{-N}2^{(n/2+1-\frac14\eta)\ell_1}2^{(n/2-1-\frac14\eta)s}\cdot
			\la^n2^{(-n-\eta)\ell_1}\log\la,\ \forall N. 
	\end{align*}
	Note that $\la\ls 2^{\ell_1}\ls 2^s$. Summing over $\ell_1,s$ with $N$ large enough, we obtain \eqref{h.3.16}.
\[\]
\noindent $\bullet$ \textbf{Handle $a^1_{jk\tau_\ell}$.}	Now we shall handle the terms involving $a^1_{jk\tau_\ell}.$ It suffices to show \begin{equation}\label{h.3.22}
		\begin{aligned}
			|\sum_{j, k}&\sum_{\tau_\ell\ge 2\la}(1-\rho_2(|j|/\tau_\ell))\rho_2(|k|/\tau_\ell)
			\cdot \int a^1_{jk\tau_\ell}\,e_j^0(x)V_{jk}\overline{e_k^0(z)}V(z)e_{\tau_\ell}(z) \overline{e_{\tau_\ell}(x)}dz| \\
			\ls &\la^{-N},\,\,\,\forall N,
		\end{aligned}
	\end{equation}
and 
	\begin{equation}\label{h.3.23}
		\begin{aligned}
			|\sum_{j, k}&\sum_{\tau_\ell\ge 2\la}\rho_2(|j|/\tau_\ell)\rho_2(|k|/\tau_\ell)
			\cdot \int a^1_{jk\tau_\ell}\,e_j^0(x)V_{jk}\overline{e_k^0(z)}V(z)e_{\tau_\ell}(z) \overline{e_{\tau_\ell}(x)}dz| \\
			\ls &\la^{n-\frac32\eta}.
		\end{aligned}
	\end{equation}First,  let
\[
		b_{k,\tau_\ell}=	\sum_j a_{jk\tau_\ell}^1e_j^0(x)V_{jk}(1-\rho_2(|j|/\tau_\ell))\rho_2(|k|/\tau_\ell).\]
	In this case, $|j|\ge 0.8\tau_\ell$ and $|k|\le0.9\tau_\ell$, so $|j|>|k|/2$. We claim that 
	\begin{equation}\label{h.3.21}
		|b_{k,\tau_\ell}|\ls \tau_\ell^{-N},\ \forall N.
	\end{equation}
	Indeed, if $|j|>2|k|$, we have
	\[|b_{k,\tau_\ell}|\ls \sum_{|j|>0.8\tau_\ell}\tau_\ell^{-2}|j|^{-1}\tau_\ell^{-N}|j|^{-n+2-\eta}\ls \tau_\ell^{-N}.\]
	If $|j|\le2|k|$, we have $|j|\approx |k|\approx \tau_\ell$, and then
	\[|b_{k,\tau_\ell}|\ls \sum_{2^m\ls \tau_\ell}\tau_\ell^{-3}\tau_\ell^{-N}2^{(-n+2-\eta)m}2^{nm}\ls \tau_\ell^{-N}.\]
	Thus the claim \eqref{h.3.21} is proved.
	Then by rough eigenfunction bounds \eqref{rough}
	\begin{align*}
		\Big|\sum_{|k|<0.9\tau_\ell}\sum_{\tau_\ell> 2\la}&\int b_{k,\tau_\ell}\overline{e_k^0(z)}V(z)e_{\tau_\ell}(z) \overline{e_{\tau_\ell}(x)}dz\Big|\\
		&\ls \|V\|_{L^1(M)}\sup_{z,x}\sum_{\tau_\ell>2\la} \tau_\ell^{-N}\cdot \tau_\ell^n|e_{\tau_\ell}(z) \overline{e_{\tau_\ell}(x)}|\\
		&\ls \sum_{2^m>2\la} 2^{-Nm}\cdot 2^{2nm}\\
		&\ls \la^{2n-N},\ \forall N.
	\end{align*}
	So we obtain \eqref{h.3.22}.
	
	Next, to prove \eqref{h.3.23}, we write
	\begin{equation}\label{h.3.24}\nonumber
		\begin{aligned}
			a^1_{jk\tau_\ell}=&\frac{h(\tau_\ell)}{(\tau_\ell^2-|j|^2)(\tau_\ell^2-|k|^2)}-\frac{h(|j|)}{(\tau_\ell^2-|j|^2)(\tau_\ell^2-|k|^2)}\\
			=&c^0_{jk\tau_\ell}-c^1_{jk\tau_\ell}.
		\end{aligned}
	\end{equation}
	The  terms involving $c_{jk\tau_\ell}^0$ satisfies the same bound as \eqref{h.3.22}. Indeed, if 
	\[b_{k,\tau_\ell}=	\sum_j c_{jk\tau_\ell}^0e_j^0(x)V_{jk}\rho_2(|j|/\tau_\ell)\rho_2(|k|/\tau_\ell)\]
	then by direct calculation, it still satisfies the bound in \eqref{h.3.21}. So we can still get the bound $\la^{-N}$ by the argument above. Thus, it suffices to consider the terms involving $c^1_{jk\tau_\ell}$. We need to show
	\begin{equation}\label{h.3.23'}
		\begin{aligned}
			|\sum_{j, k}&\sum_{\tau_\ell\ge 2\la}\rho_2(|j|/\tau_\ell)\rho_2(|k|/\tau_\ell)
			\cdot \int c^1_{jk\tau_\ell}\,e_j^0(x)V_{jk}\overline{e_k^0(z)}V(z)e_{\tau_\ell}(z) \overline{e_{\tau_\ell}(x)}dz| \\
			\ls &\la^{n-\frac32\eta}.
		\end{aligned}
	\end{equation}
	Note that in this case, we have $|j|,|k|\le0.9\tau_\ell$ on the support of $\rho_2$. We write for $N=1,2,...$,
	\begin{equation}\label{3.2.1.1}
	    	\frac{1}{\tau_\ell^2-|j|^2}=\tau_\ell^{-2}
	+\tau_\ell^{-2}\bigl(|j|/\tau_\ell\bigr)^2+
	\cdots + \tau_\ell^{-2}\bigl(|j|/\tau_\ell\bigr)^{2N-2}
	\\
	+(|j|/\tau_\ell)^{2N}\frac{1}{\tau_\ell^2-|j|^2}
	\end{equation}
	and similarly
\begin{equation}\label{3.2.1.2}
    	\frac{1}{\tau_\ell^2-|k|^2}=\tau_\ell^{-2}
	+\tau_\ell^{-2}\bigl(|k|/\tau_\ell\bigr)^2+
	\cdots + \tau_\ell^{-2}\bigl(|k|/\tau_\ell\bigr)^{2N-2}
	\\
	+(|k|/\tau_\ell)^{2N}\frac{1}{\tau_\ell^2-|k|^2}.
\end{equation}
Then we split the product of \eqref{3.2.1.1} and \eqref{3.2.1.2} into three cases:
\begin{align*}
	\frac{1}{(\tau_\ell^2-|j|^2)(\tau_\ell^2-|k|^2)}=\frac{(|j|/\tau_\ell)^{2N}}{(\tau_\ell^2-|j|^2)(\tau_\ell^2-|k|^2)}+\sum_{\mu=0}^{N-1}\frac{|j|^{2\mu}(|k|/\tau_\ell)^{2N}}{\tau_\ell^{2+2\mu}(\tau_\ell^2-|k|^2)}+\sum_{\mu_1=0}^{N-1}\sum_{\mu_2=0}^{N-1}\frac{|j|^{2\mu_2}|k|^{2\mu_1}}{\tau_\ell^{2(\mu_1+\mu_2+2)}}
\end{align*}	
See Cases 3.2.1, 3.2.2, 3.2.3 in the following.
	\subsection*{Case 3.2.1} 
	We need to show 
	\begin{multline}\label{h.3.2.1}\Big|\sum_{j,k}\sum_{\tau_\ell> 2\la}\rho_2(|j|/\tau_\ell)\rho_2(|k|/\tau_\ell)\frac{h(|j|)(|j|/\tau_\ell)^{2N}}{(\tau_\ell^2-|j|^2)(\tau_\ell^2-|k|^2)}e_j^0(x)V_{jk}\overline{e_k^0(z)}V(z)e_{\tau_\ell}(z) \overline{e_{\tau_\ell}(x)}dz\Big| \\ \ls\la^{n-\frac32\eta}.\end{multline}
	Let  $|k|\in I_{\ell_1}$, $|j|\in I_{\ell_2}$, $\tau_\ell\approx 2^s$, $2^s>2\la$. Here $\ell_1,\ell_2\ge [\log_2(10\la)]$. See the definition of $I_{\ell_1}$ before \eqref{h.3.7'}. The definition of $I_{\ell_2}$ is the same. Let
	\[b_{k,\tau_\ell}=\sum_{|j|\in I_{\ell_2}}\rho_2(|j|/\tau_\ell)\rho_2(|k|/\tau_\ell)\frac{h(|j|)(|j|/\tau_\ell)^{2N}}{(\tau_\ell^2-|j|^2)(\tau_\ell^2-|k|^2)}e_j^0(x)V_{jk}.\]
	Note that \[|h(|j|)|\ls \begin{cases}1,\ \ \ \ \ \ \ \ \ \ \ \ \ \ 2^{\ell_2}\approx10\la\\
		2^{-N_1\ell_2},\  \ \ \ \ \ \ 2^{\ell_2}>10\la,\ \forall N_1.
	\end{cases}\]
	Then we get
	\begin{align*}|b_{k,\tau_\ell}|&\ls \sum_{2^m\le 2^s}2^{2N(\ell_2-s)}2^{-4s}\cdot 2^{(-n+2-\eta)m}\min(2^{nm},2^{n\ell_2})\cdot \begin{cases}1,\ \ \ \ \ \ \ \ \ \ \ \ \ \  2^{\ell_2}\approx10\la\\
			2^{-N_1\ell_2},\ \ \ \ \ \ \ 2^{\ell_2}>10\la
		\end{cases}\\
		&\ls 2^{(-2N-4)s}2^{(2N+2-\eta)\ell_2}\cdot \begin{cases}1,\ \ \ \ \ \ \ \ \ \ \ \ \ \  2^{\ell_2}\approx10\la\\
			2^{-N_1\ell_2},\ \ \ \ \ \  2^{\ell_2}>10\la.
	\end{cases}\end{align*}
	And similarly,
	\begin{align*}|\partial_{\tau_\ell}b_{k,\tau_\ell}|
		&\ls 2^{(-2N-5)s}2^{(2N+2-\eta)\ell_2}\cdot \begin{cases}1,\ \ \ \ \ \ \ \ \ \ \ \ \ \  2^{\ell_2}\approx10\la\\
			2^{-N_1\ell_2},\ \ \ \ \ \ \ 2^{\ell_2}>10\la.
	\end{cases}\end{align*}
	By Sobolev inequality,
	\begin{align*}
		\|\sum_{|k|\in I_{\ell_1}}& b_{k, 2^s}e_k^0(\cdot)\|_{L^{p_0}(M)}\ls 2^{\ell_1/2}2^{\sigma(p_0)\ell_1}	\|\sum_{|k|\in I_{\ell_1}} b_{k, 2^s}e_k^0(\cdot)\|_{L^{2}(M)} \\
		&\ls  2^{(n/2+1-\frac14\eta)\ell_1}\cdot\sup_{k,\tau_\ell}|b_{k,\tau_\ell}|.
	\end{align*}
	And similarly we may obtain the estimate for  $\partial_\tau b_{k,\tau}$.
	
	By applying Lemma~\ref{delta1} with $\delta=2^{s}$, we have 	\begin{align*}
		|\int \sum_{\tau_\ell\approx 2^s}&\sum_{|k|\in I_{\ell_1}} b_{k,\tau_\ell} \overline{e_k^0(z)}V(z)e_{\tau_\ell}(z) \overline{e_{\tau_\ell}(x)}dz |\\
		\ls & \|V\|_{L^{\frac{n}{2-\eta}}(M)}2^{ns/2}2^{s/2}2^{\sigma(p_0)s}\cdot 2^{(-2N-4)s}2^{(2N+2-\eta)\ell_2}2^{(\frac n2+1-\frac14\eta)\ell_1}\\
		&\,\,\cdot\begin{cases}1,\ \ \ \ \ \ \ \ \ \ \ \ \ \  2^{\ell_2}\approx10\la\\
			2^{-N_1\ell_2},\ \ \ \ \ \ \ 2^{\ell_2}>10\la\end{cases}\\
		\ls& 2^{(\frac n2-\frac14\eta-2N-3)s}2^{(2N+2-\eta)\ell_2}2^{(\frac n2+1-\frac14\eta)\ell_1}\cdot\begin{cases}1,\ \ \ \ \ \ \ \ \ \ \ \ \ \ 2^{\ell_2}\approx10\la\\
			2^{-N_1\ell_2},\ \ \ \ \ \ \ 2^{\ell_2}>10\la\end{cases}
	\end{align*}
	Fix $N\ge n$ and $N_1\ge2N+2$. Note that $\la\ls 2^{\ell_1},2^{\ell_2}\ls 2^s$. Summing over $\ell_1,\ell_2,s$, we obtain \eqref{h.3.2.1}.

	\subsection*{Case 3.2.2} 
	
	For $\mu=0,1,...,N-1$, we need to show 
	\begin{multline}\label{h.3.2.2}\Big|\sum_{j,k}\sum_{\tau_\ell> 2\la}\rho_2(|j|/\tau_\ell)\rho_2(|k|/\tau_\ell)\frac{h(|j|)|j|^{2\mu}(|k|/\tau_\ell)^{2N}}{\tau_\ell^{2+2\mu}(\tau_\ell^2-|k|^2)}e_j^0(x)V_{jk}\overline{e_k^0(z)}V(z)e_{\tau_\ell}(z) \overline{e_{\tau_\ell}(x)}dz\Big|\\ \ls\la^{n-\frac32\eta}.
	\end{multline}
	Let  $|k|\in I_{\ell_1}$, $|j|\in I_{\ell_2}$, $\tau_\ell\approx 2^s$, $2^s>2\la$. As before,  $\ell_1,\ell_2\ge [\log_2(10\la)]$.
	
Let
	\[b_{k,\tau_\ell}=\sum_{|j|\in I_{\ell_2}}\rho_2(|j|/\tau_\ell)\rho_2(|k|/\tau_\ell)\frac{h(|j|)|j|^{2\mu}\tau_\ell^{-2-2\mu}(|k|/\tau_\ell)^{2N}}{\tau_\ell^2-|k|^2}e_j^0(x)V_{jk}.\]
We have
	\begin{equation}
	    \begin{aligned}
	      	|b_{k,\tau_\ell}|\ls & 2^{(-2N-4-2\mu)s}2^{2N\ell_1}2^{(2-\eta+2\mu)\ell_2} \\
	      	&\quad\cdot\begin{cases}1,\ \ \ \ \ \ \ \ \ \ \ \ \ \  2^{\ell_1}\le 80\la\\
		2^{-N_1\ell_2},\ \ \ \ \ \ \ \ \ 2^{\ell_1}>80\la\ {\rm and}\ 2^{\ell_2}>2^{\ell_1-2}\\
		2^{(n-2+\eta)(\ell_2-\ell_1)},\ \ 2^{\ell_1}>80\la\ {\rm and}\ 2^{\ell_2}\approx 10\la\\
		2^{(n-2+\eta)(\ell_2-\ell_1)}2^{-N_1\ell_2},\ \ 2^{\ell_1}>80\la\ {\rm and}\ 10\la< 2^{\ell_2}\le  2^{\ell_1-2}.
	\end{cases}  
	    \end{aligned}
	\end{equation}
And also by Sobolev inequality
	\begin{align*}
		\|\sum_{|k|\in I_{\ell_1}}& b_{k, 2^s}e_k^0(\cdot)\|_{L^{p_0}(M)}\ls 2^{\ell_1/2}2^{\sigma(p_0)\ell_1}	\|\sum_{|k|\in I_{\ell_1}} b_{k, 2^s}e_k^0(\cdot)\|_{L^{2}(M)} \\
		&\ls  2^{(n/2+1-\frac14\eta)\ell_1}\cdot\sup_{k,\tau_\ell}|b_{k,\tau_\ell}|.
	\end{align*}
	Similarly, we may obtain the estimates for $\partial_\tau b_{k,\tau}$.
	By applying Lemma~\ref{delta1} with $\delta=2^{s}$, we have 	\begin{align*}
		|\int& \sum_{\tau_\ell\approx 2^s}\sum_{|k|\in I_{\ell_1}} b_{k,\tau_\ell} \overline{e_k^0(z)}V(z)e_{\tau_\ell}(z) \overline{e_{\tau_\ell}(x)}dz | \\
		&\ls \|V\|_{L^{\frac{n}{2-\eta}}(M)}2^{ns/2}2^{s/2}2^{(\frac12-\frac14\eta)s}\cdot2^{(-2N-4-2\mu)s}2^{(2-\eta+2\mu)\ell_2}\\
		&\quad \cdot 2^{(2N+\frac n2+1-\frac14\eta)\ell_1} \cdot\begin{cases}1,\ \ \ \ \ \ \ \ \ \ \ \ \ \  2^{\ell_1}\le 80\la\\
			2^{-N_1\ell_2},\ \ \ \ \ \ \ 2^{\ell_1}>80\la\ {\rm and}\ 2^{\ell_2}>2^{\ell_1-2}\\
			2^{(n-2+\eta)(\ell_2-\ell_1)},\ \ 2^{\ell_1}>80\la\ {\rm and}\ 2^{\ell_2}\approx 10\la\\
			2^{(n-2+\eta)(\ell_2-\ell_1)}2^{-N_1\ell_2},\ \ 2^{\ell_1}>80\la\ {\rm and}\ 10\la< 2^{\ell_2}\le  2^{\ell_1-2}.
		\end{cases}
	\end{align*}
	Fix $N\ge n$ and $N_1\ge2N+2n$. Note that $\la\ls 2^{\ell_1},2^{\ell_2}\ls 2^s$. Summing over $\ell_1,\ell_2,s$, we obtain \eqref{h.3.2.2}.
	\subsection*{Case 3.2.3} 
	
	For $\mu_1,\mu_2=0,1,...,N-1$, we need to show 
	\begin{multline}
	    \Big|\sum_{j,k}\sum_{\tau_\ell> 2\la}\rho_2(|j|/\tau_\ell)\rho_2(|k|/\tau_\ell)\frac{h(|j|)|j|^{2\mu_2}|k|^{2\mu_1}}{\tau_\ell^{2(\mu_1+\mu_2+2)}}e_j^0(x)V_{jk}\overline{e_k^0(z)}V(z)e_{\tau_\ell}(z) \overline{e_{\tau_\ell}(x)}dz\Big| \\ \ls\la^{n-\frac32\eta}.
	\end{multline}
	For any fixed $k\in I_{\ell_1}$ (see the definition of $I_{\ell_1}$ before \eqref{h.3.7'}), we write	
	\begin{equation}\label{hde1}
	    \begin{aligned}
	        \rho_2(|j|/\tau_\ell)\rho_2(|k|/\tau_\ell) &\tau_\ell^{-2(\mu_1+\mu_2+2)}\\
	       & =\rho_2(|j|/\tau_\ell)\rho_2(|k|/\tau_\ell) c_0\int_0^\infty t^{\mu_1+\mu_2+1}e^{-t\tau_k^2}dt\\
		&=c_0\int_0^{2^{-2\ell_1}} t^{\mu_1+\mu_2+1}e^{-t\tau_\ell^2}dt \\
		&\quad-\big(1-\rho_2(|j|/\tau_\ell)\rho_2(|k|/\tau_\ell)\big)c_0\int_0^{2^{-2\ell_1}} t^{\mu_1+\mu_2+1}e^{-t\tau_\ell^2}dt\\ &\quad\quad\quad\quad\quad-\frac{\rho_2(|j|/\tau_\ell)\rho_2(|k|/\tau_\ell)w(\tau_\ell^22^{-2\ell_1})}{2^{2\ell_1(\mu_1+\mu_2+2)}},\\
		&=c^0_{k,\tau_\ell}-c^1_{k,\tau_\ell}-c^2_{k,\tau_\ell}.
	    \end{aligned}
	\end{equation}
 Here
	$c_0$ is a fixed normalizing constant and 
	\begin{equation}\nonumber
		\begin{aligned}
			w(x)=c_0\int_1^{\infty} t^{\mu_1+\mu_2+1}e^{-tx}dt
			\ls e^{-x/2}.
		\end{aligned}
	\end{equation}
Now we handle these three terms in Cases 3.2.3a, 3.2.3b, 3.2.3c respectively.
	
	\subsection*{Case 3.2.3a}  We need to show
	\begin{equation}\label{h.4.1}
		\begin{aligned}
			|&\sum_{j,k}\sum_{\tau_\ell> 2\la}\int h(|j|)|j|^{2\mu_2}|k|^{2\mu_1}c^0_{k,\tau_\ell}\,e_j^0(x)V_{jk}\overline{e_k^0(z)}V(z)e_{\tau_\ell}(z) \overline{e_{\tau_\ell}(x)}dz| \\
			&\ls\la^{n-\frac32\eta}.
		\end{aligned}
	\end{equation}
	We shall divide our task into proving 
	\begin{equation}\label{h.4.2}
		\begin{aligned}
			|&\sum_{j,k}\sum_{\tau_\ell\le2\la}\int h(|j|)|j|^{2\mu_2}|k|^{2\mu_1}c^0_{k,\tau_\ell}\,e_j^0(x)V_{jk}\overline{e_k^0(z)}V(z)e_{\tau_\ell}(z) \overline{e_{\tau_\ell}(x)}dz| \\
			&\ls\la^{n-\frac32\eta}
		\end{aligned}
	\end{equation}
	as well as 
	\begin{equation}\label{h.4.3}
		\begin{aligned}
			|&\sum_{j,k}\sum_{\tau_\ell}\int h(|j|)|j|^{2\mu_2}|k|^{2\mu_1}c^0_{k,\tau_\ell}\,e_j^0(x)V_{jk}\overline{e_k^0(z)}V(z)e_{\tau_\ell}(z) \overline{e_{\tau_\ell}(x)}dz| \\
			&\ls\la^{n-2\eta}.
		\end{aligned}
	\end{equation}
		Let
	\[b_{k}=\sum_{|j|\in I_{\ell_2}}h(|j|)|j|^{2\mu_2}|k|^{2\mu_1}e_j^0(x)V_{jk}.\]
	Then for $|k|\in I_{\ell_1}$, we have
	\begin{align}\label{h.4.5}|b_k|\ls  2^{2\mu_1\ell_1}2^{(2\mu_2+2-\eta)\ell_2}\cdot\begin{cases}1,\ \ \ \ \ \ \ \ \ \ \ \ \ \  2^{\ell_1}\le 80\la\ {\rm and}\ 2^{\ell_2}\approx 10\la\\
			2^{-N_1\ell_2},\ \ \ \ \ \ \ 2^{\ell_1}\le 80\la\ {\rm and}\ 2^{\ell_2}>10\la	\\
			2^{-N_1\ell_2},\ \ \ \ \ \ \ 2^{\ell_1}>80\la\ {\rm and}\ 2^{\ell_2}>2^{\ell_1-2}\\
			2^{(n-2+\eta)(\ell_2-\ell_1)},\ \ 2^{\ell_1}>80\la\ {\rm and}\ 2^{\ell_2}\approx 10\la\\
			2^{(n-2+\eta)(\ell_2-\ell_1)}2^{-N_1\ell_2},\ \ 2^{\ell_1}>80\la\ {\rm and}\ 10\la< 2^{\ell_2}\le  2^{\ell_1-2}.
	\end{cases}\end{align}
	By using Sobolev inequality,
	\begin{align*}
		\|&\sum_{|k|\in I_{\ell_1}} b_{k}e_k^0(\cdot)\|_{L^{p_0}(M)}\ls  2^{\ell_1/2}2^{\sigma(p_0)\ell_1} \|\sum_{|k|\in I_{\ell_1}} b_{k}e_k^0(\cdot)\|_{L^{2}(M)} \\ &
		\ls 2^{(n/2+1-\frac14\eta)\ell_1}\cdot \sup_{k}|b_k|.
	\end{align*}
	Then by H\"older's inequality, we have for fixed $0\le t\le 2^{-2\ell_1}$
	\begin{align*}
		|\int& \sum_{|k|\in I_{\ell_1}}\sum_{\tau_\ell\le2\la} b_{k}t^{\mu_1+\mu_2+1} e^{-t\tau_\ell^2}\overline{e_k^0(z)}V(z)e_{\tau_\ell}(z) \overline{e_{\tau_\ell}(x)}dz |\\
		\ls & \|V\|_{L^{\frac{n}{2-\eta}}(M)} \|\sum_{|k|\in I_{\ell_1}} b_{k}e_k^0(\cdot)\|_{L^{p_0}(M)}\cdot2^{-2\ell_1(\mu_1+\mu_2+1)} \la^{1/2}\la^{\sigma(p_0)}( \sum_{\tau_\ell\in  [1, 2\la]}|e_{\tau_\ell}(x)|^2 )^{1/2}  \\
		\ls &\la^{\frac n2+1-\frac14\eta}2^{(\frac n2-1-\frac14\eta-2\mu_2)\ell_1}2^{(2\mu_2+2-\eta)\ell_2}\\
		&\quad \cdot\begin{cases}1,\ \ \ \ \ \ \ \ \ \ \ \ \ \  2^{\ell_1}\le 80\la\ {\rm and}\ 2^{\ell_2}\approx 10\la\\
			2^{-N_1\ell_2},\ \ \ \ \ \ \  2^{\ell_1}\le 80\la\ {\rm and}\ 2^{\ell_2}>10\la	\\
			2^{-N_1\ell_2},\ \ \ \ \ \ \  2^{\ell_1}>80\la\ {\rm and}\ 2^{\ell_2}>2^{\ell_1-2}\\
			2^{(n-2+\eta)(\ell_2-\ell_1)},\ \ 2^{\ell_1}>80\la\ {\rm and}\ 2^{\ell_2}\approx 10\la\\
			2^{(n-2+\eta)(\ell_2-\ell_1)}2^{-N_1\ell_2},\ \ 2^{\ell_1}>80\la\ {\rm and}\ 10\la< 2^{\ell_2}\le  2^{\ell_1-2}.
		\end{cases}
	\end{align*}
	Fix $N_1\ge 2N+2n$. Integrating over $0\le t\le 2^{-2\ell_1}$ and summing over $2^{\ell_1},2^{\ell_2}\gs\la$, we obtain \eqref{h.4.2}.
	
	To prove \eqref{h.4.3}, by  the heat kernel bounds in Lemma~\ref{heatk}, we have
	\begin{align}\label{ker03}\nonumber
		|\sum_{|k|\in I_{\ell_1}}\sum_{\tau_\ell}\int_0^{2^{-2\ell_1}} &b_{k}\, t^{\mu_1+\mu_2+1}e^{-t\tau_\ell^2}\overline{e_k^0(z)}e_{\tau_\ell}(z) \overline{e_{\tau_\ell}(x)}dt|\\ \nonumber
		&\ls  \int_0^{2^{-2\ell_1}}t^{\mu_1+\mu_2+1} t^{-\frac n2}e^{-cd_g(z, x)^2/t}dt\cdot \sum_{|k|\in I_{\ell_1}}|b_k|\\\nonumber
		&\ls K(z, x)\cdot 2^{n\ell_1}\cdot \sup_{k}|b_k|\end{align}
	where 
	$$K(z, x)=\begin{cases}
		\log(2+(2^{\ell_1} d_g(z, x))^{-1})(1+2^{\ell_1} d_g(z, x))^{-N},\,\,\,\text{if}\,\,\,2\mu_1+2\mu_2+4=n \\
		2^{-(2\mu_1+2\mu_2+4-n)\ell_1}(1+2^{\ell_1} d_g(z, x))^{-N}, \,\,\, \ \ \ \ \  \ \ \text{if}\,\,\, 2\mu_1+2\mu_2+4> n \\
			d_g(z, x)^{2\mu_1+2\mu_2+4-n}(1+2^{\ell_1} d_g(z, x))^{-N}, \,\,\, \ \ \ \ \  \text{if}\,\,\, 2\mu_1+2\mu_2+4< n.
	\end{cases}
	$$				
	It is straightforward to check that 
	$$\|K(\cdot, x)\|_{L^{\frac{n}{n-2+\eta}}}\ls 2^{-(2\mu_1+2\mu_2+2+\eta)\ell_1},$$ 
	then by H\"older's inequality and \eqref{h.4.5}, we have
	\begin{align*}
		|\int \sum_{|k|\in I_{\ell_1}}\sum_{\tau_\ell}&\int_0^{2^{-2\ell_1}}b_{k}\, t^{\mu_1+\mu_2+1}e^{-t\tau_\ell^2}\overline{e_k^0(z)}V(z)e_{\tau_\ell}(z) \overline{e_{\tau_\ell}(x)}dtdz |\\
		\ls &\|V\|_{L^{\frac{n}{2-\eta}}(M)}  2^{(n-2\mu_2-2-\eta)\ell_1}2^{(2\mu_2+2-\eta)\ell_2} \\
		&\quad\cdot\begin{cases}1,\ \ \ \ \ \ \ \ \ \ \ \ \ \  2^{\ell_1}\le 80\la\ {\rm and}\ 2^{\ell_2}\approx 10\la\\
			2^{-N_1\ell_2},\ \ \ \ \ \ 2^{\ell_1}\le 80\la\ {\rm and}\ 2^{\ell_2}>10\la	\\
			2^{-N_1\ell_2},\ \ \ \ \ \ 2^{\ell_1}>80\la\ {\rm and}\ 2^{\ell_2}>2^{\ell_1-2}\\
			2^{(n-2+\eta)(\ell_2-\ell_1)},\ \ 2^{\ell_1}>80\la\ {\rm and}\ 2^{\ell_2}\approx 10\la\\
			2^{(n-2+\eta)(\ell_2-\ell_1)}2^{-N_1\ell_2},\ \ 2^{\ell_1}>80\la\ {\rm and}\ 10\la< 2^{\ell_2}\le  2^{\ell_1-2}.
		\end{cases}
	\end{align*}
	Fix $N_1\ge 2N+2n$. Summing over  $2^{\ell_1},2^{\ell_2}\gs \la$, we obtain \eqref{h.4.3}.

	\subsection*{Case 3.2.3b}  We need to show 
	\begin{equation}\label{h.4.8}
		|\sum_{j,k}\sum_{\tau_\ell> 2\la}\int h(|j|)|j|^{2\mu_2}|k|^{2\mu_1}c^1_{k,\tau_\ell}\,e_j^0(x)V_{jk}\overline{e_k^0(z)}V(z)e_{\tau_\ell}(z) \overline{e_{\tau_\ell}(x)}dz| \ls\la^{n-\frac32\eta}.
	\end{equation}
	In this case, we have $|j|\ge 0.8\tau_\ell$ or $|k|\ge 0.8\tau_\ell$ on the support of $1-\rho_2(|k|/\tau_\ell)\rho_2(|j|/\tau_\ell)$. Let $\tau_\ell\in[2^s,2^{s+1}]$, $2^s>2\la$,
	and for $|j|\in I_{\ell_2}$, let
	\begin{equation}\label{h.4.9}
		b_{k,\tau_\ell}=\sum_{|j|\in I_{\ell_2}} h(|j|)|j|^{2\mu_2}|k|^{2\mu_1}e_j^0(x)V_{jk}(1-\rho_2(|k|/\tau_\ell)\rho_2(|j|/\tau_\ell)).
	\end{equation}
	Then for $|k|\in I_{\ell_1}$, we have
	\begin{align}\label{h.4.10}|b_{k,\tau_\ell}|\ls  2^{2\mu_1\ell_1}2^{(2\mu_2+2-\eta)\ell_2}\cdot\begin{cases}1,\ \ \ \ \ \ \ \ \ \ \ \ \ \  2^{\ell_1}\le 80\la\ {\rm and}\ 2^{\ell_2}\approx 10\la\\
			2^{-N_1\ell_2},\ \ \ \ \ \ \ 2^{\ell_1}\le 80\la\ {\rm and}\ 2^{\ell_2}>10\la	\\
			2^{-N_1\ell_2},\ \ \ \ \ \ \ 2^{\ell_1}>80\la\ {\rm and}\ 2^{\ell_2}>2^{\ell_1-2}\\
			2^{(n-2+\eta)(\ell_2-\ell_1)},\ \ 2^{\ell_1}>80\la\ {\rm and}\ 2^{\ell_2}\approx 10\la\\
			2^{(n-2+\eta)(\ell_2-\ell_1)}2^{-N_1\ell_2},\ \ 2^{\ell_1}>80\la\ {\rm and}\ 10\la< 2^{\ell_2}\le  2^{\ell_1-2}.
	\end{cases}\end{align}
	and then by Sobolev inequality
	\begin{align*}
		\|\sum_{|k|\in I_{\ell_1}}& b_{k, 2^s}e_k^0(\cdot)\|_{L^{p_0}(M)}\ls 2^{\ell_1/2}2^{\sigma(p_0)\ell_1}	\|\sum_{|k|\in I_{\ell_1}} b_{k, 2^s}e_k^0(\cdot)\|_{L^{2}(M)} \\
		&\ls  2^{(n/2+1-\frac14\eta)\ell_1}\cdot\sup_{k,\tau_\ell}|b_{k,\tau_\ell}|.
	\end{align*}
	And similarly we may obtain the estimates for $\partial_\tau b_{k,\tau}$.
	By applying Lemma~\ref{delta1} with $\delta=2^s$, we have for  fixed $0\le t\le 2^{-2\ell_1}$,
	\begin{align*}
		|\int \sum_{\tau_\ell\approx 2^s }&\sum_{|k|\in I_{\ell_1}} b_{k,\tau_\ell}t^{\mu_1+\mu_2+1}e^{-t\tau_\ell^2} \overline{e_k^0(z)}V(z)e_{\tau_\ell}(z) \overline{e_{\tau_\ell}(x)}dz |\\
		&\ls \|V\|_{L^{\frac{n}{2-\eta}}(M)}2^{ns/2}2^{s/2}2^{(\frac12-\frac14\eta)s}\cdot 2^{(n/2-1-\frac14\eta-2\mu_2)\ell_1}2^{(2\mu_2+2-\eta)\ell_2} \\
		&\quad\cdot\begin{cases}1,\ \ \ \ \ \ \ \ \ \ \ \ \ \  2^{\ell_1}\le 80\la\ {\rm and}\ 2^{\ell_2}\approx 10\la\\
			2^{-N_1\ell_2},\ \ \ \ \ \ \ 2^{\ell_1}\le 80\la\ {\rm and}\ 2^{\ell_2}>10\la	\\
			2^{-N_1\ell_2},\ \ \ \ \ \ \ 2^{\ell_1}>80\la\ {\rm and}\ 2^{\ell_2}>2^{\ell_1-2}\\
			2^{(n-2+\eta)(\ell_2-\ell_1)},\ \ 2^{\ell_1}>80\la\ {\rm and}\ 2^{\ell_2}\approx 10\la\\
			2^{(n-2+\eta)(\ell_2-\ell_1)}2^{-N_1\ell_2},\ \ 2^{\ell_1}>80\la\ {\rm and}\ 10\la< 2^{\ell_2}\le  2^{\ell_1-2}.
		\end{cases}
	\end{align*}
	Fix $N_1\ge 2N+2n$. Note that $2^s\ls 2^{\ell_1}$ or $2^s\ls 2^{\ell_2}$. Integrating over $0\le t\le 2^{-2\ell_1}$ and summing over $s,\ell_1,\ell_2$, we obtain \eqref{h.4.8}.

	\subsection*{Case 3.2.3c}  We need to show 
	\begin{equation}\label{h.5.1}
		|\sum_{j,k}\sum_{\tau_\ell> 2\la}\int h(|j|)|j|^{2\mu_1}|k|^{2\mu_2}c^2_{k,\tau_\ell}\,e_j^0(x)V_{jk}\overline{e_k^0(z)}V(z)e_{\tau_\ell}(z) \overline{e_{\tau_\ell}(x)}dz|\ls\la^{n-\frac32\eta}.
	\end{equation}
	In this case, we have $|j|\le0.9\tau_\ell$ and $|k|\le 0.9\tau_\ell$ on the support of $\rho_2(|k|/\tau_\ell)\rho_2(|j|/\tau_\ell)$. 
Let $\tau_\ell\in[2^s,2^{s+1}]$, $2^s>2\la$,
	and for $|j|\in I_{\ell_2}$, let
	\begin{equation}\label{h.5.2}
		b_{k,\tau_\ell}=\sum_{|j|\in I_{\ell_2}} h(|j|)|j|^{2\mu_1}|k|^{2\mu_2}e_j^0(x)V_{jk}\rho_2(|k|/\tau_\ell)\rho_2(|j|/\tau_\ell).
	\end{equation}
	Then for $|k|\in I_{\ell_1}$, we have
	\begin{align}\label{h.5.3}|b_{k,\tau_\ell}|\ls  2^{2\mu_1\ell_1}2^{(2\mu_2+2-\eta)\ell_2}\cdot\begin{cases}1,\ \ \ \ \ \ \ \ \ \ \ \ \ \  2^{\ell_1}\le 80\la\ {\rm and}\ 2^{\ell_2}\approx 10\la\\
			2^{-N_1\ell_2},\ \ \ \ \ \ \ 2^{\ell_1}\le 80\la\ {\rm and}\ 2^{\ell_2}>10\la	\\
			2^{-N_1\ell_2},\ \ \ \ \ \ \ 2^{\ell_1}>80\la\ {\rm and}\ 2^{\ell_2}>2^{\ell_1-2}\\
			2^{(n-2+\eta)(\ell_2-\ell_1)},\ \ 2^{\ell_1}>80\la\ {\rm and}\ 2^{\ell_2}\approx 10\la\\
			2^{(n-2+\eta)(\ell_2-\ell_1)}2^{-N_1\ell_2},\ \ 2^{\ell_1}>80\la\ {\rm and}\ 10\la< 2^{\ell_2}\le  2^{\ell_1-2}
	\end{cases}\end{align}
	and then by Sobolev inequality
	\begin{align*}
		\|\sum_{|k|\in I_{\ell_1}}& b_{k, 2^s}e_k^0(\cdot)\|_{L^{p_0}(M)}\ls 2^{\ell_1/2}2^{\sigma(p_0)\ell_1}	\|\sum_{|k|\in I_{\ell_1}} b_{k, 2^s}e_k^0(\cdot)\|_{L^{2}(M)} \\
		&\ls  2^{(n/2+1-\frac14\eta)\ell_1}\cdot\sup_{k,\tau_\ell}|b_{k,\tau_\ell}|.
	\end{align*}
	And similarly we may obtain the estimates for $\partial_\tau b_{k,\tau}$.
	Recall that $|w(x)|\ls e^{-x/2}$. So for $\tau_\ell\approx 2^s$ we have
	\[|w(\tau_\ell^22^{-2\ell_1})|\ls 2^{-N_2(s-\ell_1)},\ \forall N_2.\] By applying Lemma~\ref{delta1} with $\delta=2^s$, we have for  fixed $0\le t\le 2^{-2\ell_1}$,
	\begin{align*}
		|\int& \sum_{\tau_\ell\approx 2^s }\sum_{|k|\in I_{\ell_1}} b_{k,\tau_\ell}\frac{w(\tau_\ell^22^{-2\ell_1})}{2^{2\ell_1(\mu_1+\mu_2+2)}} \overline{e_k^0(z)}V(z)e_{\tau_\ell}(z) \overline{e_{\tau_\ell}(x)}dz |\\
		&\ls  \|V\|_{L^{\frac{n}{2-\eta}}(M)}2^{ns/2}2^{s/2}2^{(\frac12-\frac14\eta)s}2^{-N_2(s-\ell_1)}\cdot2^{(\frac n2-1-\frac14\eta-2\mu_2)\ell_1}2^{(2\mu_2+2-\eta)\ell_2}\\
		&\quad\cdot\begin{cases}1,\ \ \ \ \ \ \ \ \ \ \ \ \ \  2^{\ell_1}\le 80\la\ {\rm and}\ 2^{\ell_2}\approx 10\la\\
			2^{-N_1\ell_2},\ \ \ \ \ \ \ 2^{\ell_1}\le 80\la\ {\rm and}\ 2^{\ell_2}>10\la	\\
			2^{-N_1\ell_2},\ \ \ \ \ \ \ 2^{\ell_1}>80\la\ {\rm and}\ 2^{\ell_2}>2^{\ell_1-2}\\
			2^{(n-2+\eta)(\ell_2-\ell_1)},\ \ 2^{\ell_1}>80\la\ {\rm and}\ 2^{\ell_2}\approx10\la\\
			2^{(n-2+\eta)(\ell_2-\ell_1)}2^{-N_1\ell_2},\ \ 2^{\ell_1}>80\la\ {\rm and}\ 10\la< 2^{\ell_2}\le  2^{\ell_1-2}.
		\end{cases}
	\end{align*}
	Fix $N_2\ge n$ and $N_1\ge N_2+2N+2n$. Note that $\la\ls 2^{\ell_1},2^{\ell_2}\ls 2^s$. Summing over $s,\ell_1,\ell_2$, we obtain \eqref{h.5.1}.
	
	This completes the proof of Case 3.2.3.
		

		\bibliographystyle{plain}
		
	\end{document}